\documentclass{amsart}
\usepackage{graphicx} % Required for inserting images
\usepackage{amssymb,amsmath,amsfonts,amsthm}
\usepackage[disable]{todonotes}
\usepackage{bm}
\usepackage{hyperref}
\usepackage{pgfplots,pgfplotstable}
\usepackage{multirow}
\usepackage{tikz}
\usepackage{nicefrac}
\usetikzlibrary{matrix}
\usepgfplotslibrary{groupplots}
\usepackage[margin=1.5in]{geometry}

\usepackage{macros}

\begin{document}

% -----------------------------------------------------------
% Front matter
% -----------------------------------------------------------
\title[Iterative Methods for the Boltzmann Transport Equation]{Analysis of Iterative Methods for the Linear Boltzmann Transport Equation}
\author[P. Houston]{Paul Houston}
\address[P. Houston]{School of Mathematical Sciences, University of Nottingham,
University Park, Nottingham NG7 2RD, UK}
\email[Corresponding author]{paul.houston@nottingham.ac.uk}
\author[M. E. Hubbard]{Matthew E. Hubbard}
\address[M. E. Hubbard]{School of Mathematical Sciences, University of Nottingham,
University Park, Nottingham NG7 2RD, UK}
\email{matthew.hubbard@nottingham.ac.uk}
\author[T. J. Radley]{Thomas J. Radley}
\address[T. J. Radley]{School of Mathematical Sciences, University of Nottingham,
University Park, Nottingham NG7 2RD, UK}
\email{thomas.radley1@nottingham.ac.uk}

    \begin{abstract}
        In this article we consider the iterative solution of the linear system of equations arising from the discretisation of the poly-energetic linear Boltzmann transport equation using a discontinuous Galerkin finite element approximation in space, angle, and energy. In particular, we develop preconditioned Richardson iterations which may be understood as generalisations of source iteration in the mono-energetic setting, and derive computable \emph{a posteriori} bounds for the solver error incurred due to inexact linear algebra, measured in a relevant problem-specific norm. We prove that the convergence of the resulting schemes and \emph{a posteriori} solver error estimates are independent of the discretisation parameters. We also discuss how the poly-energetic Richardson iteration may be employed as a preconditioner for the generalised minimal residual (GMRES) method. Furthermore, we show that standard implementations of GMRES based on minimising the Euclidean norm of the residual vector can be utilized to yield computable \emph{a posteriori} solver error estimates at each iteration, through judicious selections of left- and right-preconditioners for the original linear system. The effectiveness of poly-energetic source iteration and preconditioned GMRES, as well as their respective \emph{a posteriori} solver error estimates, is demonstrated through numerical examples arising in the modelling of photon transport. 


	\end{abstract}
	
	\maketitle

\section{Introduction}

The linear Boltzmann transport equation models the scattering, absorption and streaming of radiative particles through a medium and is widely used in many applications including radiotherapy treatment planning \cite{bakkali2017validation}, the design of nuclear reactors \cite{lewis1984computational} and climate sciences \cite{evans1998spherical}. We consider the stationary form of the equation whose solution, referred to as the \emph{angular flux}, models the (mean) equilibrium distribution of the particle species under consideration; this is dependent on up to six independent variables.

The high dimensionality of the underlying problem has led to the development of a variety of different discretisation approaches for the spatial, angular and energetic components of the problem. For energy, the multigroup method has historically been the predominant discretisation scheme \cite{lewis1984computational}, whereby a finite number of non-overlapping energy groups are employed and the energetic dependence of the angular flux is approximated by a piecewise-constant (or known) function. Extensions to higher-order approximations in energy include the generalised multigroup method \cite{zhu2010discrete,gibson2014stability} and the discontinuous Galerkin method outlined in our recent article \cite{houston2023efficient}. Importantly, for problems in which particles always lose energy during a scattering event, either by local energy deposition or by production of secondary particles, multigroup-based energetic discretisations permit an ordered solver strategy for the angular flux in each energy group in order of decreasing energy. In terms of the angular discretisation, popular approaches include spherical harmonics methods \cite{fletcher1983solution,lewis1984computational,stacey2018nuclear}, discrete ordinates methods \cite{lewis1984computational,stacey2018nuclear}, and mesh-based approximations constructed using, for example, continuous finite elements \cite{gao2009fast}, discontinuous finite elements \cite{hall2017hp,houston2023efficient,kophazi2015space} or wavelets \cite{adigun2018haar}. Spatially, typically finite element, finite volume, finite difference, and characteristic-based methods are employed.

When the stationary linear Boltzmann transport equation is approximated using a discrete ordinates or mesh-based method in angle and a multigroup method in energy, the resulting linear system is typically very large and thus one must resort to iterative solution methods. Assuming that the physics of the problem is such that the approximate angular flux can be computed sequentially in order of descending energy groups, one needs only to solve a set of discretised mono-energetic problems. The literature for this is vast; see \cite{adams2002fast}, for example, and the references therein. Among the most common iterative methods for mono-energetic problems is \emph{source iteration}, which exploits the fact that the system matrix can be decoupled as the sum of a scattering term and an easily-invertible transport operator. The convergence rate of source iteration is characterised by the \emph{scattering ratio} $c\in[0,1)$ which may be determined from problem data. For highly-scattering media where $c\approx1$, the convergence of source iteration can be very slow.

To address this issue, a number of acceleration techniques have been proposed. Among the most common technique is diffusion-synthetic acceleration (DSA) \cite{alcouffe1977diffusion,warsa2002fully}, in which each source iteration is interspersed with a diffusion problem whose solution is added as a correction to the current angular flux iterate. While DSA is provably and rapidly convergent independently of $c$ for some discretised model problems, it is known that the choice of discretisation of the diffusion problem must be consistent with that of the transport problem \cite{warsa2002fully}. Furthermore, the DSA preconditioner may require special modification to deal with highly heterogeneous media \cite{southworth2021diffusion}.

In recent years, purely algebraic techniques based on Krylov subspace methods have been studied. The generalised minimal residual method (GMRES) \cite{saad1986gmres} and the biconjugate gradient stabilised method (BI-CGSTAB) \cite{van1992bi} have both been employed in radiation transport applications. In \cite{patton2002application}, applications of GMRES employing preconditioners based on incomplete LU factorisations and transport sweeps were shown to be competitive against DSA. In \cite{badri2019preconditioned}, GMRES and BI-CGSTAB are employed with Jacobi and block Jacobi preconditioning. In \cite{warsa2004krylov}, it was shown that GMRES, employing a preconditioner based on a partially-consistent DSA scheme (which may not otherwise be convergent as a stationary iterative method), can be more effective than fully-consistent DSA schemes in multidimensional geometries with highly scattering regions.

The original GMRES algorithm in \cite{saad1986gmres} seeks to successively minimise the $\ell_2$-norm of a residual vector associated with a (finite-dimensional) system of linear equations, which is returned alongside the approximate solution upon termination of the algorithm. Generalisations of the GMRES method to Hilbert space settings have been made in \cite{gasparo2008some,gunnel2014note}, and in the finite-dimensional case the authors of \cite{chen1999generalizations} present a modification of GMRES in which the vector norm used in the residual minimisation condition is induced by a positive-definite matrix. This last point is significant since, in finite element applications, the vector solution of the underlying system of equations is commonly the set of coefficients in the expansion of a finite element basis - in this setting, the $\ell_2$-norm of the residual vector may not be a useful measure of the error between the true finite element solution and its approximation obtained by inexact solution of the discrete equations \cite{wathen2007preconditioning}. The problem of obtaining meaningful measures of functional error incurred by premature termination of an iterative linear solver is especially important in the context of adaptive finite element methods; see \cite{arioli2013interplay} for a survey of the interplay between discretisation and algebraic error estimation.

In this article, we consider the numerical solution of the system of linear equations arising from a full space-angle-energy discretisation of the linear Boltzmann transport equation employing the discontinuous Galerkin finite element method (DGFEM) outlined in our previous article \cite{houston2023efficient}. Here, we generalise the source iteration method to the poly-energetic setting, including a poly-energetic analogue to the mono-energetic scattering ratio, and employ the resulting scheme as a preconditioner within a GMRES method. We also discuss a simple modification of mono-energetic source iteration that exhibits faster convergence rates at no extra cost. The convergence of the poly-energetic and generalised mono-energetic source iteration schemes is proven and computable \emph{a posteriori} solver error estimates are derived, which are guaranteed to bound an associated problem-dependent norm of the functional error in a mesh-independent manner. The practical implementation of these error bounds is also discussed; in particular, we show how residual-based \emph{a posteriori} solver error estimates may be incorporated into standard implementations of GMRES in such a way that the resulting method sequentially minimises the error estimate at each iteration. Numerical experiments illustrating the practical performance of each solver are presented: we observe that, using realistic photon scattering models, the convergence rates of poly-energetic source iteration and preconditioned GMRES, as well as their respective \emph{a posteriori} solver error estimates, are highly dependent on the energy range. In particular, GMRES is particularly effective in the low-energy photon scattering regime where the convergence of source iteration stagnates.

The outline of this paper is as follows. In Section \ref{section:model_problem}, we introduce the stationary linear Boltzmann transport equation. Then, in Section \ref{section:discretisation}, we summarise its high-order space-angle-energy DGFEM discretisation as presented in \cite{houston2023efficient}. In Section \ref{section:solver_analysis}, we introduce and discuss a number of iterative methods for the solution of the resulting linear system of equations and provide convergence proofs and \emph{a posteriori} solver error bounds. Section \ref{section:implementation} discusses the implementation of the proposed iterative methods and their corresponding error estimates. Section \ref{section:numerics} presents numerical experiments which demonstrate the performance of the proposed iterative methods for both mono-energetic and poly-energetic problems. Finally, in Section \ref{section:conclusions} we summarise the work presented in this article and highlight some areas of future work.

\section{Model Problem} \label{section:model_problem}

For $d=2,3$, let $\spacedom\subset\Re^d$ denote an open bounded Lipschitz spatial domain and $\angledom=\{\bm{\mu}\in\Re^d : |\bm{\mu}|_2 = 1\}$ denote the surface of the $d$-dimensional unit sphere, where $|\cdot|_2$ denotes the $\ell_2$-norm. Given $0\le E_{\min} < E_{\max} < \infty$, let $\energydom=\left[E_{\min},E_{\max}\right]$ and $\spaceangleenergydom=\spacedom\times\angledom\times\energydom$. The linear Boltzmann transport problem is given by: find $u:\spaceangleenergydom\rightarrow\Re$ satisfying
\begin{align}
    \bm{\mu}\cdot\nabla_\mathbf{x} u + (\alpha+\beta) u &= S[u] + f \textnormal{ in } \spaceangleenergydom, \label{eqn:lbte_exact} \\
    u &= g_D \textnormal{ on } \inflowbdry. \label{eqn:lbte_bc}
\end{align}

\noindent Here, $\nabla_\mathbf{x}$ denotes the spatial gradient operator, $\alpha$ and $\beta$ denote given data terms (described below), $f$ is a forcing term, $g_D$ is a specified (inlet) boundary condition, and $\inflowbdry = \{ (\mathbf{x},\bm{\mu},E) \in\overline{\spaceangleenergydom} : \mathbf{x}\in\partial\spacedom \textnormal{ and } \bm{\mu}\cdot\mathbf{n}(\mathbf{x}) < 0 \}$ denotes the inflow boundary of $\spaceangleenergydom$, where $\mathbf{n}$ is the outward unit normal vector to $\spacedom$ on $\partial\spacedom$. The operator \emph{scattering operator} $S[u]$ is defined by
\begin{equation} \label{eqn:scatter_operator}
    S[u](\mathbf{x},\bm{\mu},E) = \int_\energydom \int_\angledom \theta(\mathbf{x},\bm{\mu}'\cdot\bm{\mu}, E'\rightarrow E) u(\mathbf{x},\bm{\mu}',E') \d \bm{\mu}' \d E',
\end{equation}

\noindent where $\theta(\mathbf{x},\bm{\mu}'\cdot\bm{\mu},E'\rightarrow E)$ is a non-negative scattering kernel satisfying $\theta(\mathbf{x},\bm{\mu}'\cdot\bm{\mu},E'\rightarrow E)=0$ for $E'<E$.
The \emph{macroscopic scattering cross-section} $\beta$ is given by
\begin{equation} \label{eqn:macro_scatter_cs}
    \beta(\mathbf{x},E) = \int_\energydom \int_\angledom \theta(\mathbf{x},\bm{\mu}\cdot\bm{\mu}',E\rightarrow E') \d \bm{\mu}' \d E'.
\end{equation}

\noindent Here, $\alpha$, referred to as the \emph{macroscopic absorption cross-section}, is assumed to be non-negative and independent of the angular variable; i.e., $\alpha(\mathbf{x},\bm{\mu},E) \equiv \alpha(\mathbf{x},E)$. We remark that such assumptions are valid for isotropic media. Finally, we assume that there exists a constant $\alpha_0>0$ such that
\begin{equation} \label{eqn:alphabar_defn}
    \overline{\alpha}(\mathbf{x},E) := \alpha(\mathbf{x},E) + \nicefrac{1}{2}\left( \beta(\mathbf{x},E)-\gamma(\mathbf{x},E) \right) \ge \alpha_0,
\end{equation}

\noindent where
\begin{equation} \label{eqn:macro_inverse_scatter_cs}
    \gamma(\mathbf{x},E) = \int_\energydom \int_\angledom \theta(\mathbf{x},\bm{\mu}'\cdot\bm{\mu},E'\rightarrow E) \d\bm{\mu}' \d E'.
\end{equation}

\noindent We remark that, in applications to photon transport, one typically has $\gamma(\mathbf{x},E)\le\beta(\mathbf{x},E)$ \cite{radley2023discontinuous}, and so \reff{eqn:alphabar_defn} holds on a finite energy domain $0<E_{min}<E_{max}<\infty$ even when $\alpha=0$.

\section{Discretisation} \label{section:discretisation}

Following \cite{houston2023efficient}, we discretise the Boltzmann transport problem \eqref{eqn:lbte_exact}-\eqref{eqn:lbte_bc} using a DGFEM approach; we recall the key constructions below.

\subsection{Spatial discretisation}

Let $\spacemesh$ denote a subdivision of the spatial domain $\spacedom$ into open disjoint elements $\spacekappa$ such that $\overline{\spacedom}=\bigcup_{\spacekappa\in\spacemesh}\overline{\spacekappa}$. Given $p\ge0$, the spatial finite element space is defined by
\begin{equation}
    \spacefespace = \left\{ v\in L_2(\spacedom) : v|_{\spacekappa}\in\mathbb{P}_{p}(\spacekappa) \textnormal{ for all } \spacekappa\in\spacemesh \right\},
\end{equation}

\noindent where $\mathbb{P}_{p}(\spacekappa)$ denotes the space of polynomials of total degree $p$ on $\spacekappa$.

For a given element $\spacekappa$, the union of its $(d-1)$-dimensional open faces is denoted by $\partial\spacekappa$. For a given direction $\bm{\mu}\in\angledom$, the inflow and outflow boundaries of $\spacekappa$ are defined, respectively, by
\begin{align*}
    \partial_-\spacekappa &= \left\{ \mathbf{x}\in\partial\spacekappa : \bm{\mu}\cdot\mathbf{n}_{\spacekappa}(\mathbf{x}) < 0 \right\}, 
    \quad
    \partial_+\spacekappa &= \left\{ \mathbf{x}\in\partial\spacekappa : \bm{\mu}\cdot\mathbf{n}_{\spacekappa}(\mathbf{x}) \ge 0 \right\},
\end{align*}

\noindent where $\mathbf{n}_{\spacekappa}(\mathbf{x})$ denotes the outward unit vector to $\spacekappa$ at $\mathbf{x}\in\partial\spacekappa$.

Given a face $F\subset\partial_-\spacekappa\setminus\partial\spacedom$, $\spacekappa\in\spacemesh$, we define the upwind jump of a (sufficiently smooth) function $v$ across F by
$
    \lfloor v \rfloor = v^+ - v^-,
$
where $v^\pm$ denotes the interior, respectively, exterior trace of $v$ on $F$, relative to $\spacekappa\in\spacemesh$.

\subsection{Angular discretisation}

We shall employ the widely-used cube-sphere discretisation of the angular domain. To this end, let $\angledomdiscrete$ denote the closed polygonal/polyhedral surface of the $d$-dimensional cube in $\Re^d$, $d=2, 3$, respectively. We write $\tilde{\anglemesh} = \{ \tilde{\anglekappa} \}$ to denote a surface mesh consisting of tensor product elements $\tilde{\anglekappa}$, i.e., so that $\angledomdiscrete = \bigcup_{\tilde{\anglekappa}\in\tilde{\anglemesh}} \tilde{\anglekappa}$.
Letting $\phi_\angledom:\angledomdiscrete\rightarrow\angledom$ denote the smooth invertible mapping defined for any $\tilde{\bm{\mu}}\in\angledomdiscrete$ by $\phi_\angledom(\tilde{\bm{\mu}}) = |\tilde{\bm{\mu}}|_2^{-1} \tilde{\bm{\mu}}$, we define a mesh $\anglemesh$ consisting of curved surface elements on $\angledom$ by
$
	\anglemesh = \{ \anglekappa = \phi_\angledom(\tilde{\anglekappa}) : \tilde{\anglekappa}\in\tilde{\anglemesh} \}
$. Then, given $q\ge0$, the angular finite element space is defined by
\begin{equation}
	\anglefespace = \left\{ v\in L_2(\angledom) : v|_{\anglekappa} = \tilde{v} \circ \phi_\angledom^{-1}(\anglekappa) \textnormal{ with } \tilde{v}\in\mathbb{Q}_{q}(\tilde{\anglekappa}) \textnormal{ for all } \anglekappa\in\anglemesh \right\},
\end{equation}

\noindent where $\mathbb{Q}_{q}(\tilde{\anglekappa})$ denotes the space of polynomials of degree at most $q$ in each of the $(d-1)$ coordinate directions parametrising $\tilde{\anglekappa}$.

\subsection{Energetic discretisation}

For $N_{\energydom}\ge1$, let $E_{\max}:= E_0 > E_1 > \dots > E_{N_\energydom} =: E_{\min}$ denote a partition of the energy domain $\energydom$ into $N_{\energydom}$ \emph{energy groups}. The interval $\kappa_g=(E_g,E_{g-1})$, $1\le g\le N_{\energydom}$, is referred to as the $g^{th}$ energy group, and we define by $\energymesh=\left\{ \kappa_g \right\}_{g=1}^{N_{\energydom}}$ to be the energetic mesh. Given $r\ge0$, the energetic finite element space is defined by
\begin{equation}
    \energyfespace = \left\{ v\in L_2(\energydom) : v|_{\kappa_g} \in \mathbb{P}_{r}(\kappa_g) \textnormal{ for all } \kappa_g\in\energymesh \right\}.
\end{equation}

\subsection{Discontinuous Galerkin method}

Using the definitions of $\spacemesh$, $\anglemesh$, and $\energymesh$ outlined in the previous sections, we define the full space-angle-energy mesh $\mesh$ by
\begin{equation}
    \mesh = \spacemesh \times \anglemesh \times \energymesh = \left\{ \kappa = \spacekappa\times\anglekappa\times\kappa_g : \spacekappa\in\spacemesh, \anglekappa\in\anglemesh, \kappa_g\in\energymesh \right\}.
\end{equation}

\noindent We combine the separate finite element spaces to define the space-angle-energy finite element space 
$
    \fespace = \spacefespace \otimes \anglefespace \otimes \energyfespace.
$
We also define the broken graph space $\brokengraphspace$ by
\begin{equation}
    \brokengraphspace = \left\{ v\in L_2(\spaceangleenergydom) : \bm{\mu}\cdot\nabla_\mathbf{x} v|_{\kappa} \in L_2(\kappa) \textnormal{ for all } \kappa\in\mesh \right\}.
\end{equation}

The DGFEM scheme reads as follows: find $u_h\in\fespace$ such that
\begin{equation} \label{eqn:dgfem_lbte_full}
    a(u_h,v_h) = s(u_h,v_h) + \ell(v_h)
\end{equation}

\noindent for all $v_h\in\fespace$, where $a:\brokengraphspace\times\brokengraphspace\rightarrow\Re$, $s:\brokengraphspace\times\brokengraphspace\rightarrow\Re$ and $\ell:\brokengraphspace\rightarrow\Re$ are defined, respectively, for all $w,v\in\brokengraphspace$, by
\begin{align*}
    a(w,v) &= \int_\energydom \int_\angledom \Bigg[ \sum_{\spacekappa\in\spacemesh} \int_{\spacekappa} (-\bm{\mu} w\cdot\nabla_\mathbf{x} v + (\alpha+\beta)wv) \d\mathbf{x} \\
    &\quad\quad\quad\quad + \int_{\partial_+\spacekappa} |\bm{\mu}\cdot\mathbf{n}| w^+v^+ \d s 
     - \int_{\partial_-\spacekappa\setminus\partial\spacedom} |\bm{\mu}\cdot\mathbf{n}| w^-v^+ \d s \Bigg] \d\bm{\mu} \d E, \\
    s(w,v) &= \int_\energydom \int_\angledom \int_\spacedom S[w] v \d\mathbf{x} \d\bm{\mu} \d E, \\
    \ell(v) &= \int_\energydom \int_\angledom \left[ \sum_{\spacekappa\in\spacemesh} \int_{\spacekappa} fv \ \d\mathbf{x} + \int_{\partial_-\spacekappa\cap\partial\spacedom} |\bm{\mu}\cdot\mathbf{n}| g_D v^+ \d s \right] \d\bm{\mu} \d E.
\end{align*}

\section{Analysis of Iterative Solvers} \label{section:solver_analysis}

In this section we study the convergence of iterative solvers for the DGFEM scheme \reff{eqn:dgfem_lbte_full}. Following \cite{houston2023efficient}, we introduce the DGFEM-energy norm $|||\cdot|||_{DG}:\brokengraphspace\rightarrow\Re$ defined for all $v\in\brokengraphspace$ by
\begin{align*}
    ||| v |||_{DG}^2 &= || \sqrt{\overline{\alpha}} v ||_{L_2(\spaceangleenergydom)}^2 \nonumber \\
    &\quad\quad + \frac12 \int_\energydom \int_\angledom \sum_{\spacekappa\in\spacemesh} \left( || v^+-v^- ||_{\partial_-\spacekappa\setminus\partial\spacedom}^2 + || v^+ ||_{\partial\spacekappa\cap\partial\spacedom}^2 \right) \d\bm{\mu} \d E,
\end{align*}

\noindent where the (semi)norm $|| v ||_\omega^2 = \int_\omega |\bm{\mu}\cdot\mathbf{n}_{\spacekappa}| v^2 \d s$ for $\omega\subset\partial\spacekappa$. We recall the following theorem from \cite{houston2023efficient}.
\begin{theorem}\label{thm:DGFEM_bilin_form_coercivity}
    The DGFEM bilinear form $a(\cdot,\cdot)-s(\cdot,\cdot)$ is coercive with respect to the DGFEM-energy norm in the sense that
    \begin{equation}
        ||| v_h |||_{DG}^2 \le a(v_h,v_h) - s(v_h,v_h)
    \end{equation}

    \noindent for all $v_h\in\fespace$.
\end{theorem}

\begin{remark}
    The above coercivity result trivially holds for all $v\in\brokengraphspace$.
\end{remark}

For the purposes of simplifying the presentation of the forthcoming linear solvers, we endow $\fespace$ with an inner product $\langle\cdot,\cdot\rangle:\fespace\times\fespace\rightarrow\mathbb{R}$ and define the linear operators $A,S:\fespace\rightarrow\fespace$, for all $w_h\in\fespace$, respectively, by
\begin{align}
    \langle Aw_h,v_h \rangle &= a(w_h,v_h) \textnormal{ for all } v_h\in\fespace, \label{eqn:dgfem_transport_operator} \\
    \langle Sw_h,v_h \rangle &= s(w_h,v_h) \textnormal{ for all } v_h\in\fespace. \label{eqn:dgfem_scattering operator}
\end{align}

\noindent Furthermore, by defining $F\in\fespace$ by
\begin{equation} \label{eqn:dgfem_source_term_function}
    \langle F,v_h\rangle = \ell(v_h) \textnormal{ for all } v_h\in\fespace,
\end{equation}

\noindent the DGFEM \reff{eqn:dgfem_lbte_full} can be equivalently defined: find $u_h \in\fespace$ such that
\begin{equation} \label{eqn:dgfem_lbte_operator}
    (A-S)u_h = F.
\end{equation}

For the analysis of eigenvalues of linear operators on $\fespace$, we introduce a finite element space $\fespace_\mathbb{C}$ of complex-valued discontinuous piecewise-polynomial functions, which may be understood as the linear span of functions in $\fespace$ over the field $\mathbb{C}$. In such cases, the definitions of the bilinear forms $a$ and $s$, as well as the linear operators $A$ and $S$, are extended in a trivial way, and now $\langle\cdot,\cdot\rangle:\fespace_\mathbb{C}\times\fespace_\mathbb{C}\rightarrow\mathbb{C}$ denotes a sesquilinear form which is linear in the first argument and antilinear in the second argument.

\subsection{Source iteration}

Source iteration seeks to solve the following sequence of DGFEM problems: given $u_h^{(0)}\in\fespace$, find $\{u_h^{(n)}\}_{n=0}^\infty\subset\fespace$ such that
\begin{equation} \label{eqn:dgfem_si_operator}
    Au_h^{(n+1)} = Su_h^{(n)} + F,
\end{equation}

\noindent for $n\ge0$. We briefly remark that \reff{eqn:dgfem_si_operator} is well-posed and that $a(\cdot,\cdot)$ is coercive with respect to the norm $||| \cdot |||_{a}:\brokengraphspace\rightarrow\Re$ defined for all $v\in\brokengraphspace$ by
\begin{align}
    ||| v |||_{a}^2 &= || \sqrt{\alpha+\beta} v ||_{L_2(\spaceangleenergydom)}^2 \\
    &\quad\quad + \frac12 \int_\energydom \int_\angledom \left( \sum_{\spacekappa\in\spacemesh} || v^+-v^- ||_{\partial_-\spacekappa\setminus\partial\spacedom}^2 + || v^+ ||_{\partial\spacekappa\cap\partial\spacedom}^2 \right) \d\bm{\mu} \d E
\end{align}

\noindent in the sense that
\begin{equation} \label{eqn:dgfem_a_norm}
    ||| v |||_a^2 = a(v,v) \qquad \forall v\in\brokengraphspace.
\end{equation}

\begin{theorem} \label{thm:dgfem_si_contraction_apost}
    Let $u_h\in\fespace$ denote the exact solution to the DGFEM problem \reff{eqn:dgfem_lbte_full} and $\{u_h^{(n)}\}_{n=0}^\infty\subset\fespace$ denote the sequence of approximate solutions to $u_h$ defined by \reff{eqn:dgfem_si_operator}. Then we have that
    \begin{equation}
        ||| u_h - u_h^{(n+1)} |||_{a} \le \sqrt{q_\beta q_\gamma} \ ||| u_h-u_h^{(n)} |||_{a}
    \end{equation}

    \noindent for all $n\ge0$, where
    \begin{align*}
        q_\beta = \sup_{\mathbf{x}\in\spacedom,E\in\energydom} \ \frac{\beta(\mathbf{x},E)}{\alpha(\mathbf{x},E)+\beta(\mathbf{x},E)}, \qquad
        q_\gamma = \sup_{\mathbf{x}\in\spacedom,E\in\energydom} \ \frac{\gamma(\mathbf{x},E)}{\alpha(\mathbf{x},E)+\beta(\mathbf{x},E)}.
    \end{align*}

    \noindent Thus, we have the following \emph{a posteriori} solver error bound:
    \begin{equation}
        ||| u_h-u_h^{(n+1)} |||_{DG} \le \sqrt{r_\gamma} \ || \sqrt{\beta} (u_h^{(n)}-u_h^{(n+1)}) ||_{L_2(\spaceangleenergydom)}
    \end{equation}

    \noindent for all $n\ge0$, where
    \begin{equation*}
        r_\gamma = \sup_{\mathbf{x}\in\spacedom,E\in\energydom} \ \frac{\gamma(\mathbf{x},E)}{\overline{\alpha}(\mathbf{x},E)}.
    \end{equation*}
\end{theorem}

\begin{proof}
    Starting from the coercivity result \reff{eqn:dgfem_a_norm}, we have, using \reff{eqn:dgfem_si_operator}, that
    \begin{align*}
        ||| u_h-u_h^{(n+1)} |||_{a}^2 &= a(u_h-u_h^{(n+1)},u_h-u_h^{(n+1)}) \\
        &= s(u_h-u_h^{(n)},u_h-u_h^{(n+1)}).
    \end{align*}

    \noindent Invoking the Cauchy-Schwarz inequality, we obtain
    \begin{align*}
        &s(u_h-u_h^{(n)},u_h-u_h^{(n+1)}) \\
        &\quad\quad = \int_\energydom \int_\angledom \int_\spacedom S[u_h-u_h^{(n)}] (u_h-u_h^{(n+1)}) \ \d\mathbf{x}\d\bm{\mu}\d E \\
        &\quad\quad \le || \sqrt{\beta} (u_h-u_h^{(n)}) ||_{L_2(\spaceangleenergydom)} \ || \sqrt{\gamma} (u_h-u_h^{(n+1)}) ||_{L_2(\spaceangleenergydom)} \\
        &\quad\quad \le \sqrt{q_\beta q_\gamma} \ ||| u_h-u_h^{(n)} |||_{a} \ ||| u_h-u_h^{(n+1)} |||_{a}.
    \end{align*}

    \noindent The first part of the theorem now follows directly. To derive the \emph{a posteriori} error bound, we note that, by using Theorem \ref{thm:DGFEM_bilin_form_coercivity}, we have
    \begin{align*}
        ||| u_h-u_h^{(n+1)} |||_{DG}^2 &\le a(u_h-u_h^{(n+1)},u_h-u_h^{(n+1)}) - s(u_h-u_h^{(n+1)},u_h-u_h^{(n+1)}) \\
        &= s(u_h^{(n+1)}-u_h^{(n)},u_h-u_h^{(n+1)}) \\
        &\le || \sqrt{\beta} (u_h^{(n)}-u_h^{(n+1)}) ||_{L_2(\spaceangleenergydom)} \ || \sqrt{\gamma} (u_h-u_h^{(n+1)}) ||_{L_2(\spaceangleenergydom)} \\
        &\le \sqrt{r_\gamma} \ || \sqrt{\beta} (u_h^{(n)}-u_h^{(n+1)}) ||_{L_2(\spaceangleenergydom)} \ ||| u_h-u_h^{(n+1)} |||_{DG},
    \end{align*}

    \noindent as required.
\end{proof}

\begin{corollary}
    If $\sqrt{q_\beta q_\gamma} < 1$, the sequence of approximate solutions $\{u_h^{(n)}\}_{n=0}^\infty \subset\fespace$ generated by \reff{eqn:dgfem_si_operator} satisfies $u_h^{(n)}\rightarrow u_h$ as $n\rightarrow\infty$.
\end{corollary}

Motivated by the operator form \reff{eqn:dgfem_si_operator}, we note that source iteration can be understood as a preconditioned Richardson iteration of the form
\begin{equation} \label{eqn:dgfem_general_richardson}
    u_h^{(n+1)} = u_h^{(n)} + P^{-1}\left( F - (A-S)u_h^{(n)} \right),
\end{equation}

\noindent where the preconditioner $P^{-1}=A^{-1}$. The contraction factor
\begin{equation} \label{eqn:scatater_ratio_poly_defn}
    c = \sqrt{q_\beta q_\gamma} = \sqrt{ \sup_{\mathbf{x}\in\spacedom,E\in\energydom} \frac{\beta(\mathbf{x},E)}{\alpha(\mathbf{x},E)+\beta(\mathbf{x},E)} \cdot \sup_{\mathbf{x}\in\spacedom,E\in\energydom} \frac{\gamma(\mathbf{x},E)}{\alpha(\mathbf{x},E)+\beta(\mathbf{x},E)} }
\end{equation}

\noindent may be understood as a poly-energetic analogue of the mono-energetic scattering ratio $c_{mono}$ \cite{adams2002fast}:
\begin{equation} \label{eqn:scatter_ratio_mono_defn}
    c_{mono} = \sup_{\mathbf{x}\in\spacedom} \frac{\beta(\mathbf{x})}{\alpha(\mathbf{x})+\beta(\mathbf{x})}.
\end{equation}

We also note that the spectral radius $\rho(G)$ of the source iteration operator $G=A^{-1}S$ satisfies $\rho(G)\le c$, as suggested by Theorem \ref{thm:dgfem_si_contraction_apost}. To see this, let $(\lambda,w_h)\in(\mathbb{C},\fespace_\mathbb{C})$ be an eigenvalue-eigenvector pair of $G$, so that we have
\begin{equation*}
    A^{-1}Sw_h = \lambda w_h.
\end{equation*}

\noindent It can be shown that a characterisation of $\lambda$ is given by
$
    \lambda = s(w_h,w_h^*)/a(w_h,w_h^*).
$
The boundedness of $|\lambda|$ follows by straightforward considerations.

In the mono-energetic case that, it can be shown that, provided the bilinear form $s(\cdot,\cdot)$ is symmetric positive-semidefinite, the spectrum of $A^{-1}S$ lies in a disc $D(a,r)$ of radius $r=\frac{c}{2}=\frac{c_{mono}}{2}$ centred at $r=\frac{c}{2}=\frac{c_{mono}}{2}$ \cite{radley2023discontinuous}. This is because, by the characterisation above, any $\lambda\in\sigma(A^{-1}S)$ can be written in the following form for some $w_h=x_h+iy_h$ with $x_h,y_h\in\fespace$:
\begin{align*}
    \lambda = \frac{s(w_h,w_h^*)}{a(w_h,w_h^*)} 
    = \frac{s(x_h,x_h) + s(y_h,y_h)}{a(x_h,x_h) + a(y_h,y_h) + i\left[ a(y_h,x_h)-a(x_h,y_h) \right]},
\end{align*}

\noindent where the symmetry of $s(\cdot,\cdot)$ is used to simplify the numerator and all bilinear forms are now real-valued. The bounds on $\sigma(A^{-1}S)$ in the complex plane are then shown after some simple considerations \cite{radley2023discontinuous}.

\subsubsection{Generalised source iteration for mono-energetic problems}

In the case where the energetic dependence of the test and trial functions is suppressed, we consider the following modification of \reff{eqn:dgfem_si_operator}: given $u_h^{(0)}\in\fespace$ and $0<\omega<1$, find $\{u_h^{(n)}\}_{n=0}^\infty\subset\fespace$ such that
\begin{equation} \label{eqn:dgfem_si_generalised_operator}
    (A-\omega M)u_h^{(n+1)} = (S-\omega M)u_h^{(n)} + F,
\end{equation}

\noindent for $n\ge0$, where the operator $M:\fespace\rightarrow\fespace$ is defined for all $w_h\in\fespace$ by
\begin{equation} \label{eqn:dgfem_mass_overrelax_operator}
    \langle Mw_h,v_h\rangle = m(w_h,v_h) \textnormal{ for all }v_h\in\fespace,
\end{equation}
and 
\begin{equation}
    m(w,v) = \int_\angledom \int_\spacedom \beta(\mathbf{x}) w(\mathbf{x},\bm{\mu}) v(\mathbf{x},\bm{\mu}) \d\mathbf{x}\d\bm{\mu}.
\end{equation}

In view of developing a convergence theory analogous to that of Theorem~\ref{thm:dgfem_si_contraction_apost}, we introduce the following lemma.
\begin{lemma} \label{lemma:scatter_mass_contraction}
    Assuming that $\theta(\mathbf{x},\bm{\mu}\cdot\bm{\mu}')$ admits the form
    \begin{equation*}
        \theta(\mathbf{x},\bm{\mu}\cdot\bm{\mu}') = \sum_{\ell=0}^\infty \theta_\ell(\mathbf{x}) P_{\ell}(\bm{\mu}\cdot\bm{\mu}'),
    \end{equation*}

    \noindent where $P_\ell$ denotes the Legendre polynomial of degree $\ell$ \cite{stacey2018nuclear}, we have that
    \begin{equation}
        |s(w,v)-\omega m(w,v)| \le r(\theta,\omega) m(w,w)^\frac12 m(v,v)^\frac12
    \end{equation}

    \noindent for all $w,v\in L_2(\spacedom\times\angledom)$, where 
    \begin{equation*}
        r(\theta,\omega) = \sup \left\{ \sup_{\mathbf{x}\in\spacedom} \left| \frac{\theta_\ell(\mathbf{x})}{\beta(\mathbf{x})} - \omega \right| : 0\le\ell<\infty \right\}.
    \end{equation*}
\end{lemma}

\begin{proof}
    It is convenient to prove the more general inequality
    \begin{equation*}
        |s(w,v^*)-\omega m(w,v^*)| \le r(\theta,\omega) m(w,w^*)^\frac12 m(v,v^*)^\frac12
    \end{equation*}

    \noindent in the case where $w$ and $v$ are possibly complex-valued and $d=3$; the lemma is then proven by selecting $w$ and $v$ to be real-valued. Here, we use the notation $v^*$ to denote the complex conjugate of a complex-valued function $v\in L_2(\spacedom\times\angledom)$. The proof generalises trivially to the case $d=2$.

    Let $w$ and $v$ admit the following spherical harmonics decompositions:
    \begin{align*}
        w(\mathbf{x},\bm{\mu}) &= \sum_{\ell=0}^\infty \frac{2\ell+1}{4\pi} \sum_{m=-\ell}^\ell w_{\ell,m}(\mathbf{x}) Y_{\ell,m}(\bm{\mu}), \\
        v(\mathbf{x},\bm{\mu}) &= \sum_{\ell=0}^\infty \frac{2\ell+1}{4\pi} \sum_{m=-\ell}^\ell v_{\ell,m}(\mathbf{x}) Y_{\ell,m}(\bm{\mu}),
    \end{align*}

    \noindent where $\{Y_{\ell,m} : 0\le\ell<\infty, -\ell\le m\le\ell\}$ denotes the set of spherical harmonics functions satisfying the following orthogonality property on $L_2(\angledom)$:
    \begin{equation*}
        \int_\angledom Y_{\ell,m}(\bm{\mu}) Y_{\ell',m'}^*(\bm{\mu}) \d\bm{\mu} = \frac{4\pi}{2\ell+1} \delta_{\ell,\ell'} \delta_{m,m'}
    \end{equation*}

    \noindent and $\delta_{ij}$ denotes the Kronecker delta.

    Substituting these expressions for $w$ and $v$ into $m(w,v^*)$ yields
    \begin{align*}
        m(w,v^*) &= \int_\angledom \int_\spacedom \beta(\mathbf{x}) w(\mathbf{x},\bm{\mu}) v^*(\mathbf{x},\bm{\mu}) \d\mathbf{x}\d\bm{\mu} \\
        %
        %&= \int_\spacedom \beta(\mathbf{x}) \int_\angledom  \left( \sum_{\ell=0}^\infty \frac{2\ell+1}{4\pi} \sum_{m=-\ell}^\ell w_{\ell,m}(\mathbf{x}) Y_{\ell,m}(\bm{\mu}') \right) \left( \sum_{\ell'=0}^\infty \frac{2\ell'+1}{4\pi} \sum_{m'=-\ell'}^{\ell'} v_{\ell',m'}^*(\mathbf{x}) Y_{\ell',m'}^*(\bm{\mu}) \right) \d\bm{\mu}\d\mathbf{x} \\
        %
        %&= \sum_{\ell=0}^\infty \sum_{\ell'=0}^\infty \frac{2\ell+1}{4\pi} \frac{2\ell'+1}{4\pi} \sum_{m=-\ell}^{\ell} \sum_{m'=-\ell'}^{\ell'} \int_\spacedom \beta(\mathbf{x}) w_{\ell,m}(\mathbf{x}) v_{\ell',m'}^*(\mathbf{x}) \int_\angledom Y_{\ell,m}(\bm{\mu}) Y_{\ell',m'}^*(\bm{\mu}) \d\bm{\mu}\d\mathbf{x} \\
        %
        %&= \sum_{\ell=0}^\infty \sum_{\ell'=0}^\infty \frac{2\ell+1}{4\pi} \frac{2\ell'+1}{4\pi} \sum_{m=-\ell}^{\ell} \sum_{m'=-\ell'}^{\ell'} \int_\spacedom \beta(\mathbf{x}) w_{\ell,m}(\mathbf{x}) v_{\ell',m'}^*(\mathbf{x}) \d\mathbf{x} \cdot \frac{4\pi}{2\ell'+1} \delta_{\ell,\ell'} \delta_{m,m'} \\
        %
        &= \sum_{\ell=0}^\infty \frac{2\ell+1}{4\pi} \sum_{m=-\ell}^\ell \int_\spacedom \beta(\mathbf{x}) w_{\ell,m}(\mathbf{x}) v_{\ell,m}^*(\mathbf{x}) \d\mathbf{x}.
    \end{align*}

    \noindent Substituting the expressions for $w$, $v$ and $\theta$ (defined above) into $s(w,v^*)$, and noting that the Legendre polynomials $P_\ell(\bm{\mu}\cdot\bm{\mu}')$ satisfy
    \begin{equation*}
        P_\ell(\bm{\mu}\cdot\bm{\mu}') = \frac{2\ell+1}{4\pi} \sum_{m=-\ell}^\ell Y_{\ell,m}^*(\bm{\mu}') Y_{\ell,m}(\bm{\mu}),
    \end{equation*}

    \noindent we have
    \begin{align*}
        s(w,v^*) &= \int_\angledom \int_\spacedom \int_\angledom \theta(\mathbf{x},\bm{\mu}\cdot\bm{\mu}') w(\mathbf{x},\bm{\mu}') v^*(\mathbf{x},\bm{\mu}) \d\bm{\mu}'\d\mathbf{x}\d\bm{\mu} \\
        %
        %&= \int_\spacedom \int_\angledom \int_\angledom \sum_{r=0}^\infty \frac{2r+1}{4\pi} \theta_r(\mathbf{x}) \sum_{s=-r}^r Y_{r,s}^*(\bm{\mu}') Y_{r,s}(\bm{\mu}) \\
        %&\quad\quad \cdot \sum_{\ell=0}^\infty \frac{2\ell+1}{4\pi} \sum_{m=-\ell}^\ell w_{\ell,m}(\mathbf{x}) Y_{\ell,m}(\bm{\mu}') \cdot \sum_{\ell'=0}^\infty \frac{2\ell'+1}{4\pi} \sum_{m'=-\ell'}^{\ell'} v_{\ell',m'}^*(\mathbf{x}) Y_{\ell',m'}^*(\bm{\mu}) \d\bm{\mu}'\d\bm{\mu}\d\mathbf{x} \\
        %
        %&= \sum_{r=0}^\infty \sum_{\ell=0}^\infty \sum_{\ell'=0}^\infty \frac{2r+1}{4\pi} \frac{2\ell+1}{4\pi} \frac{2\ell'+1}{4\pi} \sum_{s=-r}^r \sum_{m=-\ell}^\ell \sum_{m'=-\ell'}^{\ell'} \cdot \\
        %&\quad\quad \int_\spacedom \theta_r(\mathbf{x}) w_{\ell,m}(\mathbf{x}) v_{\ell',m'}^*(\mathbf{x}) \left( \int_\angledom Y_{r,s}^*(\bm{\mu}') Y_{l,m}(\bm{\mu}') \d\bm{\mu}' \right) \left( \int_\angledom Y_{r,s}(\bm{\mu}) Y_{l',m'}^*(\bm{\mu}) \d\bm{\mu} \right) \d\mathbf{x} \\
        %
        %&= \sum_{r=0}^\infty \sum_{\ell=0}^\infty \sum_{\ell'=0}^\infty \frac{2r+1}{4\pi} \frac{2\ell+1}{4\pi} \frac{2\ell'+1}{4\pi} \sum_{s=-r}^r \sum_{m=-\ell}^\ell \sum_{m'=-\ell'}^{\ell'} \cdot \\
        %&\quad\quad \int_\spacedom \theta_r(\mathbf{x}) w_{\ell,m}(\mathbf{x}) v_{\ell',m'}^*(\mathbf{x}) \d\mathbf{x} \cdot \frac{4\pi}{2\ell'+1} \delta_{r,\ell'} \delta_{s,m'} \cdot \frac{4\pi}{2\ell+1} \delta_{r,\ell} \delta_{s,m} \\
        %
        &= \sum_{r=0}^\infty \frac{2r+1}{4\pi} \sum_{s=-r}^r \int_\spacedom \theta_r(\mathbf{x}) w_{r,s}(\mathbf{x}) v_{r,s}^*(\mathbf{x}) \d\mathbf{x}.
    \end{align*}

    \noindent Hence, we deduce that
    \begin{align*}
        &\left|  s(w,v^*) - \omega m(w,v^*) \right| \\
        &= \left| \sum_{\ell=0}^\infty \frac{2\ell+1}{4\pi} \sum_{m=-\ell}^\ell \int_\spacedom \left( \theta_\ell(\mathbf{x}) - \omega\beta(\mathbf{x}) \right) w_{\ell,m}(\mathbf{x}) v_{\ell,m}^*(\mathbf{x}) \d\mathbf{x} \right| \\
        &\le \sum_{\ell=0}^\infty \frac{2\ell+1}{4\pi} \sum_{m=-\ell}^\ell \int_\spacedom \left| \theta_\ell(\mathbf{x}) - \omega\beta(\mathbf{x}) \right| \left|w_{\ell,m}(\mathbf{x})\right| \left|v_{\ell,m}(\mathbf{x})\right| \d\mathbf{x} \\ 
        &\le \underbrace{\sup_{\ell=0}^\infty \sup_{\mathbf{x}\in\spacedom} \left| \frac{\theta_\ell(\mathbf{x}) - \omega\beta(\mathbf{x})}{\beta(\mathbf{x})} \right|}_{=:r(\theta,\omega)} \sum_{\ell=0}^\infty \frac{2\ell+1}{4\pi} \sum_{m=-\ell}^\ell \int_\spacedom \beta(\mathbf{x}) \left| w_{\ell,m}(\mathbf{x}) \right| \left| v_{\ell,m}(\mathbf{x}) \right| \d\mathbf{x} \\
        &\le r(\theta,\omega) \sum_{\ell=0}^\infty \frac{2\ell+1}{4\pi} \sum_{m=-\ell}^\ell \|\sqrt{\beta} \, w_{\ell,m}\|_{L_2(\Omega)} \|\sqrt{\beta} \, v_{\ell,m}\|_{L_2(\Omega)} \\
        &\le r(\theta,\omega) \left( \sum_{\ell=0}^\infty \frac{2\ell+1}{4\pi} \sum_{m=-\ell}^\ell \|\sqrt{\beta} \, w_{\ell,m}\|_{L_2(\Omega)}^2 \right)^\frac12 \left( \sum_{\ell=0}^\infty \frac{2\ell+1}{4\pi} \sum_{m=-\ell}^\ell \|\sqrt{\beta} \, v_{\ell,m}\|_{L_2(\Omega)}^2 \right)^\frac12 \\
        &= r(\theta,\omega) m(w,w^*)^\frac12 m(v,v^*)^\frac12.
    \end{align*}
\end{proof}

\begin{remark} \label{remark:mercer_condn_remark_simplification}
    We point out that $\beta = \theta_0$. Moreover, if the differential scattering cross-section satisfies Mercer's condition \cite{schoenberg1988positive} (so that $s(\cdot,\cdot)$ is symmetric and positive-semidefinite), we have that $r(\theta,\omega)=\max\{\omega,1-\omega\}$.
\end{remark}

Before we analyse the convergence of the iteration \reff{eqn:dgfem_si_generalised_operator}, we briefly remark that it is well-posed and that, for $0\le \omega < 1$, $a(\cdot,\cdot)-\omega m(\cdot,\cdot)$ is coercive with respect to the norm $||| v |||_{a,\omega}:\brokengraphspace\rightarrow\Re$ defined for all $v\in\brokengraphspace$ by:
\begin{align*}
	||| v |||_{a,\omega}^2 &= || \sqrt{\alpha+(1-\omega)\beta} v ||_{L_2(\spaceangleenergydom)}^2 \\
	&\quad\quad + \frac12 \int_\angledom \left( \sum_{\spacekappa\in\spacemesh} ||v^+-v^-||_{\partial_-\spacekappa\setminus\partial\spacedom}^2 + ||v^+||_{\partial\spacekappa\cap\partial\spacedom}^2 \right) \ \d\bm{\mu}
\end{align*}

\noindent in the sense that
\begin{equation*}
	||| v |||_{a,\omega}^2 = a(v,v) - \omega m(v,v)
\end{equation*}

\noindent for all $v\in\brokengraphspace$. Note that $|||\cdot|||_a=|||\cdot|||_{a,0}$.

\begin{theorem} \label{thm:dgfem_general_si_contraction_apost}
    Let $u_h\in\fespace$ denote the exact solution to the DGFEM problem \reff{eqn:dgfem_lbte_full} and $\{u_h^{(n)}\}_{n=0}^\infty\subset\fespace$ denote the sequence of approximations to $u_h$ defined by \reff{eqn:dgfem_si_generalised_operator}. Then we have
    \begin{equation}
        ||| u_h - u_h^{(n+1)} |||_{a,\omega} \le r(\theta,\omega) q_{\beta,\omega} ||| u_h - u_h^{(n)} |||_{a,\omega}
    \end{equation}

    \noindent for all $n\ge0$, where
    \begin{equation*}
        q_{\beta,\omega} = \sup_{\mathbf{x}\in\spacedom} \frac{\beta(\mathbf{x})}{\alpha(\mathbf{x}) + (1-\omega)\beta(\mathbf{x})}.
    \end{equation*}

    \noindent Thus, we have the following \emph{a posteriori} solver error estimate:
    \begin{equation}
        ||| u_h-u_h^{(n+1)} |||_{DG} \le r(\theta,\omega) \sqrt{r_\beta} ||\sqrt{\beta} (u_h^{(n+1)}-u_h^{(n)})||_{L_2(\spaceangleenergydom)}
    \end{equation}

    \noindent for all $n\ge0$, where
    \begin{equation}
        r_\beta = \sup_{\mathbf{x}\in\spacedom} \frac{\beta(\mathbf{x})}{\alpha(\mathbf{x})}.
    \end{equation}
\end{theorem}

\begin{proof}
    The proof closely follows that of Theorem \ref{thm:dgfem_si_contraction_apost}, albeit in the mono-energetic case, whereby we exploit the coercivity of $a(\cdot,\cdot)-\omega m(\cdot,\cdot)$ with respect to $|||\cdot|||_{a,\omega}$. Using Lemma \ref{lemma:scatter_mass_contraction} and \reff{eqn:dgfem_si_generalised_operator}, we have
    \begin{align*}
        ||| u_h-u_h^{(n+1)} |||_{a,\omega}^2 &= a( u_h-u_h^{(n+1)} , u_h-u_h^{(n+1)} ) 
         - \omega m( u_h-u_h^{(n+1)} , u_h-u_h^{(n+1)} ) \\
        &= s( u_h-u_h^{(n)} , u_h-u_h^{(n+1)} ) 
         - \omega m( u_h-u_h^{(n)} , u_h-u_h^{(n+1)} ) \\
        &\le r(\theta,\omega) \ || \sqrt{\beta} (u_h-u_h^{(n)}) ||_{L_2(\spaceangleenergydom)} \ || \sqrt{\beta} (u_h-u_h^{(n+1)}) ||_{L_2(\spaceangleenergydom)} \\
        &\le r(\theta,\omega) \sup_{\mathbf{x}\in\spacedom} \frac{\beta(\mathbf{x})}{\alpha(\mathbf{x})+(1-\omega)\beta(\mathbf{x})} \ ||| u_h-u_h^{(n)} |||_{a,\omega} \ ||| u_h-u_h^{(n+1)} |||_{a,\omega},
    \end{align*}

    \noindent as required. To show the \emph{a posteriori} error bound, we have
    \begin{align*}
        ||| u_h-u_h^{(n+1)} |||_{DG}^2 &\le a(u_h-u_h^{(n+1)},u_h-u_h^{(n+1)}) - s(u_h-u_h^{(n+1)},u_h-u_h^{(n+1)}) \\
        %
        %&= s(u_h-u_h^{(n)},u_h-u_h^{(n+1)}) - s(u_h-u_h^{(n+1)},u_h-u_h^{(n+1)}) \\
        %&\quad\quad + \omega m( u_h^{(n)} - u_h^{(n+1)} , u_h-u_h^{(n+1)} ) \\
        %
        &= s(u_h^{(n+1)}-u_h^{(n)} , u_h-u_h^{(n+1)}) - \omega m(u_h^{(n+1)}-u_h^{(n)} , u_h-u_h^{(n+1)}) \\
        &\le r(\theta,\omega) \ || \sqrt{\beta} (u_h^{(n+1)}-u_h^{(n)}) ||_{L_2(\spaceangleenergydom)} \ || \sqrt{\beta} (u_h-u_h^{(n+1)}) ||_{L_2(\spaceangleenergydom)} \\
        &\le r(\theta,\omega) \sqrt{\sup_{\mathbf{x}\in\spacedom} \frac{\beta(\mathbf{x})}{\overline{\alpha}(\mathbf{x})}} \ || \sqrt{\beta} (u_h^{(n+1)}-u_h^{(n)}) ||_{L_2(\spaceangleenergydom)} \ || u_h-u_h^{(n+1)} ||_{DG}.
    \end{align*}

    \noindent Noting that the integrals over $\energydom$ appearing in \reff{eqn:macro_scatter_cs} and \reff{eqn:macro_inverse_scatter_cs} are removed in the mono-energetic case, and that $\theta(\mathbf{x},\bm{\mu}'\cdot\bm{\mu},E'\rightarrow E)=\theta(\mathbf{x},\bm{\mu}'\cdot\bm{\mu})$, we have that $\beta=\gamma$ and so \reff{eqn:alphabar_defn} reduces to $\overline{\alpha}=\alpha$. The second part of the theorem is proven on rearrangement.
\end{proof}

\begin{corollary}
    If $r(\theta,\omega) q_{\beta,\omega}<1$, the sequence of approximate solutions $\{u_h^{(n)}\}\subset\fespace$ generated by \reff{eqn:dgfem_si_generalised_operator} satisfies $u_h^{(n)}\rightarrow u_h$ as $n\rightarrow\infty$.
\end{corollary}

As in the case of source iteration, we note that generalised source iteration can be understood as a preconditioned Richardson iteration of the form given in \reff{eqn:dgfem_general_richardson} where now the preconditioner $P^{-1}=(A-\omega M)^{-1}$. Likewise, we also note that the spectral radius $\rho(G_\omega)$ of the generalised source iteration operator $G_\omega = (A-\omega M)^{-1}(S-\omega M)$ satisfies $\rho(G) \le c_{mono,\omega}$ where
\begin{equation*}
    c_{mono,\omega} = r(\theta,\omega) \cdot \sup_{\mathbf{x}\in\spacedom} \frac{\beta(\mathbf{x})}{\alpha(\mathbf{x})+(1-\omega)\beta(\mathbf{x})} = \frac{r(\theta,\omega) c_{mono}}{1-\omega c_{mono}},
\end{equation*}

\noindent as suggested by Theorem \ref{thm:dgfem_general_si_contraction_apost}. The proof follows identically as in the case of standard (poly-energetic) source iteration.

Under the assumptions on $s(\cdot,\cdot)$ in Remark \ref{remark:mercer_condn_remark_simplification}, it can be shown that the spectrum of $(A-\omega M)^{-1}(S-\omega M)$ lies in the union of two discs $D(a_1,r_1)\cup D(a_2,r_2)$ of radii $r_1=\frac{(1-\omega) c_{mono}}{2(1-\omega c_{mono})}$, $r_2=\frac{\omega c_{mono}}{2(1-\omega c_{mono})}$ and centred at $a_1=\frac{(1-\omega)c_{mono}}{2(1-\omega c_{mono})}$, $a_2=\frac{-\omega c_{mono}}{2(1-\omega c_{mono})}$, respectively, cf. \cite{radley2023discontinuous}.

\subsection{GMRES}

The generalised minimal residual (GMRES) method \cite{saad1986gmres} is one of the most popular Krylov subspace methods for large, nonsymmetric sparse systems of linear equations and has been proposed for neutron transport applications in \cite{patton2002application}. Given the problem
\begin{equation} \label{eqn:general_operator_eqn}
    Ax = b,
\end{equation}

\noindent where $x$ and $b$ are elements of a (possibly infinite-dimensional) complex separable Hilbert space $X$ endowed with an inner product $(\cdot,\cdot)_X$, $A:X\rightarrow X$ denotes a bounded and invertible linear operator, and $r_0=b-Ax_0$ denotes the initial residual induced by an initial guess $x_0\in X$, GMRES constructs a sequence of Krylov subspaces $\{K_n\}_{n\ge0}$, where
$
    K_n = \textnormal{span}\{ r_0, Ar_0, \dots, A^{n-1}r_0 \}.
$
At iteration $n$, the $n^{th}$ iterate $x_n\in X$ is chosen to satisfy
\begin{equation} \label{eqn:gmres_minim_problem}
    x_n = {\arg\min}_{y\in x_0+K_n} || b - Ay ||_X
\end{equation}

\noindent and GMRES computes the residual error $||r_n||_X=||b-Ax_n||_X$ as a byproduct. Typically, $X$ is usually taken to be $\mathbb{R}^n$ or $\mathbb{C}^n$ and the inner product $(\cdot,\cdot)_X$ is taken to be the Euclidean inner product; generalisations to other Hilbert spaces have been considered, cf., \cite{gasparo2008some,gunnel2014note}.

When $X=\mathbb{R}^n$ endowed with the usual inner product and $A=Q\Lambda Q^{-1}$ represents a diagonalisable matrix with $\Lambda=\textnormal{diag}\{\lambda_1,\dots,\lambda_n\}$ a diagonal matrix of eigenvalues and $Q$ a matrix of corresponding eigenvectors, the residual error at the $k^{th}$ iteration of GMRES satisfies \cite{saad2003iterative}
\begin{equation*}
    \frac{||r_k||_2}{||r_0||_2} \le ||Q||_2 ||Q^{-1}||_2 \min_{p\in\mathbb{P}_k,p(0)=1} \max_{i=1,\dots,n} |p(\lambda_i)|,
\end{equation*}

\noindent where $\mathbb{P}_k$ denotes the set of polynomials of degree $k$ with leading coefficient equal to 1. In particular, if $\sigma(A)$ is contained in the disc $D(a,r)$ centred at $a\in\mathbb{C}$ with radius $r>0$, then we have
\begin{align*}
    \min_{p\in\mathbb{P}_k,p(0)=1} \max_{i=1,\dots,n} |p(\lambda_i)| &\le \min_{p\in\mathbb{P}_k,p(0)=1} \max_{\lambda\in D(a,r)} |p(\lambda)| \\
    &\le \max_{\lambda\in D(a,r)} \left| \left( 1-\frac{z}{a} \right)^k \right| = \left( \frac{r}{|a|} \right)^k.
\end{align*}

While we do not give a formal proof of convergence of GMRES applied to the operator equation \reff{eqn:dgfem_lbte_operator}, we point out that pre-multiplying the equation by $A^{-1}$ yields the following (equivalent) problem for $u_h$:
\begin{equation} \label{eqn:dgfem_lbte_operator_preconditioned}
    (I-A^{-1}S) u_h = A^{-1} F.
\end{equation}

\noindent By earlier remarks, we recall that $\sigma(I-A^{-1}S) \subset D(a,r)$ with $a=1$ and $r=\sqrt{q_\beta q_\gamma}$; moreover, if $s(\cdot,\cdot)$ is symmetric positive-semidefinite, we have that $\sigma(I-A^{-1}S) \subset D(a,r)$ with $a=1-\frac12\sqrt{q_\beta q_\gamma}$ and $r=\frac12\sqrt{q_\beta q_\gamma}$. Thus, GMRES applied to the preconditioned operator equation \reff{eqn:dgfem_lbte_operator_preconditioned} yields
\begin{equation*}
    \frac{||r_k||_2}{||r_0||_2} \le ||Q||_2 ||Q^{-1}||_2 \left( q_\beta q_\gamma \right)^{\frac{k}{2}}
\end{equation*}

\noindent in the former case and 
\begin{equation*}
    \frac{||r_k||_2}{||r_0||_2} \le ||Q||_2 ||Q^{-1}||_2 \left( \frac{\sqrt{q_\beta q_\gamma}}{2-\sqrt{q_\beta q_\gamma}} \right)^k
\end{equation*}

\noindent in the latter case.

As remarked in \cite{wathen2007preconditioning}, the $\ell_2$-norm of the residual $r_k$ is not always the best norm in which to measure the true solver error in finite element applications. In Section \ref{section:implementation}, we discuss how the operator equation \reff{eqn:general_operator_eqn} may be modified in order for standard implementations of GMRES to solve the minimisation problem \reff{eqn:gmres_minim_problem} in the case where $X=(\fespace,||\cdot||_{L_2(\spaceangleenergydom;\overline{\alpha})})$ denotes a finite element space equipped with a weighted $L_2$-norm $||\cdot||_{L_2(\spaceangleenergydom;\overline{\alpha})}$.

\subsection{Residual-based \emph{a posteriori} solver error estimation}

For any $w_h\in\fespace$, we introduce the \emph{residual functional} $R[w_h]:\fespace\rightarrow\Re$ defined for all $v_h\in\fespace$ by
\begin{equation} \label{eqn:solver_residual_functional}
    R[w_h](v_h) = \ell(v_h) - \left[ a(w_h,v_h) - s(w_h,v_h) \right].
\end{equation}

\noindent We remark that \reff{eqn:solver_residual_functional}, as well as the theorem below, hold in both the poly-energetic or mono-energetic settings with appropriate modifications to the finite element space $\fespace$.

\begin{theorem} \label{thm:general_solver_apost_est}
    Let $\hat{u}_h\in\fespace$ denote an approximation to the exact solution $u_h\in\fespace$ of \reff{eqn:dgfem_lbte_full}. Then the following \emph{a posteriori} error estimate holds:
    \begin{equation}
        ||| u_h - \hat{u}_h |||_{DG} \le \sqrt{\sum_{\kappa\in\mesh} \eta_\kappa(\hat{u}_h)^2} = || \sqrt{\overline{\alpha}} r_h(\hat{u}_h)||_{L_2(\spaceangleenergydom)},
    \end{equation}

    \noindent where $\eta_\kappa(\hat{u}_h) = || \sqrt{\overline{\alpha}} r_{h}(\hat{u}_h) ||_{L_2(\kappa)}$ and $r_{h}(\hat{u}_h)\in\fespace$ denotes the unique solution to the variational problem
    \begin{equation} \label{eqn:apost_error_riesz_residual_fn}
        ( \overline{\alpha} r_h(\hat{u}_h) , v_h )_{L_2(\spaceangleenergydom)} = R[\hat{u}_h](v_h) \quad \forall v_h\in\fespace.
    \end{equation}
\end{theorem}

\begin{proof}
    We have that
    \begin{align*}
        ||| u_h-\hat{u}_h |||_{DG}^2 &\le a(u_h-\hat{u}_h,u_h-\hat{u}_h) - s(u_h-\hat{u}_h,u_h-\hat{u}_h) \\
        &= \ell(u_h-\hat{u}_h) - \left[ a(\hat{u}_h,u_h-\hat{u}_h) - s(\hat{u}_h,u_h-\hat{u}_h) \right] \\
        &= R[\hat{u}_h](u_h-\hat{u}_h) 
        = (\overline{\alpha} r_h(\hat{u}_h) , u_h-\hat{u}_h)_{L_2(\spaceangleenergydom)} \\
        %
        %&= \sum_{\kappa\in\mesh} (\overline{\alpha} r_h(\hat{u}_h) , u_h-\hat{u}_h)_{L_2(\kappa)} \\
        %
        %&\le \sum_{\kappa\in\mesh} || \sqrt{\overline{\alpha}} r_h(\hat{u}_h) ||_{L_2(\kappa)} \ || \sqrt{\overline{\alpha}} (u_h-\hat{u}_h) ||_{L_2(\kappa)} \\
        %
        %&\le \left( \sum_{\kappa\in\mesh} ||\sqrt{\overline{\alpha}} r_h(\hat{u}_h)||_{L_2(\kappa)}^2 \right)^\frac12 \left( \sum_{\kappa\in\mesh} ||\sqrt{\overline{\alpha}}(u_h-\hat{u}_h)||_{L_2(\kappa)}^2 \right)^\frac12 \\
        %
        &\le \left( \sum_{\kappa\in\mesh} ||\sqrt{\overline{\alpha}} r_h(\hat{u}_h)||_{L_2(\kappa)}^2 \right)^\frac12 ||| u_h-\hat{u}_h |||_{DG},
    \end{align*}

    \noindent as required.
\end{proof}

\begin{remark}
    The computation of the local solver error estimates $\eta_\kappa(\hat{u}_h)$ requires the solution $r_h(\hat{u}_h)$ of \reff{eqn:apost_error_riesz_residual_fn}. This can be achieved via local problems on each $\kappa\in\mesh$ as follows. Let $\fespace_\kappa\subset\fespace$ denote the finite element subspace of piecewise-polynomial functions supported on $\kappa=\spacekappa\times\anglekappa\times\kappa_g$. The local problem for $r_\kappa(\hat{u}_h) := r_h(\hat{u}_h)|_\kappa$ reads as follows: find $r_\kappa(\hat{u}_h)\in\fespace_\kappa$ such that
    \begin{equation*}
        (\overline{\alpha} r_\kappa(\hat{u}_h), v_h)_{L_2(\kappa)} = R_\kappa[\hat{u}_h](v_h)
        \quad \forall v_h\in\fespace_\kappa,
    \end{equation*}

    \noindent  where we note that $R[\hat{u}_h](v_h)$ (with $v_h\in\fespace$) can be written in the form
    \begin{equation*}
        R[\hat{u}_h](v_h) = \sum_{\kappa\in\mesh} R_\kappa[\hat{u}_h](v_h),
    \end{equation*}

    \noindent with each $R_\kappa[\hat{u}_h](v_h)$ given by
    \begin{align*}
        R_\kappa[\hat{u}_h](v_h) &= \int_{\kappa_g} \int_{\anglekappa} \int_{\spacekappa} \left( f + S[\hat{u}_h] - \bm{\mu}\cdot\nabla_\mathbf{x}\hat{u}_h - (\alpha+\beta)\hat{u}_h \right) v_h \d\mathbf{x}\d\bm{\mu}\d E \\
        &\quad\quad + \int_{\kappa_g} \int_{\anglekappa} \int_{\partial_-\spacekappa\setminus\partial\spacedom} |\bm{\mu}\cdot\mathbf{n}| (\hat{u}_h^- - \hat{u}_h^+)v_h^+ \d s\d\bm{\mu}\d E \\
        &\quad\quad + \int_{\kappa_g} \int_{\anglekappa} \int_{\partial_-\spacekappa\cap\partial\spacedom} |\bm{\mu}\cdot\mathbf{n}| (g_D - \hat{u}_h^+)v_h^+ \d s\d\bm{\mu}\d E.
    \end{align*}
\end{remark}

\section{Implementation} \label{section:implementation}

We give a brief overview of the practical implementations of the source iteration, generalised source iteration and GMRES \emph{a posteriori} solver error estimates. The starting point for all three methods is a matrix representation of the DGFEM scheme \reff{eqn:dgfem_lbte_full}. Denoting by $\{\phi_i\}_{i=1}^N\subset\fespace$, $N=\dim\fespace$, a basis of the space-angle (or space-angle-energy) finite element space, we introduce the matrices $\mathbf{A},\mathbf{S},\mathbf{M}\in\mathbb{R}^{N\times N}$ with entries
$(\mathbf{A})_{ij} = a(\phi_j,\phi_i)$, $(\mathbf{S})_{ij} = s(\phi_j,\phi_i)$, $(\mathbf{M})_{ij} = m(\phi_j,\phi_i)$, respectively, and the vector $\mathbf{F}\in\mathbb{R}^{N}$, where $(\mathbf{F})_i=\ell(\phi_i)$. Thereby, \reff{eqn:dgfem_lbte_full} can be recast as the linear equation
\begin{equation} \label{eqn:dgfem_lbte_matrix_form}
    \mathbf{Au} = \mathbf{Su}+\mathbf{F}
\end{equation}

\noindent for the vector of coefficients $\mathbf{u}\in\mathbb{R}^N$ in the expansion of $u_h = \sum_{i=1}^N (\mathbf{u})_i \phi_i$.

Since the dimension of the finite element space $N$ may be very large, it is infeasible to store $\mathbf{A}$, $\mathbf{S}$ and $\mathbf{M}$ in full. We instead opt for matrix-free actions of $\mathbf{A}^{-1}$ and $\mathbf{S}$ on vectors in $\mathbb{R}^N$, cf. \cite{houston2023efficient}, in the case of source iteration, or matrix-free actions of $(\mathbf{A}-\omega\mathbf{M})^{-1}$ and $\mathbf{S}-\omega\mathbf{M}$ on vectors in $\mathbb{R}^{N}$ in the case of generalised source iteration. Since $\mathbf{M}$ is a block-diagonal matrix with each block corresponding to a space-angle element, we remark that the action $(\mathbf{A}-\omega\mathbf{M})^{-1}$ (respectively, $\mathbf{S}-\omega\mathbf{M}$) on a vector is not more computationally challenging than that of $\mathbf{A}^{-1}$ (respectively, $\mathbf{S}$).

\subsection{Source iteration and generalised source iteration}

The source iteration scheme, as well as its generalised variant in the mono-energetic case, can be expressed in the form
\begin{equation} \label{eqn:dgfem_generalised_si_matrix_form}
    \mathbf{u}_{n+1} = \mathbf{P}^{-1}(\mathbf{Ku}_n + \mathbf{F}),
\end{equation}

\noindent where $\mathbf{u}_n$ (respectively, $\mathbf{u}_{n+1}$) denotes the expansion of $u_h^{(n)}$ (respectively, $u_h^{(n+1)}$) in the basis of $\fespace$. The matrices $\mathbf{P},\mathbf{K}\in\mathbb{R}^{N\times N}$ are given, respectively, by $\mathbf{P}=\mathbf{A}$ and $\mathbf{K}=\mathbf{S}$ in the source iteration case and by $\mathbf{P}=\mathbf{A}-\omega\mathbf{M}$ and $\mathbf{K}=\mathbf{S}-\omega\mathbf{M}$ in the generalised source iteration case. It can be shown that, using the initial condition $\mathbf{u}_0=\bm{0}$, the $n^{th}$ iterate can be expressed as
\begin{equation} \label{eqn:matrix_iterates}
    \mathbf{u}_n = \left(\sum_{k=0}^{n-1} \mathbf{G}^k\right)\mathbf{P}^{-1}\mathbf{F},
\end{equation}

\noindent where $\mathbf{G}=\mathbf{P}^{-1}\mathbf{K}$.
It is convenient to recast this iteration as
\begin{align*}
    \mathbf{u}_{n+1} &= \mathbf{u}_n+\mathbf{r}_{n}, \\
    \mathbf{r}_{n+1} &= \mathbf{Gr}_n,
\end{align*}

\noindent with initial conditions $\mathbf{u}_0=\bm{0}$ and $\mathbf{r}_0=\mathbf{P}^{-1}\mathbf{F}$, in which case it is equivalent to \reff{eqn:dgfem_generalised_si_matrix_form}. It can be shown that this iteration generates the same iterates \reff{eqn:matrix_iterates} as well as a sequence of preconditioned residuals $\mathbf{r}_n = \mathbf{P}^{-1}(\mathbf{F}-(\mathbf{A}-\mathbf{S})\mathbf{u}_n)$.

Since the \emph{a posteriori} error estimates of Theorems \ref{thm:dgfem_si_contraction_apost} and \ref{thm:dgfem_general_si_contraction_apost} are of the form
\begin{equation*}
    ||| u_h-u_h^{(n+1)} |||_{DG} \le c || \sqrt{\beta} (u_h^{(n)}-u_h^{(n+1)}) ||_{L_2(\spaceangleenergydom)},
\end{equation*}

\noindent we need only compute the constant $c$ for the given problem and solver and assemble the \emph{a posteriori} error estimate element-wise. Furthermore, since
\begin{equation*}
    u_h^{(n+1)} - u_h^{(n)} = \sum_{i=1}^N (\mathbf{r}_{n})_i \phi_i,
\end{equation*}

\noindent we may write
\begin{equation*}
    || \sqrt{\beta} (u_h^{(n+1)} - u_h^{(n)}) ||_{L_2(\spaceangleenergydom)} = \sqrt{\mathbf{r}_n\cdot\mathbf{Mr}_n}.
\end{equation*}

We briefly discuss the cost per iteration of the method \reff{eqn:dgfem_generalised_si_matrix_form}. The matrices $\mathbf{P}$ and $\mathbf{K}$ may be expressed as block-matrices $\mathbf{P}=(\mathbf{P}_{ij})_{i,j=1}^{M}$ and $\mathbf{K}=(\mathbf{K}_{ij})_{i,j=1}^{M}$, where each $\mathbf{P}_{ij},\mathbf{K}_{ij}\in\mathbb{R}^{\dim\spacefespace\times\dim\spacefespace}$ and $M=\dim\anglefespace$ in the mono-energetic setting or $M=\dim(\anglefespace\otimes\energyfespace)$ in the poly-energetic setting. One typically has $\mathbf{P}_{ij}=\bm{0}$ for $i\ne j$. In the classical multigroup discrete ordinates treatment of the linear Boltzmann transport equation, and in the implementation given in \cite{houston2023efficient}, the on-diagonal matrix blocks $\mathbf{P}_{ii}$ are themselves block-diagonal, with each block corresponding to the system matrix arising from a DGFEM discretisation of a transport problem. For this reason, the action of $\mathbf{P}^{-1}$ is assumed to be computationally inexpensive. On the other hand, the matrix $\mathbf{K}$ is block-dense, owing to the coupling between angle-energy element pairs via \reff{eqn:dgfem_scattering operator}. Thus, the action of $\mathbf{K}$ typically dominates the computational time taken to perform a single iteration of \reff{eqn:dgfem_generalised_si_matrix_form}.

\subsection{GMRES}

We employ a standard implementation of GMRES which accepts (the action of) a matrix $\mathbf{A}$, a right-hand side $\mathbf{b}$ and desired tolerance \verb|TOL|. We denote this method by \verb|gmres|$(\mathbf{A},\mathbf{b},$\verb|TOL|$)$, which returns an approximate solution $\hat{\mathbf{x}}$ to the matrix equation $\mathbf{Ax}=\mathbf{b}$ as well as the corresponding residual norm $||\mathbf{r}||_2=||\mathbf{b}-\mathbf{A}\hat{\mathbf{x}}||_2 \le $\verb|TOL|. For simplicity of presentation, we do not consider using GMRES with restarts, but stress that the implementation of the method remains unchanged in this setting.

By a judicious selection of $\mathbf{A}$ and $\mathbf{b}$, we can restructure the original matrix equation \reff{eqn:dgfem_lbte_matrix_form} in such a way that \verb|gmres| returns the \emph{a posteriori} error estimate in Theorem \ref{thm:general_solver_apost_est}. Moreover, we can do so using any preconditioner $\mathbf{P}\approx\mathbf{A}-\mathbf{S}$ that we have available. Let $\hat{\mathbf{M}}\in\mathbb{R}^{N\times N}$ denote the matrix with entries
\begin{equation*}
    (\hat{\mathbf{M}})_{ij} = \int_\energydom \int_\angledom \int_\spacedom \overline{\alpha}(\mathbf{x},\bm{\mu},E) \phi_i(\mathbf{x},\bm{\mu},E) \phi_j(\mathbf{x},\bm{\mu},E) \ \d\mathbf{x}\d\bm{\mu}\d E.
\end{equation*}

\noindent Moreover, let $\mathbf{L}$ denote the lower-triangular matrix arising in the Cholesky decomposition of $\hat{\mathbf{M}}=\mathbf{LL}^*$. By executing
\begin{align}
    \hat{\mathbf{z}} &= \mbox{\tt gmres}(\mathbf{L}^{-1}(\mathbf{A}-\mathbf{S})\mathbf{P}^{-1}\mathbf{L},\mathbf{L}^{-1}\mathbf{F},\mbox{\tt TOL}), \label{eqn:gmres_apost_est_1} \\
    \hat{\mathbf{u}} &= \mathbf{P}^{-1}\mathbf{L}\hat{\mathbf{z}}, \label{eqn:gmres_apost_est_2}
\end{align}

\noindent we can ensure that $|||u_h-\hat{u}_h|||_{DG} \le $ \verb|TOL|. To see this, note that the first equation yields an auxiliary solution $\hat{\mathbf{z}}$ satisfying
\begin{equation*}
    || \mathbf{L}^{-1} (\mathbf{F}-(\mathbf{A}-\mathbf{S})\mathbf{P}^{-1}\mathbf{L}\hat{\mathbf{z}}) ||_2 \le \textnormal{\tt TOL},
\end{equation*}

\noindent and the second equation tells us that the left-hand-side is equal to $||\mathbf{L}^{-1}\hat{\mathbf{r}}||_2 = \sqrt{\hat{\mathbf{r}}^*\hat{\mathbf{M}}^{-1}\hat{\mathbf{r}}}$, where $\hat{\mathbf{r}}=\mathbf{F}-(\mathbf{A}-\mathbf{S})\hat{\mathbf{u}}$ denotes a residual vector. Finally, we relate the residual vector $\hat{\mathbf{r}}$ to the residual function $r_h(\hat{u}_h)$ appearing in the statement of Theorem \ref{thm:general_solver_apost_est} through the latter's expansion in the basis of $\fespace$:
\begin{equation*}
    r_h(\hat{u}_h) = \sum_{j=1}^N (\mathbf{r})_j \phi_j.
\end{equation*}

\noindent Using this expansion of $r_h(\hat{u}_h)$ and setting $v_h=\phi_i$ for $1\le i\le N$ in \reff{eqn:apost_error_riesz_residual_fn} yields the following equation:
\begin{equation*}
    \hat{\mathbf{M}}\mathbf{r} = \hat{\mathbf{r}}.
\end{equation*}

\noindent Thus,
\begin{equation*}
    \sqrt{\hat{\mathbf{r}}^*\hat{\mathbf{M}}^{-1}\hat{\mathbf{r}}} = \sqrt{\mathbf{r}^*\hat{\mathbf{M}}\mathbf{r}} = || \sqrt{\overline{\alpha}} r_h(\hat{u}_h) ||_{L_2(\spaceangleenergydom)},
\end{equation*}

\noindent which is precisely the \emph{a posteriori} error estimator of Theorem \ref{thm:general_solver_apost_est}.

For simplicity, we consider the choice of preconditioner $\mathbf{P}=\mathbf{A}$; i.e. the transport component of the original system matrix in \reff{eqn:dgfem_lbte_matrix_form}. This allows us to write the preconditioned system matrix passed to \verb|gmres| as $\mathbf{L}^{-1}(\mathbf{I}-\mathbf{SA}^{-1})\mathbf{L}$. In addition to the computation of the actions of $\mathbf{A}^{-1}$ and $\mathbf{S}$ on a vector as outlined before for source iteration, the GMRES scheme \reff{eqn:gmres_apost_est_1}-\reff{eqn:gmres_apost_est_2} additionally requires the computation of the Cholesky decomposition of a (possibly very large) matrix $\hat{\mathbf{M}}$, which in principle could be very computationally costly. However, if one employs an orthogonal basis $\{\phi_i\}_{i=1}^N$ with respect to the $\overline{\alpha}$-weighted $L_2(\spaceangleenergydom)$-inner product, both $\hat{\mathbf{M}}$ and $\mathbf{L}$ become diagonal matrices, the latter of which has diagonal entries $(\mathbf{L})_{ii}=||\sqrt{\overline{\alpha}}\phi_i||_{L_2(\spaceangleenergydom)}$. For this reason, the actions of $\mathbf{L}$ and $\mathbf{L}^{-1}$ at each GMRES iteration are assumed to be slightly faster than that of $\mathbf{A}^{-1}$ and significantly faster than that of $\mathbf{S}$. Thus, the cost of evaluating the action of the matrix $\mathbf{L}^{-1}(\mathbf{A}-\mathbf{S})\mathbf{A}^{-1}\mathbf{L}=\mathbf{L}^{-1}(\mathbf{I}-\mathbf{SA}^{-1})\mathbf{L}$ on a vector is comparable to a single step of source iteration, which performs the action of $\mathbf{A}^{-1}\mathbf{S}$ on a vector.

Finally, we consider the cost of orthogonalisation in the GMRES algorithm. It is expected that the cost of orthogonalising one vector against another is comparable to that of applying $\mathbf{L}$ or $\mathbf{L}^{-1}$ to a vector. However, the number of orthogonalisations performed at iteration $n$ of GMRES is equal to $n$. Thus, provided that the dimension of the largest Krylov subspace generated is not too large, we still expect that the action of $\mathbf{S}$ at each iteration is the most computationally-expensive operation in the scheme. We conclude that a single step of GMRES is thus comparable in terms of computational cost to a single step of source iteration in the sense that the additional cost of vector orthogonalisation, as well as the actions of $\mathbf{L}$ and $\mathbf{L}^{-1}$, are small compared to those of $\mathbf{A}^{-1}$ and $\mathbf{S}$ which are required by both algorithms.

We briefly remark on the importance of the factor $\mathbf{L}^{-1}$ in the \emph{a posteriori} error estimator $||\mathbf{L}^{-1}\hat{\mathbf{r}}||_2 = ||\sqrt{\overline{\alpha}}r_h(\hat{u}_h)||_{L_2(\spaceangleenergydom)}$. As noted earlier, equations \reff{eqn:gmres_apost_est_1}-\reff{eqn:gmres_apost_est_2}, generate a discrete solution $\hat{\mathbf{u}}$ and thus a corresponding finite element function $\hat{u}_h$ for which $|||u_h-\hat{u}_h|||\le$ \verb|TOL|. Observing that $\mathbf{L}=\textnormal{diag}(||\sqrt{\overline{\alpha}}\phi_i||_{L_2(\spaceangleenergydom)})$, it can be shown that
\begin{equation} \label{eqn:upper_lower_bnds_resvec}
    \min_i ||\sqrt{\overline{\alpha}}\phi_i||_{L_2(\spaceangleenergydom)} \le \frac{||\hat{\mathbf{r}}||_2}{||\sqrt{\overline{\alpha}}r_h(\hat{u}_h)||_{L_2(\spaceangleenergydom)}} \le \max_i ||\sqrt{\overline{\alpha}}\phi_i||_{L_2(\spaceangleenergydom)},
\end{equation}

\noindent Thus, if \reff{eqn:gmres_apost_est_1}-\reff{eqn:gmres_apost_est_2} are executed without including $\mathbf{L}$ and $\mathbf{L}^{-1}$ (or equivalently if one sets $\mathbf{L}=\mathbf{I}$) and the \verb|gmres| call in equation \reff{eqn:gmres_apost_est_1} returns an approximation $\hat{\mathbf{z}}$ such that $||\hat{\mathbf{r}}||_2\le$ \verb|TOL|, then the true DGFEM-energy-norm error actually satisfies
\begin{equation*}
    |||u_h-\hat{u}_h|||_{DG} \le \frac{\mbox{\tt TOL}}{\min_i ||\sqrt{\overline{\alpha}}\phi_i||_{L_2(\spaceangleenergydom)}}.
\end{equation*}

\noindent Furthermore, if the diagonal elements of $\mathbf{L}$ vary in magnitude - for instance, if $\mesh$ is highly non-uniform or if $\overline{\alpha}$ is highly anisotropic - the inequalities \reff{eqn:upper_lower_bnds_resvec} may not be sharp. Thus, the significance of the matrix $\mathbf{L}$ is to weight the components of the residual vector $\hat{\mathbf{r}}$ in such a way that $|||u_h-\hat{u}_h|||_{DG}$ is bounded above by the statement in Theorem \ref{thm:general_solver_apost_est}.

\section{Numerical Results} \label{section:numerics}

\subsection{Mono-energetic benchmark}

We consider the application of mono-energetic source iteration, generalised source iteration and right-preconditioned GMRES for the numerical solution of the DGFEM problem to a benchmark posed in two spatial dimensions and one angular dimension. The focus of this benchmark is to investigate the behaviour of all three methods with respect to changes in discretisation parameters such as mesh granularity and polynomial degree of the approximation, as well as changes in problem parameters such as the domain size and optical thickness, which we take as the product $(\alpha+\beta)h$ between the spatial mesh size parameter and the macroscopic total cross-section.

We employ spatial domains $(0,L)^2$ (in dimensionless units) with $L\in\{\frac{1}{10},1,10\}$ and macroscopic total cross-sections $\sigma=\alpha+\beta$ with $\sigma\in\{\frac{1}{10},1,10\}$. The macroscopic absorption and scattering cross-sections are taken to be the globally-constant functions $\alpha=(1-c)\sigma$ and $\beta=c\sigma$, respectively, where $c\in\{\frac{1}{10},\frac{3}{10},\frac{5}{10},\frac{7}{10},\frac{9}{10}\}$ denotes the scattering ratio $c=\frac{\beta}{\alpha+\beta}$. The differential scattering cross-section is taken to be $\theta = \frac{\beta}{|\angledom|}$, which models isotropic scattering. For each test problem, the forcing data $f$ and the boundary data $g_D$ are chosen such that the analytical solution to \reff{eqn:lbte_exact} is given by
\begin{equation*}
    u(\mathbf{x},\bm{\mu}) = \textnormal{exp}\left( -(\mathbf{x}\cdot\bm{\mu})^2 \right).
\end{equation*}

A sequence of space-angle meshes $\mesh=\spacemesh\times\anglemesh$ is employed, consisting of \linebreak$|\spacemesh|\in\{ 4,16,64,256 \}$ spatial elements and $|\anglemesh|\in\{ 8,16,32,64 \}$ curved angular elements (constructed as in Section \ref{section:discretisation}). Each space-angle element $\kappa\in\spacekappa\times\anglekappa\in\mesh$ is endowed with a finite element space $\mathbb{P}_p(\spacekappa)\times\mathbb{Q}_q(\anglekappa)$ employing a uniform polynomial degree $p=q\in\{0,1,2\}$. The total number of degrees of freedom of the finite element space $\fespace=\spacefespace\otimes\anglefespace$ varies from $|\fespace|=32$ for the coarsest discretisation to $|\fespace|=294912$ for the finest discretisation.

The \emph{a posteriori} error estimates of Theorems \ref{thm:dgfem_si_contraction_apost} and \ref{thm:dgfem_general_si_contraction_apost} are applied for the standard and generalised versions of source iterations, whereas the residual-based \emph{a posteriori} error estimate of Theorem \ref{thm:general_solver_apost_est} is applied for the right-preconditioned GMRES method. In the case of generalised source iteration, the optimal relaxation parameter $\omega=\frac12$, which minimises the contraction factor in the statement of Theorem \ref{thm:dgfem_general_si_contraction_apost}, is employed.

\subsubsection{Independence on discretisation parameters}

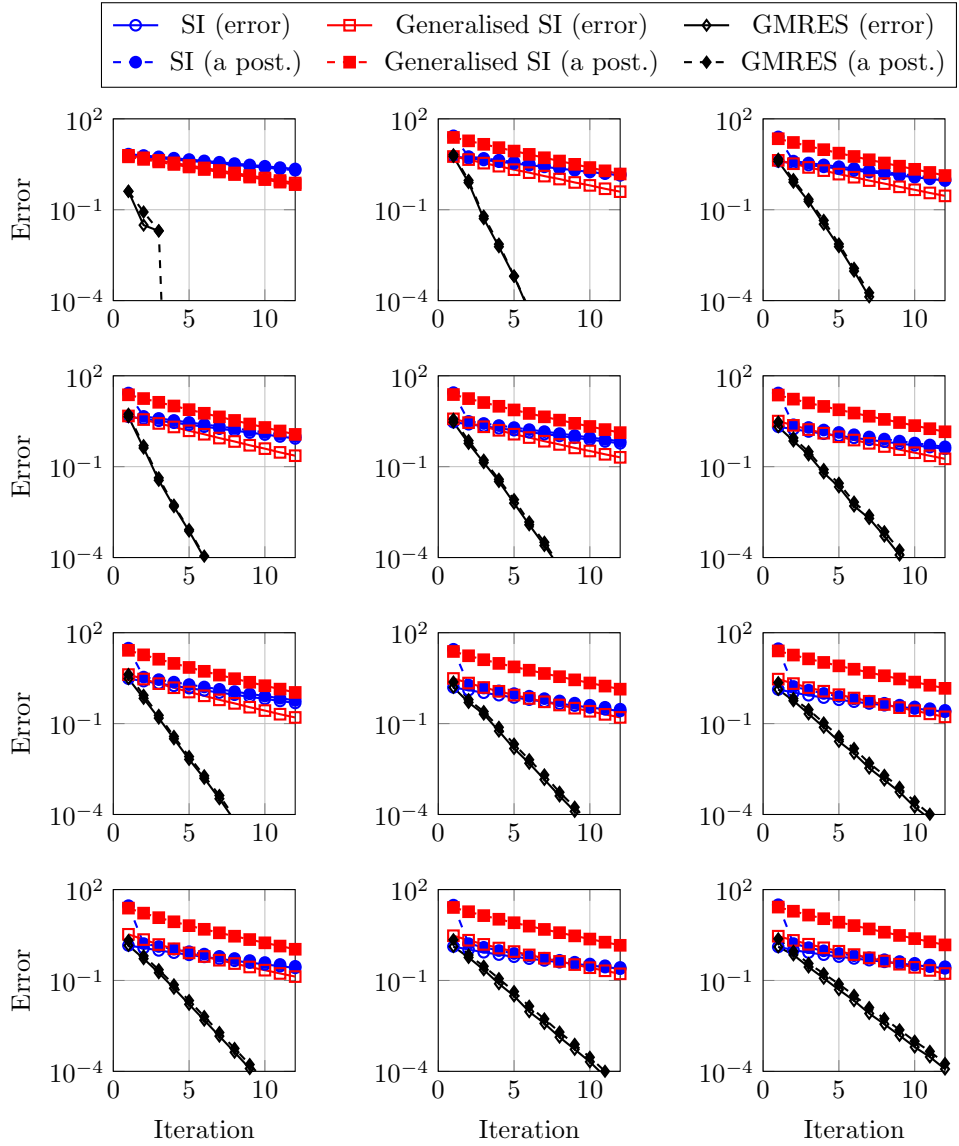
\begin{figure}[h!]
\centering
\begin{tikzpicture}
	% Plot SI vs MSI vs GMRES convergence results for each (h,p) pair
	\begin{groupplot}[group style={group size= 3 by 4, horizontal sep=1.9cm, vertical sep=1cm}, height=4cm, width=4cm]
		% (ref,p) = (1,0).
		\nextgroupplot[xmode=linear, xmin=0, xmax=12, xlabel=,
					   ymode=log, log basis y={10}, ymin=1e-4, ymax=1e2, ylabel=Error,
					   axis background/.style={fill=gray!0}, 
					   legend pos=north west,
					   grid=both, grid style={line width=.1pt, draw=gray!10}, major grid style={line width=.2pt,draw=gray!50}]
			
			% source_iteration plots.
			\addplot+[mark=o, thick, solid, blue, mark options={blue, solid}] table [x=krylov_dim, y=dg_error_solver, col sep=comma] {deal_II_results/monoenergetic_benchmark/source_iteration/c9/a1_l1/monoenergetic_apost_si_2D_errors_p0_n1};
			\label{plots:monoenergetic_benchmark_disc_param_indep_si_err}
			\addplot+[mark=*, thick, dashed, blue, mark options={blue, solid}] table [x=krylov_dim, y=solver_apost_estimate, col sep=comma] {deal_II_results/monoenergetic_benchmark/source_iteration/c9/a1_l1/monoenergetic_apost_si_2D_errors_p0_n1}; 
			\label{plots:monoenergetic_benchmark_disc_param_indep_si_est}
			
			% modified_source_iteration plots.
			\addplot+[mark=square, thick, solid, red, mark options={red, solid}] table [x=krylov_dim, y=dg_error_solver, col sep=comma] {deal_II_results/monoenergetic_benchmark/modified_source_iteration/c9/a1_l1/monoenergetic_apost_msi_2D_errors_p0_n1};
			\label{plots:monoenergetic_benchmark_disc_param_indep_msi_err}
			\addplot+[mark=square*, thick, dashed, red, mark options={red, solid}] table [x=krylov_dim, y=solver_apost_estimate, col sep=comma] {deal_II_results/monoenergetic_benchmark/modified_source_iteration/c9/a1_l1/monoenergetic_apost_msi_2D_errors_p0_n1}; 
			\label{plots:monoenergetic_benchmark_disc_param_indep_msi_est}
			
			% GMRES plots.
			\addplot+[mark=diamond, thick, solid, black, mark options={black, solid}] table [x=krylov_dim, y=dg_error_solver, col sep=comma] {deal_II_results/monoenergetic_benchmark/gmres/c9/a1_l1/monoenergetic_apost_gmres_2D_errors_p0_n1};
			\label{plots:monoenergetic_benchmark_disc_param_indep_gmres_err}
			\addplot+[mark=diamond*, thick, dashed, black, mark options={black, solid}] table [x=krylov_dim, y=solver_apost_estimate, col sep=comma] {deal_II_results/monoenergetic_benchmark/gmres/c9/a1_l1/monoenergetic_apost_gmres_2D_errors_p0_n1}; 
			\label{plots:monoenergetic_benchmark_disc_param_indep_gmres_est}

			% Place coordinate for first figure.			
			\coordinate (top) at (rel axis cs:0,1); % coordinate at top of the first plot
			
		% (ref,p) = (1,1).
		\nextgroupplot[xmode=linear, xmin=0, xmax=12, xlabel=,
					   ymode=log, log basis y={10}, ymin=1e-4, ymax=1e2, ylabel=,
					   axis background/.style={fill=gray!0}, 
					   legend pos=north west,
					   grid=both, grid style={line width=.1pt, draw=gray!10}, major grid style={line width=.2pt,draw=gray!50}]
			
			% source_iteration plots.
			\addplot+[mark=o, thick, solid, blue, mark options={blue, solid}] table [x=krylov_dim, y=dg_error_solver, col sep=comma] {deal_II_results/monoenergetic_benchmark/source_iteration/c9/a1_l1/monoenergetic_apost_si_2D_errors_p1_n1};
			\addplot+[mark=*, thick, dashed, blue, mark options={blue, solid}] table [x=krylov_dim, y=solver_apost_estimate, col sep=comma] {deal_II_results/monoenergetic_benchmark/source_iteration/c9/a1_l1/monoenergetic_apost_si_2D_errors_p1_n1}; 
			
			% modified_source_iteration plots.
			\addplot+[mark=square, thick, solid, red, mark options={red, solid}] table [x=krylov_dim, y=dg_error_solver, col sep=comma] {deal_II_results/monoenergetic_benchmark/modified_source_iteration/c9/a1_l1/monoenergetic_apost_msi_2D_errors_p1_n1};
			\addplot+[mark=square*, thick, dashed, red, mark options={red, solid}] table [x=krylov_dim, y=solver_apost_estimate, col sep=comma] {deal_II_results/monoenergetic_benchmark/modified_source_iteration/c9/a1_l1/monoenergetic_apost_msi_2D_errors_p1_n1}; 
			
			% GMRES plots.
			\addplot+[mark=diamond, thick, solid, black, mark options={black, solid}] table [x=krylov_dim, y=dg_error_solver, col sep=comma] {deal_II_results/monoenergetic_benchmark/gmres/c9/a1_l1/monoenergetic_apost_gmres_2D_errors_p1_n1};
			\addplot+[mark=diamond*, thick, dashed, black, mark options={black, solid}] table [x=krylov_dim, y=solver_apost_estimate, col sep=comma] {deal_II_results/monoenergetic_benchmark/gmres/c9/a1_l1/monoenergetic_apost_gmres_2D_errors_p1_n1}; 
					   
		% (ref,p) = (1,2).
		\nextgroupplot[xmode=linear, xmin=0, xmax=12, xlabel=,
					   ymode=log, log basis y={10}, ymin=1e-4, ymax=1e2, ylabel=,
					   axis background/.style={fill=gray!0}, 
					   legend pos=north west,
					   grid=both, grid style={line width=.1pt, draw=gray!10}, major grid style={line width=.2pt,draw=gray!50}]
			
			% source_iteration plots.
			\addplot+[mark=o, thick, solid, blue, mark options={blue, solid}] table [x=krylov_dim, y=dg_error_solver, col sep=comma] {deal_II_results/monoenergetic_benchmark/source_iteration/c9/a1_l1/monoenergetic_apost_si_2D_errors_p2_n1};
			\addplot+[mark=*, thick, dashed, blue, mark options={blue, solid}] table [x=krylov_dim, y=solver_apost_estimate, col sep=comma] {deal_II_results/monoenergetic_benchmark/source_iteration/c9/a1_l1/monoenergetic_apost_si_2D_errors_p2_n1}; 
			
			% modified_source_iteration plots.
			\addplot+[mark=square, thick, solid, red, mark options={red, solid}] table [x=krylov_dim, y=dg_error_solver, col sep=comma] {deal_II_results/monoenergetic_benchmark/modified_source_iteration/c9/a1_l1/monoenergetic_apost_msi_2D_errors_p2_n1};
			\addplot+[mark=square*, thick, dashed, red, mark options={red, solid}] table [x=krylov_dim, y=solver_apost_estimate, col sep=comma] {deal_II_results/monoenergetic_benchmark/modified_source_iteration/c9/a1_l1/monoenergetic_apost_msi_2D_errors_p2_n1}; 
			
			% GMRES plots.
			\addplot+[mark=diamond, thick, solid, black, mark options={black, solid}] table [x=krylov_dim, y=dg_error_solver, col sep=comma] {deal_II_results/monoenergetic_benchmark/gmres/c9/a1_l1/monoenergetic_apost_gmres_2D_errors_p2_n1};
			\addplot+[mark=diamond*, thick, dashed, black, mark options={black, solid}] table [x=krylov_dim, y=solver_apost_estimate, col sep=comma] {deal_II_results/monoenergetic_benchmark/gmres/c9/a1_l1/monoenergetic_apost_gmres_2D_errors_p2_n1}; 
					   
		% (ref,p) = (2,0).
		\nextgroupplot[xmode=linear, xmin=0, xmax=12, xlabel=,
					   ymode=log, log basis y={10}, ymin=1e-4, ymax=1e2, ylabel=Error,
					   axis background/.style={fill=gray!0}, 
					   legend pos=north west,
					   grid=both, grid style={line width=.1pt, draw=gray!10}, major grid style={line width=.2pt,draw=gray!50}]
			
			% source_iteration plots.
			\addplot+[mark=o, thick, solid, blue, mark options={blue, solid}] table [x=krylov_dim, y=dg_error_solver, col sep=comma] {deal_II_results/monoenergetic_benchmark/source_iteration/c9/a1_l1/monoenergetic_apost_si_2D_errors_p0_n2};
			\addplot+[mark=*, thick, dashed, blue, mark options={blue, solid}] table [x=krylov_dim, y=solver_apost_estimate, col sep=comma] {deal_II_results/monoenergetic_benchmark/source_iteration/c9/a1_l1/monoenergetic_apost_si_2D_errors_p0_n2}; 
			
			% modified_source_iteration plots.
			\addplot+[mark=square, thick, solid, red, mark options={red, solid}] table [x=krylov_dim, y=dg_error_solver, col sep=comma] {deal_II_results/monoenergetic_benchmark/modified_source_iteration/c9/a1_l1/monoenergetic_apost_msi_2D_errors_p0_n2};
			\addplot+[mark=square*, thick, dashed, red, mark options={red, solid}] table [x=krylov_dim, y=solver_apost_estimate, col sep=comma] {deal_II_results/monoenergetic_benchmark/modified_source_iteration/c9/a1_l1/monoenergetic_apost_msi_2D_errors_p0_n2}; 
			
			% GMRES plots.
			\addplot+[mark=diamond, thick, solid, black, mark options={black, solid}] table [x=krylov_dim, y=dg_error_solver, col sep=comma] {deal_II_results/monoenergetic_benchmark/gmres/c9/a1_l1/monoenergetic_apost_gmres_2D_errors_p0_n2};
			\addplot+[mark=diamond*, thick, dashed, black, mark options={black, solid}] table [x=krylov_dim, y=solver_apost_estimate, col sep=comma] {deal_II_results/monoenergetic_benchmark/gmres/c9/a1_l1/monoenergetic_apost_gmres_2D_errors_p0_n2};  
					   
		% (ref,p) = (2,1).
		\nextgroupplot[xmode=linear, xmin=0, xmax=12, xlabel=,
					   ymode=log, log basis y={10}, ymin=1e-4, ymax=1e2, ylabel=,
					   axis background/.style={fill=gray!0}, 
					   legend pos=north west,
					   grid=both, grid style={line width=.1pt, draw=gray!10}, major grid style={line width=.2pt,draw=gray!50}]
			
			% source_iteration plots.
			\addplot+[mark=o, thick, solid, blue, mark options={blue, solid}] table [x=krylov_dim, y=dg_error_solver, col sep=comma] {deal_II_results/monoenergetic_benchmark/source_iteration/c9/a1_l1/monoenergetic_apost_si_2D_errors_p1_n2};
			\addplot+[mark=*, thick, dashed, blue, mark options={blue, solid}] table [x=krylov_dim, y=solver_apost_estimate, col sep=comma] {deal_II_results/monoenergetic_benchmark/source_iteration/c9/a1_l1/monoenergetic_apost_si_2D_errors_p1_n2}; 
			
			% modified_source_iteration plots.
			\addplot+[mark=square, thick, solid, red, mark options={red, solid}] table [x=krylov_dim, y=dg_error_solver, col sep=comma] {deal_II_results/monoenergetic_benchmark/modified_source_iteration/c9/a1_l1/monoenergetic_apost_msi_2D_errors_p1_n2};
			\addplot+[mark=square*, thick, dashed, red, mark options={red, solid}] table [x=krylov_dim, y=solver_apost_estimate, col sep=comma] {deal_II_results/monoenergetic_benchmark/modified_source_iteration/c9/a1_l1/monoenergetic_apost_msi_2D_errors_p1_n2}; 
			
			% GMRES plots.
			\addplot+[mark=diamond, thick, solid, black, mark options={black, solid}] table [x=krylov_dim, y=dg_error_solver, col sep=comma] {deal_II_results/monoenergetic_benchmark/gmres/c9/a1_l1/monoenergetic_apost_gmres_2D_errors_p1_n2};
			\addplot+[mark=diamond*, thick, dashed, black, mark options={black, solid}] table [x=krylov_dim, y=solver_apost_estimate, col sep=comma] {deal_II_results/monoenergetic_benchmark/gmres/c9/a1_l1/monoenergetic_apost_gmres_2D_errors_p1_n2}; 
					   
		% (ref,p) = (2,2).
		\nextgroupplot[xmode=linear, xmin=0, xmax=12, xlabel=,
					   ymode=log, log basis y={10}, ymin=1e-4, ymax=1e2, ylabel=,
					   axis background/.style={fill=gray!0}, 
					   legend pos=north west,
					   grid=both, grid style={line width=.1pt, draw=gray!10}, major grid style={line width=.2pt,draw=gray!50}]
			
			% source_iteration plots.
			\addplot+[mark=o, thick, solid, blue, mark options={blue, solid}] table [x=krylov_dim, y=dg_error_solver, col sep=comma] {deal_II_results/monoenergetic_benchmark/source_iteration/c9/a1_l1/monoenergetic_apost_si_2D_errors_p2_n2};
			\addplot+[mark=*, thick, dashed, blue, mark options={blue, solid}] table [x=krylov_dim, y=solver_apost_estimate, col sep=comma] {deal_II_results/monoenergetic_benchmark/source_iteration/c9/a1_l1/monoenergetic_apost_si_2D_errors_p2_n2}; 
			
			% modified_source_iteration plots.
			\addplot+[mark=square, thick, solid, red, mark options={red, solid}] table [x=krylov_dim, y=dg_error_solver, col sep=comma] {deal_II_results/monoenergetic_benchmark/modified_source_iteration/c9/a1_l1/monoenergetic_apost_msi_2D_errors_p2_n2};
			\addplot+[mark=square*, thick, dashed, red, mark options={red, solid}] table [x=krylov_dim, y=solver_apost_estimate, col sep=comma] {deal_II_results/monoenergetic_benchmark/modified_source_iteration/c9/a1_l1/monoenergetic_apost_msi_2D_errors_p2_n2}; 
			
			% GMRES plots.
			\addplot+[mark=diamond, thick, solid, black, mark options={black, solid}] table [x=krylov_dim, y=dg_error_solver, col sep=comma] {deal_II_results/monoenergetic_benchmark/gmres/c9/a1_l1/monoenergetic_apost_gmres_2D_errors_p2_n2};
			\addplot+[mark=diamond*, thick, dashed, black, mark options={black, solid}] table [x=krylov_dim, y=solver_apost_estimate, col sep=comma] {deal_II_results/monoenergetic_benchmark/gmres/c9/a1_l1/monoenergetic_apost_gmres_2D_errors_p2_n2}; 
					   
		% (ref,p) = (3,0).
		\nextgroupplot[xmode=linear, xmin=0, xmax=12, xlabel=,
					   ymode=log, log basis y={10}, ymin=1e-4, ymax=1e2, ylabel=Error,
					   axis background/.style={fill=gray!0}, 
					   legend pos=north west,
					   grid=both, grid style={line width=.1pt, draw=gray!10}, major grid style={line width=.2pt,draw=gray!50}]
			
			% source_iteration plots.
			\addplot+[mark=o, thick, solid, blue, mark options={blue, solid}] table [x=krylov_dim, y=dg_error_solver, col sep=comma] {deal_II_results/monoenergetic_benchmark/source_iteration/c9/a1_l1/monoenergetic_apost_si_2D_errors_p0_n3};
			\addplot+[mark=*, thick, dashed, blue, mark options={blue, solid}] table [x=krylov_dim, y=solver_apost_estimate, col sep=comma] {deal_II_results/monoenergetic_benchmark/source_iteration/c9/a1_l1/monoenergetic_apost_si_2D_errors_p0_n3}; 
			
			% modified_source_iteration plots.
			\addplot+[mark=square, thick, solid, red, mark options={red, solid}] table [x=krylov_dim, y=dg_error_solver, col sep=comma] {deal_II_results/monoenergetic_benchmark/modified_source_iteration/c9/a1_l1/monoenergetic_apost_msi_2D_errors_p0_n3};
			\addplot+[mark=square*, thick, dashed, red, mark options={red, solid}] table [x=krylov_dim, y=solver_apost_estimate, col sep=comma] {deal_II_results/monoenergetic_benchmark/modified_source_iteration/c9/a1_l1/monoenergetic_apost_msi_2D_errors_p0_n3}; 
			
			% GMRES plots.
			\addplot+[mark=diamond, thick, solid, black, mark options={black, solid}] table [x=krylov_dim, y=dg_error_solver, col sep=comma] {deal_II_results/monoenergetic_benchmark/gmres/c9/a1_l1/monoenergetic_apost_gmres_2D_errors_p0_n3};
			\addplot+[mark=diamond*, thick, dashed, black, mark options={black, solid}] table [x=krylov_dim, y=solver_apost_estimate, col sep=comma] {deal_II_results/monoenergetic_benchmark/gmres/c9/a1_l1/monoenergetic_apost_gmres_2D_errors_p0_n3}; 
					   
		% (ref,p) = (3,1).
		\nextgroupplot[xmode=linear, xmin=0, xmax=12, xlabel=,
					   ymode=log, log basis y={10}, ymin=1e-4, ymax=1e2, ylabel=,
					   axis background/.style={fill=gray!0}, 
					   legend pos=north west,
					   grid=both, grid style={line width=.1pt, draw=gray!10}, major grid style={line width=.2pt,draw=gray!50}]
			
			% source_iteration plots.
			\addplot+[mark=o, thick, solid, blue, mark options={blue, solid}] table [x=krylov_dim, y=dg_error_solver, col sep=comma] {deal_II_results/monoenergetic_benchmark/source_iteration/c9/a1_l1/monoenergetic_apost_si_2D_errors_p1_n3};
			\addplot+[mark=*, thick, dashed, blue, mark options={blue, solid}] table [x=krylov_dim, y=solver_apost_estimate, col sep=comma] {deal_II_results/monoenergetic_benchmark/source_iteration/c9/a1_l1/monoenergetic_apost_si_2D_errors_p1_n3}; 
			
			% modified_source_iteration plots.
			\addplot+[mark=square, thick, solid, red, mark options={red, solid}] table [x=krylov_dim, y=dg_error_solver, col sep=comma] {deal_II_results/monoenergetic_benchmark/modified_source_iteration/c9/a1_l1/monoenergetic_apost_msi_2D_errors_p1_n3};
			\addplot+[mark=square*, thick, dashed, red, mark options={red, solid}] table [x=krylov_dim, y=solver_apost_estimate, col sep=comma] {deal_II_results/monoenergetic_benchmark/modified_source_iteration/c9/a1_l1/monoenergetic_apost_msi_2D_errors_p1_n3}; 
			
			% GMRES plots.
			\addplot+[mark=diamond, thick, solid, black, mark options={black, solid}] table [x=krylov_dim, y=dg_error_solver, col sep=comma] {deal_II_results/monoenergetic_benchmark/gmres/c9/a1_l1/monoenergetic_apost_gmres_2D_errors_p1_n3};
			\addplot+[mark=diamond*, thick, dashed, black, mark options={black, solid}] table [x=krylov_dim, y=solver_apost_estimate, col sep=comma] {deal_II_results/monoenergetic_benchmark/gmres/c9/a1_l1/monoenergetic_apost_gmres_2D_errors_p1_n3}; 
					   
		% (ref,p) = (3,2).
		\nextgroupplot[xmode=linear, xmin=0, xmax=12, xlabel=,
					   ymode=log, log basis y={10}, ymin=1e-4, ymax=1e2, ylabel=,
					   axis background/.style={fill=gray!0}, 
					   legend pos=north west,
					   grid=both, grid style={line width=.1pt, draw=gray!10}, major grid style={line width=.2pt,draw=gray!50}]
			
			% source_iteration plots.
			\addplot+[mark=o, thick, solid, blue, mark options={blue, solid}] table [x=krylov_dim, y=dg_error_solver, col sep=comma] {deal_II_results/monoenergetic_benchmark/source_iteration/c9/a1_l1/monoenergetic_apost_si_2D_errors_p2_n3};
			\addplot+[mark=*, thick, dashed, blue, mark options={blue, solid}] table [x=krylov_dim, y=solver_apost_estimate, col sep=comma] {deal_II_results/monoenergetic_benchmark/source_iteration/c9/a1_l1/monoenergetic_apost_si_2D_errors_p2_n3}; 
			
			% modified_source_iteration plots.
			\addplot+[mark=square, thick, solid, red, mark options={red, solid}] table [x=krylov_dim, y=dg_error_solver, col sep=comma] {deal_II_results/monoenergetic_benchmark/modified_source_iteration/c9/a1_l1/monoenergetic_apost_msi_2D_errors_p2_n3};
			\addplot+[mark=square*, thick, dashed, red, mark options={red, solid}] table [x=krylov_dim, y=solver_apost_estimate, col sep=comma] {deal_II_results/monoenergetic_benchmark/modified_source_iteration/c9/a1_l1/monoenergetic_apost_msi_2D_errors_p2_n3}; 
			
			% GMRES plots.
			\addplot+[mark=diamond, thick, solid, black, mark options={black, solid}] table [x=krylov_dim, y=dg_error_solver, col sep=comma] {deal_II_results/monoenergetic_benchmark/gmres/c9/a1_l1/monoenergetic_apost_gmres_2D_errors_p2_n3};
			\addplot+[mark=diamond*, thick, dashed, black, mark options={black, solid}] table [x=krylov_dim, y=solver_apost_estimate, col sep=comma] {deal_II_results/monoenergetic_benchmark/gmres/c9/a1_l1/monoenergetic_apost_gmres_2D_errors_p2_n3}; 
					   
		% (ref,p) = (4,0).
		\nextgroupplot[xmode=linear, xmin=0, xmax=12, xlabel=Iteration,
					   ymode=log, log basis y={10}, ymin=1e-4, ymax=1e2, ylabel=Error,
					   axis background/.style={fill=gray!0}, 
					   legend pos=north west,
					   grid=both, grid style={line width=.1pt, draw=gray!10}, major grid style={line width=.2pt,draw=gray!50}]
			
			% source_iteration plots.
			\addplot+[mark=o, thick, solid, blue, mark options={blue, solid}] table [x=krylov_dim, y=dg_error_solver, col sep=comma] {deal_II_results/monoenergetic_benchmark/source_iteration/c9/a1_l1/monoenergetic_apost_si_2D_errors_p0_n4};
			\addplot+[mark=*, thick, dashed, blue, mark options={blue, solid}] table [x=krylov_dim, y=solver_apost_estimate, col sep=comma] {deal_II_results/monoenergetic_benchmark/source_iteration/c9/a1_l1/monoenergetic_apost_si_2D_errors_p0_n4}; 
			
			% modified_source_iteration plots.
			\addplot+[mark=square, thick, solid, red, mark options={red, solid}] table [x=krylov_dim, y=dg_error_solver, col sep=comma] {deal_II_results/monoenergetic_benchmark/modified_source_iteration/c9/a1_l1/monoenergetic_apost_msi_2D_errors_p0_n4};
			\addplot+[mark=square*, thick, dashed, red, mark options={red, solid}] table [x=krylov_dim, y=solver_apost_estimate, col sep=comma] {deal_II_results/monoenergetic_benchmark/modified_source_iteration/c9/a1_l1/monoenergetic_apost_msi_2D_errors_p0_n4}; 
			
			% GMRES plots.
			\addplot+[mark=diamond, thick, solid, black, mark options={black, solid}] table [x=krylov_dim, y=dg_error_solver, col sep=comma] {deal_II_results/monoenergetic_benchmark/gmres/c9/a1_l1/monoenergetic_apost_gmres_2D_errors_p0_n4};
			\addplot+[mark=diamond*, thick, dashed, black, mark options={black, solid}] table [x=krylov_dim, y=solver_apost_estimate, col sep=comma] {deal_II_results/monoenergetic_benchmark/gmres/c9/a1_l1/monoenergetic_apost_gmres_2D_errors_p0_n4}; 
					   
		% (ref,p) = (4,1).
		\nextgroupplot[xmode=linear, xmin=0, xmax=12, xlabel=Iteration,
					   ymode=log, log basis y={10}, ymin=1e-4, ymax=1e2, ylabel=,
					   axis background/.style={fill=gray!0}, 
					   legend pos=north west,
					   grid=both, grid style={line width=.1pt, draw=gray!10}, major grid style={line width=.2pt,draw=gray!50}]
			
			% source_iteration plots.
			\addplot+[mark=o, thick, solid, blue, mark options={blue, solid}] table [x=krylov_dim, y=dg_error_solver, col sep=comma] {deal_II_results/monoenergetic_benchmark/source_iteration/c9/a1_l1/monoenergetic_apost_si_2D_errors_p1_n4};
			\addplot+[mark=*, thick, dashed, blue, mark options={blue, solid}] table [x=krylov_dim, y=solver_apost_estimate, col sep=comma] {deal_II_results/monoenergetic_benchmark/source_iteration/c9/a1_l1/monoenergetic_apost_si_2D_errors_p1_n4}; 
			
			% modified_source_iteration plots.
			\addplot+[mark=square, thick, solid, red, mark options={red, solid}] table [x=krylov_dim, y=dg_error_solver, col sep=comma] {deal_II_results/monoenergetic_benchmark/modified_source_iteration/c9/a1_l1/monoenergetic_apost_msi_2D_errors_p1_n4};
			\addplot+[mark=square*, thick, dashed, red, mark options={red, solid}] table [x=krylov_dim, y=solver_apost_estimate, col sep=comma] {deal_II_results/monoenergetic_benchmark/modified_source_iteration/c9/a1_l1/monoenergetic_apost_msi_2D_errors_p1_n4}; 
			
			% GMRES plots.
			\addplot+[mark=diamond, thick, solid, black, mark options={black, solid}] table [x=krylov_dim, y=dg_error_solver, col sep=comma] {deal_II_results/monoenergetic_benchmark/gmres/c9/a1_l1/monoenergetic_apost_gmres_2D_errors_p1_n4};
			\addplot+[mark=diamond*, thick, dashed, black, mark options={black, solid}] table [x=krylov_dim, y=solver_apost_estimate, col sep=comma] {deal_II_results/monoenergetic_benchmark/gmres/c9/a1_l1/monoenergetic_apost_gmres_2D_errors_p1_n4}; 
					   
		% (ref,p) = (4,2).
		\nextgroupplot[xmode=linear, xmin=0, xmax=12, xlabel=Iteration,
					   ymode=log, log basis y={10}, ymin=1e-4, ymax=1e2, ylabel=,
					   axis background/.style={fill=gray!0}, 
					   legend pos=north west,
					   grid=both, grid style={line width=.1pt, draw=gray!10}, major grid style={line width=.2pt,draw=gray!50}]
			
			% source_iteration plots.
			\addplot+[mark=o, thick, solid, blue, mark options={blue, solid}] table [x=krylov_dim, y=dg_error_solver, col sep=comma] {deal_II_results/monoenergetic_benchmark/source_iteration/c9/a1_l1/monoenergetic_apost_si_2D_errors_p2_n4};
			\addplot+[mark=*, thick, dashed, blue, mark options={blue, solid}] table [x=krylov_dim, y=solver_apost_estimate, col sep=comma] {deal_II_results/monoenergetic_benchmark/source_iteration/c9/a1_l1/monoenergetic_apost_si_2D_errors_p2_n4}; 
			
			% modified_source_iteration plots.
			\addplot+[mark=square, thick, solid, red, mark options={red, solid}] table [x=krylov_dim, y=dg_error_solver, col sep=comma] {deal_II_results/monoenergetic_benchmark/modified_source_iteration/c9/a1_l1/monoenergetic_apost_msi_2D_errors_p2_n4};
			\addplot+[mark=square*, thick, dashed, red, mark options={red, solid}] table [x=krylov_dim, y=solver_apost_estimate, col sep=comma] {deal_II_results/monoenergetic_benchmark/modified_source_iteration/c9/a1_l1/monoenergetic_apost_msi_2D_errors_p2_n4}; 
			
			% GMRES plots.
			\addplot+[mark=diamond, thick, solid, black, mark options={black, solid}] table [x=krylov_dim, y=dg_error_solver, col sep=comma] {deal_II_results/monoenergetic_benchmark/gmres/c9/a1_l1/monoenergetic_apost_gmres_2D_errors_p2_n4};
			\addplot+[mark=diamond*, thick, dashed, black, mark options={black, solid}] table [x=krylov_dim, y=solver_apost_estimate, col sep=comma] {deal_II_results/monoenergetic_benchmark/gmres/c9/a1_l1/monoenergetic_apost_gmres_2D_errors_p2_n4}; 
		
			% Place coordinate for next figure.			
            \coordinate (bot) at (rel axis cs:1,0); % coordinate at bottom of the last plot
	\end{groupplot}
	
	% legend
	\path (top|-current bounding box.north)--
    	coordinate(legendpos)
    	(bot|-current bounding box.north);
	\matrix[
    	matrix of nodes,
    	anchor=south,
    	draw,
    	inner sep=0.2em,
    	draw
  	]at([yshift=1ex]legendpos)
  	{
    	\ref{plots:monoenergetic_benchmark_disc_param_indep_si_err}& SI (error) &[5pt]
    	\ref{plots:monoenergetic_benchmark_disc_param_indep_msi_err}& Generalised SI (error) &[5pt]
    	\ref{plots:monoenergetic_benchmark_disc_param_indep_gmres_err}& GMRES (error) \\
    	\ref{plots:monoenergetic_benchmark_disc_param_indep_si_est}& SI (a post.) &[5pt]
    	\ref{plots:monoenergetic_benchmark_disc_param_indep_msi_est}& Generalised SI (a post.) &[5pt]
    	\ref{plots:monoenergetic_benchmark_disc_param_indep_gmres_est}& GMRES (a post.) \\};
    	
\end{tikzpicture}
\caption{Independence of the convergence of source iteration, generalised source iteration and right-preconditioned GMRES on the discretisation parameters for the model problem posed on the spatial domain $\spacedom=(0,10)^2$ and with $\sigma=10$, $c=\frac{9}{10}$. Columns (left-to-right): $p=0$, $p=1$, $p=2$. Rows (top-to-bottom): $|\mesh|=32$, $|\mesh|=256$, $|\mesh|=2048$, $|\mesh|=16384$.}
\label{fig:mono_benchmark_param_indep}
\end{figure}

Since the results of Theorems \ref{thm:dgfem_si_contraction_apost} and \ref{thm:dgfem_general_si_contraction_apost} are independent of the mesh size parameters and element-wise polynomial degrees, we expect the behaviour of source iteration and its generalised variant to be similarly independent of such parameters. We also expect the right-preconditioned GMRES method to converge at the same rate independently of the discretisation parameters since we are employing preconditioning based on source iteration. This convergence behaviour of all methods is indeed observed in Figure \ref{fig:mono_benchmark_param_indep}.

\subsubsection{Dependence on scattering ratio}

\begin{figure}[h!]
\centering
\begin{tikzpicture}
	\begin{groupplot}[group style={group size= 2 by 3, horizontal sep=2.2cm, vertical sep=1cm}, height=5.2cm, width=5.2cm]
		% c = 0.3 errors
		\nextgroupplot[xmode=linear, xmin=0, xmax=12, xlabel=,
					   ymode=log, log basis y={10}, ymin=1e-10, ymax=1e2, ylabel=Error,
					   axis background/.style={fill=gray!0}, 
					   legend pos=north west,
					   grid=both, grid style={line width=.1pt, draw=gray!10}, major grid style={line width=.2pt,draw=gray!50}]
			
			% source_iteration plots.
			\addplot+[mark=o, thick, solid, blue, mark options={blue, solid}] table [x=krylov_dim, y=dg_error_solver, col sep=comma] {deal_II_results/monoenergetic_benchmark/source_iteration/c3/a1_l1/monoenergetic_apost_si_2D_errors_p2_n4};
			\label{plots:monoenergetic_benchmark_disc_scat_ratio_dep_si_err}
			\addplot+[mark=*, thick, dashed, blue, mark options={blue, solid}] table [x=krylov_dim, y=solver_apost_estimate, col sep=comma] {deal_II_results/monoenergetic_benchmark/source_iteration/c3/a1_l1/monoenergetic_apost_si_2D_errors_p2_n4}; 
			\label{plots:monoenergetic_benchmark_disc_scat_ratio_dep_si_est}
			
			% modified_source_iteration plots.
			\addplot+[mark=square, thick, solid, red, mark options={red, solid}] table [x=krylov_dim, y=dg_error_solver, col sep=comma] {deal_II_results/monoenergetic_benchmark/modified_source_iteration/c3/a1_l1/monoenergetic_apost_msi_2D_errors_p2_n4};
			\label{plots:monoenergetic_benchmark_disc_scat_ratio_dep_msi_err}
			\addplot+[mark=square*, thick, dashed, red, mark options={red, solid}] table [x=krylov_dim, y=solver_apost_estimate, col sep=comma] {deal_II_results/monoenergetic_benchmark/modified_source_iteration/c3/a1_l1/monoenergetic_apost_msi_2D_errors_p2_n4}; 
			\label{plots:monoenergetic_benchmark_disc_scat_ratio_dep_msi_est}
			
			% GMRES plots.
			\addplot+[mark=diamond, thick, solid, black, mark options={black, solid}] table [x=krylov_dim, y=dg_error_solver, col sep=comma] {deal_II_results/monoenergetic_benchmark/gmres/c3/a1_l1/monoenergetic_apost_gmres_2D_errors_p2_n4};
			\label{plots:monoenergetic_benchmark_disc_scat_ratio_dep_gmres_err}
			\addplot+[mark=diamond*, thick, dashed, black, mark options={black, solid}] table [x=krylov_dim, y=solver_apost_estimate, col sep=comma] {deal_II_results/monoenergetic_benchmark/gmres/c3/a1_l1/monoenergetic_apost_gmres_2D_errors_p2_n4}; 
			\label{plots:monoenergetic_benchmark_disc_scat_ratio_dep_gmres_est}

			% Place coordinate for first figure.			
			\coordinate (top) at (rel axis cs:0,1); % coordinate at top of the first plot
		
		% c = 0.3 effectivities.
		\nextgroupplot[xmode=linear, xmin=0, xmax=12, xlabel=,
					   ymode=linear, log basis y={10}, ymin=0, ymax=10, ylabel=Effectivity,
					   axis background/.style={fill=gray!0}, 
					   legend pos=north west,
					   grid=both, grid style={line width=.1pt, draw=gray!10}, major grid style={line width=.2pt,draw=gray!50}]
			
			% source_iteration plots.
			\addplot+[mark=o, thick, solid, blue, mark options={blue, solid}] table [x=krylov_dim, y expr=\thisrowno{4}/\thisrowno{3}, col sep=comma] {deal_II_results/monoenergetic_benchmark/source_iteration/c3/a1_l1/monoenergetic_apost_si_2D_errors_p2_n4};
			
			% modified_source_iteration plots.
			\addplot+[mark=square, thick, solid, red, mark options={red, solid}] table [x=krylov_dim, y expr=\thisrowno{4}/\thisrowno{3}, col sep=comma] {deal_II_results/monoenergetic_benchmark/modified_source_iteration/c3/a1_l1/monoenergetic_apost_msi_2D_errors_p2_n4};
			
			% GMRES plots.
			\addplot+[mark=diamond, thick, solid, black, mark options={black, solid}] table [x=krylov_dim, y expr=\thisrowno{4}/\thisrowno{3}, col sep=comma] {deal_II_results/monoenergetic_benchmark/gmres/c3/a1_l1/monoenergetic_apost_gmres_2D_errors_p2_n4};
		
		% c = 0.5 errors.
		\nextgroupplot[xmode=linear, xmin=0, xmax=12, xlabel=,
					   ymode=log, log basis y={10}, ymin=1e-10, ymax=1e2, ylabel=Error,
					   axis background/.style={fill=gray!0}, 
					   legend pos=north west,
					   grid=both, grid style={line width=.1pt, draw=gray!10}, major grid style={line width=.2pt,draw=gray!50}]
			
			% source_iteration plots.
			\addplot+[mark=o, thick, solid, blue, mark options={blue, solid}] table [x=krylov_dim, y=dg_error_solver, col sep=comma] {deal_II_results/monoenergetic_benchmark/source_iteration/c5/a1_l1/monoenergetic_apost_si_2D_errors_p2_n4};
			\addplot+[mark=*, thick, dashed, blue, mark options={blue, solid}] table [x=krylov_dim, y=solver_apost_estimate, col sep=comma] {deal_II_results/monoenergetic_benchmark/source_iteration/c5/a1_l1/monoenergetic_apost_si_2D_errors_p2_n4}; 
			
			% modified_source_iteration plots.
			\addplot+[mark=square, thick, solid, red, mark options={red, solid}] table [x=krylov_dim, y=dg_error_solver, col sep=comma] {deal_II_results/monoenergetic_benchmark/modified_source_iteration/c5/a1_l1/monoenergetic_apost_msi_2D_errors_p2_n4};
			\addplot+[mark=square*, thick, dashed, red, mark options={red, solid}] table [x=krylov_dim, y=solver_apost_estimate, col sep=comma] {deal_II_results/monoenergetic_benchmark/modified_source_iteration/c5/a1_l1/monoenergetic_apost_msi_2D_errors_p2_n4}; 
			
			% GMRES plots.
			\addplot+[mark=diamond, thick, solid, black, mark options={black, solid}] table [x=krylov_dim, y=dg_error_solver, col sep=comma] {deal_II_results/monoenergetic_benchmark/gmres/c5/a1_l1/monoenergetic_apost_gmres_2D_errors_p2_n4};
			\addplot+[mark=diamond*, thick, dashed, black, mark options={black, solid}] table [x=krylov_dim, y=solver_apost_estimate, col sep=comma] {deal_II_results/monoenergetic_benchmark/gmres/c5/a1_l1/monoenergetic_apost_gmres_2D_errors_p2_n4}; 
		
		% c = 0.5 effectivities.
		\nextgroupplot[xmode=linear, xmin=0, xmax=12, xlabel=,
					   ymode=linear, log basis y={10}, ymin=0, ymax=10, ylabel=Effectivity,
					   axis background/.style={fill=gray!0}, 
					   legend pos=north west,
					   grid=both, grid style={line width=.1pt, draw=gray!10}, major grid style={line width=.2pt,draw=gray!50}]
			
			% source_iteration plots.
			\addplot+[mark=o, thick, solid, blue, mark options={blue, solid}] table [x=krylov_dim, y expr=\thisrowno{4}/\thisrowno{3}, col sep=comma] {deal_II_results/monoenergetic_benchmark/source_iteration/c5/a1_l1/monoenergetic_apost_si_2D_errors_p2_n4};
			
			% modified_source_iteration plots.
			\addplot+[mark=square, thick, solid, red, mark options={red, solid}] table [x=krylov_dim, y expr=\thisrowno{4}/\thisrowno{3}, col sep=comma] {deal_II_results/monoenergetic_benchmark/modified_source_iteration/c5/a1_l1/monoenergetic_apost_msi_2D_errors_p2_n4};
			
			% GMRES plots.
			\addplot+[mark=diamond, thick, solid, black, mark options={black, solid}] table [x=krylov_dim, y expr=\thisrowno{4}/\thisrowno{3}, col sep=comma] {deal_II_results/monoenergetic_benchmark/gmres/c5/a1_l1/monoenergetic_apost_gmres_2D_errors_p2_n4};
		
		% c = 0.7 errors.
		\nextgroupplot[xmode=linear, xmin=0, xmax=12, xlabel=Iteration,
					   ymode=log, log basis y={10}, ymin=1e-10, ymax=1e2, ylabel=Error,
					   axis background/.style={fill=gray!0}, 
					   legend pos=north west,
					   grid=both, grid style={line width=.1pt, draw=gray!10}, major grid style={line width=.2pt,draw=gray!50}]
			
			% source_iteration plots.
			\addplot+[mark=o, thick, solid, blue, mark options={blue, solid}] table [x=krylov_dim, y=dg_error_solver, col sep=comma] {deal_II_results/monoenergetic_benchmark/source_iteration/c7/a1_l1/monoenergetic_apost_si_2D_errors_p2_n4};
			\addplot+[mark=*, thick, dashed, blue, mark options={blue, solid}] table [x=krylov_dim, y=solver_apost_estimate, col sep=comma] {deal_II_results/monoenergetic_benchmark/source_iteration/c7/a1_l1/monoenergetic_apost_si_2D_errors_p2_n4}; 
			
			% modified_source_iteration plots.
			\addplot+[mark=square, thick, solid, red, mark options={red, solid}] table [x=krylov_dim, y=dg_error_solver, col sep=comma] {deal_II_results/monoenergetic_benchmark/modified_source_iteration/c7/a1_l1/monoenergetic_apost_msi_2D_errors_p2_n4};
			\addplot+[mark=square*, thick, dashed, red, mark options={red, solid}] table [x=krylov_dim, y=solver_apost_estimate, col sep=comma] {deal_II_results/monoenergetic_benchmark/modified_source_iteration/c7/a1_l1/monoenergetic_apost_msi_2D_errors_p2_n4}; 
			
			% GMRES plots.
			\addplot+[mark=diamond, thick, solid, black, mark options={black, solid}] table [x=krylov_dim, y=dg_error_solver, col sep=comma] {deal_II_results/monoenergetic_benchmark/gmres/c7/a1_l1/monoenergetic_apost_gmres_2D_errors_p2_n4};
			\addplot+[mark=diamond*, thick, dashed, black, mark options={black, solid}] table [x=krylov_dim, y=solver_apost_estimate, col sep=comma] {deal_II_results/monoenergetic_benchmark/gmres/c7/a1_l1/monoenergetic_apost_gmres_2D_errors_p2_n4}; 
		
		% c = 0.7 effectivities.
		\nextgroupplot[xmode=linear, xmin=0, xmax=12, xlabel=Iteration,
					   ymode=linear, log basis y={10}, ymin=0, ymax=10, ylabel=Effectivity,
					   axis background/.style={fill=gray!0}, 
					   legend pos=north west,
					   grid=both, grid style={line width=.1pt, draw=gray!10}, major grid style={line width=.2pt,draw=gray!50}]
			
			% source_iteration plots.
			\addplot+[mark=o, thick, solid, blue, mark options={blue, solid}] table [x=krylov_dim, y expr=\thisrowno{4}/\thisrowno{3}, col sep=comma] {deal_II_results/monoenergetic_benchmark/source_iteration/c7/a1_l1/monoenergetic_apost_si_2D_errors_p2_n4};
			
			% modified_source_iteration plots.
			\addplot+[mark=square, thick, solid, red, mark options={red, solid}] table [x=krylov_dim, y expr=\thisrowno{4}/\thisrowno{3}, col sep=comma] {deal_II_results/monoenergetic_benchmark/modified_source_iteration/c7/a1_l1/monoenergetic_apost_msi_2D_errors_p2_n4};
			
			% GMRES plots.
			\addplot+[mark=diamond, thick, solid, black, mark options={black, solid}] table [x=krylov_dim, y expr=\thisrowno{4}/\thisrowno{3}, col sep=comma] {deal_II_results/monoenergetic_benchmark/gmres/c7/a1_l1/monoenergetic_apost_gmres_2D_errors_p2_n4};
			
			% Place coordinate for next figure.			
            \coordinate (bot) at (rel axis cs:1,0); % coordinate at bottom of the last plot
	\end{groupplot}
	
    % legend
	\path (top|-current bounding box.north)--
    	coordinate(legendpos)
    	(bot|-current bounding box.north);
	\matrix[
    	matrix of nodes,
    	anchor=south,
    	draw,
    	inner sep=0.2em,
    	draw
  	]at([yshift=1ex]legendpos)
  	{
    	\ref{plots:monoenergetic_benchmark_disc_scat_ratio_dep_si_err}& SI (error) &[5pt]
    	\ref{plots:monoenergetic_benchmark_disc_scat_ratio_dep_msi_err}& Generalised SI (error) &[5pt]
    	\ref{plots:monoenergetic_benchmark_disc_scat_ratio_dep_gmres_err}& GMRES (error) \\
    	\ref{plots:monoenergetic_benchmark_disc_scat_ratio_dep_si_est}& SI (a post.) &[5pt]
    	\ref{plots:monoenergetic_benchmark_disc_scat_ratio_dep_msi_est}& Generalised SI (a post.) &[5pt]
    	\ref{plots:monoenergetic_benchmark_disc_scat_ratio_dep_gmres_est}& GMRES (a post.) \\
     };
\end{tikzpicture}
\caption{Dependence of the convergence of source iteration, generalised source iteration and right-preconditioned GMRES on the scattering ratio $c=\frac{\beta}{\alpha+\beta}$ for the model problem posed on the spatial domain $\spacedom=(0,10)^2$ and with $\sigma=10$. Left column: DGFEM-energy norm solver errors and \emph{a posteriori} error estimates. Right column: \emph{a posteriori} effectivities. Top row: $c=\frac{3}{10}$. Middle row: $c=\frac{5}{10}$. Bottom row: $c=\frac{7}{10}$.}
\label{fig:mono_benchmark_scatter_ratio_dep}
\end{figure}
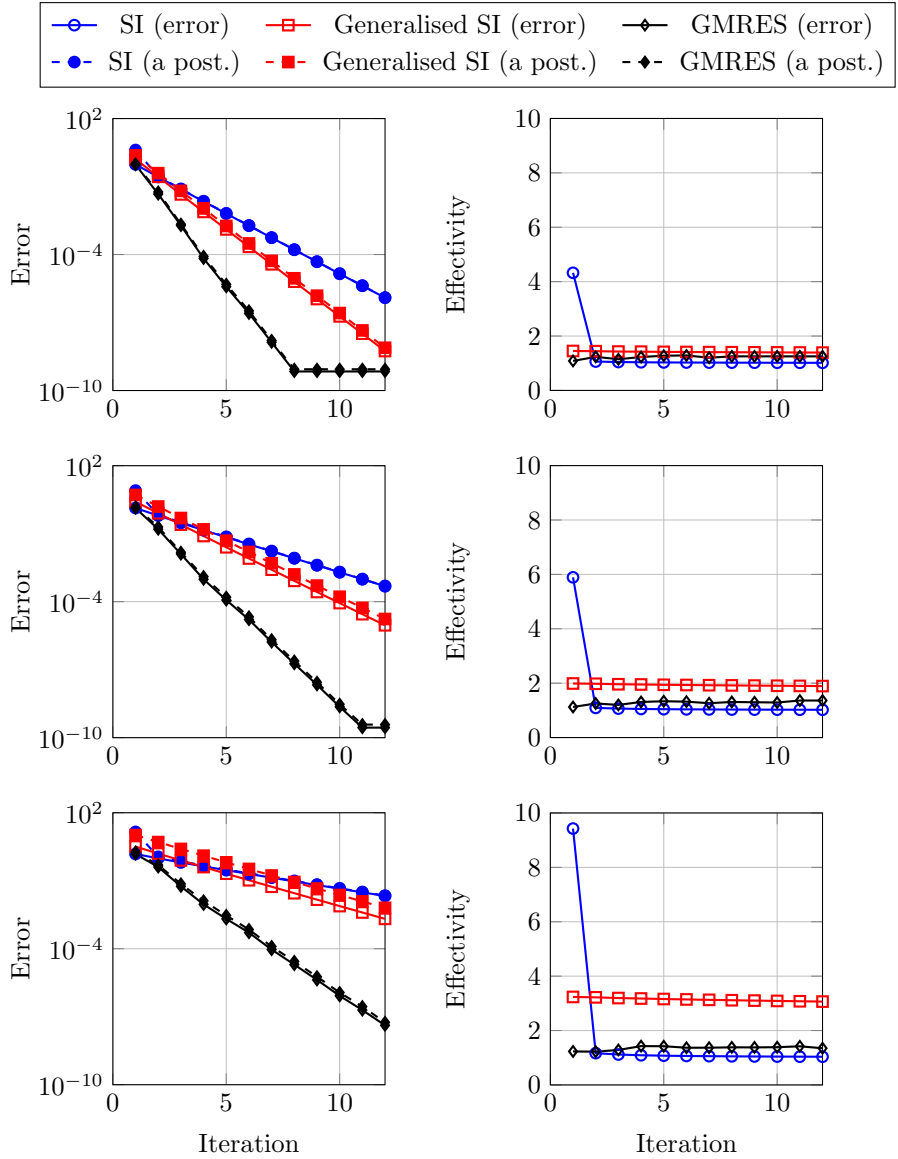

Figure \ref{fig:mono_benchmark_scatter_ratio_dep} shows how the rates of convergence of source iteration and its generalised variant depend on the scattering ratio $c=\frac{\beta}{\alpha+\beta}$. This behaviour is predicted in the statements of Theorems \ref{thm:dgfem_si_contraction_apost} and \ref{thm:dgfem_general_si_contraction_apost}, respectively. Similarly, the dependence on $c$ of the rate of convergence of GMRES is also observed. For each scattering ratio tested, GMRES is seen to converge most rapidly, followed by generalised source iteration and then by standard source iteration. The effectivities of the \emph{a posteriori} solver error estimators employed by all three methods remain close to 1 on almost every iteration, with the effectivity of the generalised source iteration error estimator deteriorating more rapidly as the scattering ratio increases.

\subsubsection{Dependence on optical thickness}

\begin{figure}[h!]
\centering
\begin{tikzpicture}
	\begin{groupplot}[group style={group size= 3 by 3, horizontal sep=1.9cm, vertical sep=1cm}, height=4cm, width=4cm]
		% (sigma,L) = (0.1,0.1)
		\nextgroupplot[xmode=linear, xmin=0, xmax=12, xlabel=,
					   ymode=log, log basis y={10}, ymin=1e-10, ymax=1e2, ylabel=Error,
					   axis background/.style={fill=gray!0}, 
					   legend pos=north west,
					   grid=both, grid style={line width=.1pt, draw=gray!10}, major grid style={line width=.2pt,draw=gray!50}]

			% source_iteration plots.
			\addplot+[mark=o, thick, solid, blue, mark options={blue, solid}] table [x=krylov_dim, y=dg_error_solver, col sep=comma] {deal_II_results/monoenergetic_benchmark/source_iteration/c7/a-1_l-1/monoenergetic_apost_si_2D_errors_p2_n4};
			\label{plots:monoenergetic_benchmark_disc_opt_thick_dep_si_err}
			\addplot+[mark=*, thick, dashed, blue, mark options={blue, solid}] table [x=krylov_dim, y=solver_apost_estimate, col sep=comma] {deal_II_results/monoenergetic_benchmark/source_iteration/c7/a-1_l-1/monoenergetic_apost_si_2D_errors_p2_n4}; 
			\label{plots:monoenergetic_benchmark_disc_opt_thick_dep_si_est}

			% modified_source_iteration plots.
			\addplot+[mark=square, thick, solid, red, mark options={red, solid}] table [x=krylov_dim, y=dg_error_solver, col sep=comma] {deal_II_results/monoenergetic_benchmark/modified_source_iteration/c7/a-1_l-1/monoenergetic_apost_msi_2D_errors_p2_n4};
			\label{plots:monoenergetic_benchmark_disc_opt_thick_dep_msi_err}
			\addplot+[mark=square*, thick, dashed, red, mark options={red, solid}] table [x=krylov_dim, y=solver_apost_estimate, col sep=comma] {deal_II_results/monoenergetic_benchmark/modified_source_iteration/c7/a-1_l-1/monoenergetic_apost_msi_2D_errors_p2_n4}; 
			\label{plots:monoenergetic_benchmark_disc_opt_thick_dep_msi_est}

			% GMRES plots.
			\addplot+[mark=diamond, thick, solid, black, mark options={black, solid}] table [x=krylov_dim, y=dg_error_solver, col sep=comma] {deal_II_results/monoenergetic_benchmark/gmres/c7/a-1_l-1/monoenergetic_apost_gmres_2D_errors_p2_n4};
			\label{plots:monoenergetic_benchmark_disc_opt_thick_dep_gmres_err}
			\addplot+[mark=diamond*, thick, dashed, black, mark options={black, solid}] table [x=krylov_dim, y=solver_apost_estimate, col sep=comma] {deal_II_results/monoenergetic_benchmark/gmres/c7/a-1_l-1/monoenergetic_apost_gmres_2D_errors_p2_n4}; 
			\label{plots:monoenergetic_benchmark_disc_opt_thick_dep_gmres_est}

			% Place coordinate for first figure.			
			\coordinate (top) at (rel axis cs:0,1); % coordinate at top of the first plot
			
		% (sigma,L) = (0.1,1)
		\nextgroupplot[xmode=linear, xmin=0, xmax=12, xlabel=,
					   ymode=log, log basis y={10}, ymin=1e-10, ymax=1e2, ylabel=,
					   axis background/.style={fill=gray!0}, 
					   legend pos=north west,
					   grid=both, grid style={line width=.1pt, draw=gray!10}, major grid style={line width=.2pt,draw=gray!50}]

			% source_iteration plots.
			\addplot+[mark=o, thick, solid, blue, mark options={blue, solid}] table [x=krylov_dim, y=dg_error_solver, col sep=comma] {deal_II_results/monoenergetic_benchmark/source_iteration/c7/a-1_l0/monoenergetic_apost_si_2D_errors_p2_n4};
			\addplot+[mark=*, thick, dashed, blue, mark options={blue, solid}] table [x=krylov_dim, y=solver_apost_estimate, col sep=comma] {deal_II_results/monoenergetic_benchmark/source_iteration/c7/a-1_l0/monoenergetic_apost_si_2D_errors_p2_n4}; 

			% modified_source_iteration plots.
			\addplot+[mark=square, thick, solid, red, mark options={red, solid}] table [x=krylov_dim, y=dg_error_solver, col sep=comma] {deal_II_results/monoenergetic_benchmark/modified_source_iteration/c7/a-1_l0/monoenergetic_apost_msi_2D_errors_p2_n4};
			\addplot+[mark=square*, thick, dashed, red, mark options={red, solid}] table [x=krylov_dim, y=solver_apost_estimate, col sep=comma] {deal_II_results/monoenergetic_benchmark/modified_source_iteration/c7/a-1_l0/monoenergetic_apost_msi_2D_errors_p2_n4}; 

			% GMRES plots.
			\addplot+[mark=diamond, thick, solid, black, mark options={black, solid}] table [x=krylov_dim, y=dg_error_solver, col sep=comma] {deal_II_results/monoenergetic_benchmark/gmres/c7/a-1_l0/monoenergetic_apost_gmres_2D_errors_p2_n4};
			\addplot+[mark=diamond*, thick, dashed, black, mark options={black, solid}] table [x=krylov_dim, y=solver_apost_estimate, col sep=comma] {deal_II_results/monoenergetic_benchmark/gmres/c7/a-1_l0/monoenergetic_apost_gmres_2D_errors_p2_n4}; 
					   
		% (sigma,L) = (0.1,10)
		\nextgroupplot[xmode=linear, xmin=0, xmax=12, xlabel=,
					   ymode=log, log basis y={10}, ymin=1e-10, ymax=1e2, ylabel=,
					   axis background/.style={fill=gray!0}, 
					   legend pos=north west,
					   grid=both, grid style={line width=.1pt, draw=gray!10}, major grid style={line width=.2pt,draw=gray!50}]

			% source_iteration plots.
			\addplot+[mark=o, thick, solid, blue, mark options={blue, solid}] table [x=krylov_dim, y=dg_error_solver, col sep=comma] {deal_II_results/monoenergetic_benchmark/source_iteration/c7/a-1_l1/monoenergetic_apost_si_2D_errors_p2_n4};
			\addplot+[mark=*, thick, dashed, blue, mark options={blue, solid}] table [x=krylov_dim, y=solver_apost_estimate, col sep=comma] {deal_II_results/monoenergetic_benchmark/source_iteration/c7/a-1_l1/monoenergetic_apost_si_2D_errors_p2_n4}; 

			% modified_source_iteration plots.
			\addplot+[mark=square, thick, solid, red, mark options={red, solid}] table [x=krylov_dim, y=dg_error_solver, col sep=comma] {deal_II_results/monoenergetic_benchmark/modified_source_iteration/c7/a-1_l1/monoenergetic_apost_msi_2D_errors_p2_n4};
			\addplot+[mark=square*, thick, dashed, red, mark options={red, solid}] table [x=krylov_dim, y=solver_apost_estimate, col sep=comma] {deal_II_results/monoenergetic_benchmark/modified_source_iteration/c7/a-1_l1/monoenergetic_apost_msi_2D_errors_p2_n4}; 

			% GMRES plots.
			\addplot+[mark=diamond, thick, solid, black, mark options={black, solid}] table [x=krylov_dim, y=dg_error_solver, col sep=comma] {deal_II_results/monoenergetic_benchmark/gmres/c7/a-1_l1/monoenergetic_apost_gmres_2D_errors_p2_n4};
			\addplot+[mark=diamond*, thick, dashed, black, mark options={black, solid}] table [x=krylov_dim, y=solver_apost_estimate, col sep=comma] {deal_II_results/monoenergetic_benchmark/gmres/c7/a-1_l1/monoenergetic_apost_gmres_2D_errors_p2_n4}; 
			
		% (sigma,L) = (1,0.1)
		\nextgroupplot[xmode=linear, xmin=0, xmax=12, xlabel=,
					   ymode=log, log basis y={10}, ymin=1e-10, ymax=1e2, ylabel=Error,
					   axis background/.style={fill=gray!0}, 
					   legend pos=north west,
					   grid=both, grid style={line width=.1pt, draw=gray!10}, major grid style={line width=.2pt,draw=gray!50}]

			% source_iteration plots.
			\addplot+[mark=o, thick, solid, blue, mark options={blue, solid}] table [x=krylov_dim, y=dg_error_solver, col sep=comma] {deal_II_results/monoenergetic_benchmark/source_iteration/c7/a0_l-1/monoenergetic_apost_si_2D_errors_p2_n4};
			\addplot+[mark=*, thick, dashed, blue, mark options={blue, solid}] table [x=krylov_dim, y=solver_apost_estimate, col sep=comma] {deal_II_results/monoenergetic_benchmark/source_iteration/c7/a0_l-1/monoenergetic_apost_si_2D_errors_p2_n4}; 

			% modified_source_iteration plots.
			\addplot+[mark=square, thick, solid, red, mark options={red, solid}] table [x=krylov_dim, y=dg_error_solver, col sep=comma] {deal_II_results/monoenergetic_benchmark/modified_source_iteration/c7/a0_l-1/monoenergetic_apost_msi_2D_errors_p2_n4};
			\addplot+[mark=square*, thick, dashed, red, mark options={red, solid}] table [x=krylov_dim, y=solver_apost_estimate, col sep=comma] {deal_II_results/monoenergetic_benchmark/modified_source_iteration/c7/a0_l-1/monoenergetic_apost_msi_2D_errors_p2_n4}; 

			% GMRES plots.
			\addplot+[mark=diamond, thick, solid, black, mark options={black, solid}] table [x=krylov_dim, y=dg_error_solver, col sep=comma] {deal_II_results/monoenergetic_benchmark/gmres/c7/a0_l-1/monoenergetic_apost_gmres_2D_errors_p2_n4};
			\addplot+[mark=diamond*, thick, dashed, black, mark options={black, solid}] table [x=krylov_dim, y=solver_apost_estimate, col sep=comma] {deal_II_results/monoenergetic_benchmark/gmres/c7/a0_l-1/monoenergetic_apost_gmres_2D_errors_p2_n4}; 
			
		% (sigma,L) = (1,1)
		\nextgroupplot[xmode=linear, xmin=0, xmax=12, xlabel=,
					   ymode=log, log basis y={10}, ymin=1e-10, ymax=1e2, ylabel=,
					   axis background/.style={fill=gray!0}, 
					   legend pos=north west,
					   grid=both, grid style={line width=.1pt, draw=gray!10}, major grid style={line width=.2pt,draw=gray!50}]

			% source_iteration plots.
			\addplot+[mark=o, thick, solid, blue, mark options={blue, solid}] table [x=krylov_dim, y=dg_error_solver, col sep=comma] {deal_II_results/monoenergetic_benchmark/source_iteration/c7/a0_l0/monoenergetic_apost_si_2D_errors_p2_n4};
			\addplot+[mark=*, thick, dashed, blue, mark options={blue, solid}] table [x=krylov_dim, y=solver_apost_estimate, col sep=comma] {deal_II_results/monoenergetic_benchmark/source_iteration/c7/a0_l0/monoenergetic_apost_si_2D_errors_p2_n4}; 

			% modified_source_iteration plots.
			\addplot+[mark=square, thick, solid, red, mark options={red, solid}] table [x=krylov_dim, y=dg_error_solver, col sep=comma] {deal_II_results/monoenergetic_benchmark/modified_source_iteration/c7/a0_l0/monoenergetic_apost_msi_2D_errors_p2_n4};
			\addplot+[mark=square*, thick, dashed, red, mark options={red, solid}] table [x=krylov_dim, y=solver_apost_estimate, col sep=comma] {deal_II_results/monoenergetic_benchmark/modified_source_iteration/c7/a0_l0/monoenergetic_apost_msi_2D_errors_p2_n4}; 

			% GMRES plots.
			\addplot+[mark=diamond, thick, solid, black, mark options={black, solid}] table [x=krylov_dim, y=dg_error_solver, col sep=comma] {deal_II_results/monoenergetic_benchmark/gmres/c7/a0_l0/monoenergetic_apost_gmres_2D_errors_p2_n4};
			\addplot+[mark=diamond*, thick, dashed, black, mark options={black, solid}] table [x=krylov_dim, y=solver_apost_estimate, col sep=comma] {deal_II_results/monoenergetic_benchmark/gmres/c7/a0_l0/monoenergetic_apost_gmres_2D_errors_p2_n4}; 
					   
		% (sigma,L) = (1,10)
		\nextgroupplot[xmode=linear, xmin=0, xmax=12, xlabel=,
					   ymode=log, log basis y={10}, ymin=1e-10, ymax=1e2, ylabel=,
					   axis background/.style={fill=gray!0}, 
					   legend pos=north west,
					   grid=both, grid style={line width=.1pt, draw=gray!10}, major grid style={line width=.2pt,draw=gray!50}]

			% source_iteration plots.
			\addplot+[mark=o, thick, solid, blue, mark options={blue, solid}] table [x=krylov_dim, y=dg_error_solver, col sep=comma] {deal_II_results/monoenergetic_benchmark/source_iteration/c7/a0_l1/monoenergetic_apost_si_2D_errors_p2_n4};
			\addplot+[mark=*, thick, dashed, blue, mark options={blue, solid}] table [x=krylov_dim, y=solver_apost_estimate, col sep=comma] {deal_II_results/monoenergetic_benchmark/source_iteration/c7/a0_l1/monoenergetic_apost_si_2D_errors_p2_n4}; 

			% modified_source_iteration plots.
			\addplot+[mark=square, thick, solid, red, mark options={red, solid}] table [x=krylov_dim, y=dg_error_solver, col sep=comma] {deal_II_results/monoenergetic_benchmark/modified_source_iteration/c7/a0_l1/monoenergetic_apost_msi_2D_errors_p2_n4};
			\addplot+[mark=square*, thick, dashed, red, mark options={red, solid}] table [x=krylov_dim, y=solver_apost_estimate, col sep=comma] {deal_II_results/monoenergetic_benchmark/modified_source_iteration/c7/a0_l1/monoenergetic_apost_msi_2D_errors_p2_n4}; 

			% GMRES plots.
			\addplot+[mark=diamond, thick, solid, black, mark options={black, solid}] table [x=krylov_dim, y=dg_error_solver, col sep=comma] {deal_II_results/monoenergetic_benchmark/gmres/c7/a0_l1/monoenergetic_apost_gmres_2D_errors_p2_n4};
			\addplot+[mark=diamond*, thick, dashed, black, mark options={black, solid}] table [x=krylov_dim, y=solver_apost_estimate, col sep=comma] {deal_II_results/monoenergetic_benchmark/gmres/c7/a0_l1/monoenergetic_apost_gmres_2D_errors_p2_n4}; 
			
		% (sigma,L) = (10,0.1)
		\nextgroupplot[xmode=linear, xmin=0, xmax=12, xlabel=Iteration,
					   ymode=log, log basis y={10}, ymin=1e-10, ymax=1e2, ylabel=Error,
					   axis background/.style={fill=gray!0}, 
					   legend pos=north west,
					   grid=both, grid style={line width=.1pt, draw=gray!10}, major grid style={line width=.2pt,draw=gray!50}]

			% source_iteration plots.
			\addplot+[mark=o, thick, solid, blue, mark options={blue, solid}] table [x=krylov_dim, y=dg_error_solver, col sep=comma] {deal_II_results/monoenergetic_benchmark/source_iteration/c7/a1_l-1/monoenergetic_apost_si_2D_errors_p2_n4};
			\addplot+[mark=*, thick, dashed, blue, mark options={blue, solid}] table [x=krylov_dim, y=solver_apost_estimate, col sep=comma] {deal_II_results/monoenergetic_benchmark/source_iteration/c7/a1_l-1/monoenergetic_apost_si_2D_errors_p2_n4}; 

			% modified_source_iteration plots.
			\addplot+[mark=square, thick, solid, red, mark options={red, solid}] table [x=krylov_dim, y=dg_error_solver, col sep=comma] {deal_II_results/monoenergetic_benchmark/modified_source_iteration/c7/a1_l-1/monoenergetic_apost_msi_2D_errors_p2_n4};
			\addplot+[mark=square*, thick, dashed, red, mark options={red, solid}] table [x=krylov_dim, y=solver_apost_estimate, col sep=comma] {deal_II_results/monoenergetic_benchmark/modified_source_iteration/c7/a1_l-1/monoenergetic_apost_msi_2D_errors_p2_n4}; 

			% GMRES plots.
			\addplot+[mark=diamond, thick, solid, black, mark options={black, solid}] table [x=krylov_dim, y=dg_error_solver, col sep=comma] {deal_II_results/monoenergetic_benchmark/gmres/c7/a1_l-1/monoenergetic_apost_gmres_2D_errors_p2_n4};
			\addplot+[mark=diamond*, thick, dashed, black, mark options={black, solid}] table [x=krylov_dim, y=solver_apost_estimate, col sep=comma] {deal_II_results/monoenergetic_benchmark/gmres/c7/a1_l-1/monoenergetic_apost_gmres_2D_errors_p2_n4}; 
			
		% (sigma,L) = (10,1)
		\nextgroupplot[xmode=linear, xmin=0, xmax=12, xlabel=Iteration,
					   ymode=log, log basis y={10}, ymin=1e-10, ymax=1e2, ylabel=,
					   axis background/.style={fill=gray!0}, 
					   legend pos=north west,
					   grid=both, grid style={line width=.1pt, draw=gray!10}, major grid style={line width=.2pt,draw=gray!50}]

			% source_iteration plots.
			\addplot+[mark=o, thick, solid, blue, mark options={blue, solid}] table [x=krylov_dim, y=dg_error_solver, col sep=comma] {deal_II_results/monoenergetic_benchmark/source_iteration/c7/a1_l0/monoenergetic_apost_si_2D_errors_p2_n4};
			\addplot+[mark=*, thick, dashed, blue, mark options={blue, solid}] table [x=krylov_dim, y=solver_apost_estimate, col sep=comma] {deal_II_results/monoenergetic_benchmark/source_iteration/c7/a1_l0/monoenergetic_apost_si_2D_errors_p2_n4}; 

			% modified_source_iteration plots.
			\addplot+[mark=square, thick, solid, red, mark options={red, solid}] table [x=krylov_dim, y=dg_error_solver, col sep=comma] {deal_II_results/monoenergetic_benchmark/modified_source_iteration/c7/a1_l0/monoenergetic_apost_msi_2D_errors_p2_n4};
			\addplot+[mark=square*, thick, dashed, red, mark options={red, solid}] table [x=krylov_dim, y=solver_apost_estimate, col sep=comma] {deal_II_results/monoenergetic_benchmark/modified_source_iteration/c7/a1_l0/monoenergetic_apost_msi_2D_errors_p2_n4}; 

			% GMRES plots.
			\addplot+[mark=diamond, thick, solid, black, mark options={black, solid}] table [x=krylov_dim, y=dg_error_solver, col sep=comma] {deal_II_results/monoenergetic_benchmark/gmres/c7/a1_l0/monoenergetic_apost_gmres_2D_errors_p2_n4};
			\addplot+[mark=diamond*, thick, dashed, black, mark options={black, solid}] table [x=krylov_dim, y=solver_apost_estimate, col sep=comma] {deal_II_results/monoenergetic_benchmark/gmres/c7/a1_l0/monoenergetic_apost_gmres_2D_errors_p2_n4}; 
					   
		% (sigma,L) = (10,10)
		\nextgroupplot[xmode=linear, xmin=0, xmax=12, xlabel=Iteration,
					   ymode=log, log basis y={10}, ymin=1e-10, ymax=1e2, ylabel=,
					   axis background/.style={fill=gray!0}, 
					   legend pos=north west,
					   grid=both, grid style={line width=.1pt, draw=gray!10}, major grid style={line width=.2pt,draw=gray!50}]

			% source_iteration plots.
			\addplot+[mark=o, thick, solid, blue, mark options={blue, solid}] table [x=krylov_dim, y=dg_error_solver, col sep=comma] {deal_II_results/monoenergetic_benchmark/source_iteration/c7/a1_l1/monoenergetic_apost_si_2D_errors_p2_n4};
			\addplot+[mark=*, thick, dashed, blue, mark options={blue, solid}] table [x=krylov_dim, y=solver_apost_estimate, col sep=comma] {deal_II_results/monoenergetic_benchmark/source_iteration/c7/a1_l1/monoenergetic_apost_si_2D_errors_p2_n4}; 

			% modified_source_iteration plots.
			\addplot+[mark=square, thick, solid, red, mark options={red, solid}] table [x=krylov_dim, y=dg_error_solver, col sep=comma] {deal_II_results/monoenergetic_benchmark/modified_source_iteration/c7/a1_l1/monoenergetic_apost_msi_2D_errors_p2_n4};
			\addplot+[mark=square*, thick, dashed, red, mark options={red, solid}] table [x=krylov_dim, y=solver_apost_estimate, col sep=comma] {deal_II_results/monoenergetic_benchmark/modified_source_iteration/c7/a1_l1/monoenergetic_apost_msi_2D_errors_p2_n4}; 

			% GMRES plots.
			\addplot+[mark=diamond, thick, solid, black, mark options={black, solid}] table [x=krylov_dim, y=dg_error_solver, col sep=comma] {deal_II_results/monoenergetic_benchmark/gmres/c7/a1_l1/monoenergetic_apost_gmres_2D_errors_p2_n4};
			\addplot+[mark=diamond*, thick, dashed, black, mark options={black, solid}] table [x=krylov_dim, y=solver_apost_estimate, col sep=comma] {deal_II_results/monoenergetic_benchmark/gmres/c7/a1_l1/monoenergetic_apost_gmres_2D_errors_p2_n4}; 
			
			% Place coordinate for next figure.			
            \coordinate (bot) at (rel axis cs:1,0); % coordinate at bottom of the last plot
	\end{groupplot}

    % legend
	\path (top|-current bounding box.north)--
    	coordinate(legendpos)
    	(bot|-current bounding box.north);
	\matrix[
    	matrix of nodes,
    	anchor=south,
    	draw,
    	inner sep=0.2em,
    	draw
  	]at([yshift=1ex]legendpos)
  	{
    	\ref{plots:monoenergetic_benchmark_disc_opt_thick_dep_si_err}& SI (error) &[5pt]
    	\ref{plots:monoenergetic_benchmark_disc_opt_thick_dep_msi_err}& Generalised SI (error) &[5pt]
    	\ref{plots:monoenergetic_benchmark_disc_opt_thick_dep_gmres_err}& GMRES (error) \\
    	\ref{plots:monoenergetic_benchmark_disc_opt_thick_dep_si_est}& SI (a post.) &[5pt]
    	\ref{plots:monoenergetic_benchmark_disc_opt_thick_dep_msi_est}& Generalised SI (a post.) &[5pt]
    	\ref{plots:monoenergetic_benchmark_disc_opt_thick_dep_gmres_est}& GMRES (a post.) \\
    }; 
\end{tikzpicture}
\caption{Dependence of the convergence of source iteration, generalised source iteration and right-preconditioned GMRES on the domain size $L\in\{\frac{1}{10},1,10\}$ and macroscopic total cross-section $\sigma\in\{\frac{1}{10},1,10\}$ for the model problem posed on the spatial domain $\spacedom=(0,L)^2$ and employing spatial meshes $\spacemesh$ with $|\spacemesh|=256$. Left column: $L=\frac{1}{10}$. Middle column: $L=1$. Right column: $L=10$. Top row: $\sigma=\frac{1}{10}$. Middle row: $\sigma=1$. Bottom row: $\sigma=10$.}
\label{fig:mono_benchmark_optical_thickness_dep_errs}
\end{figure}
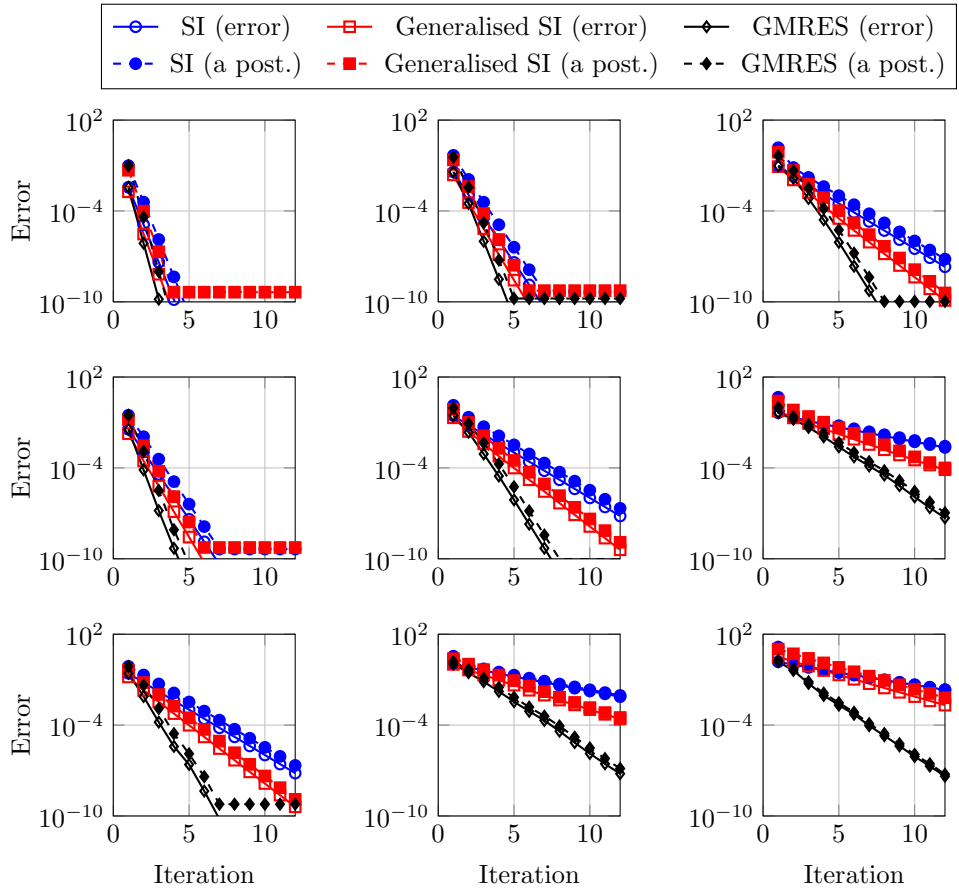

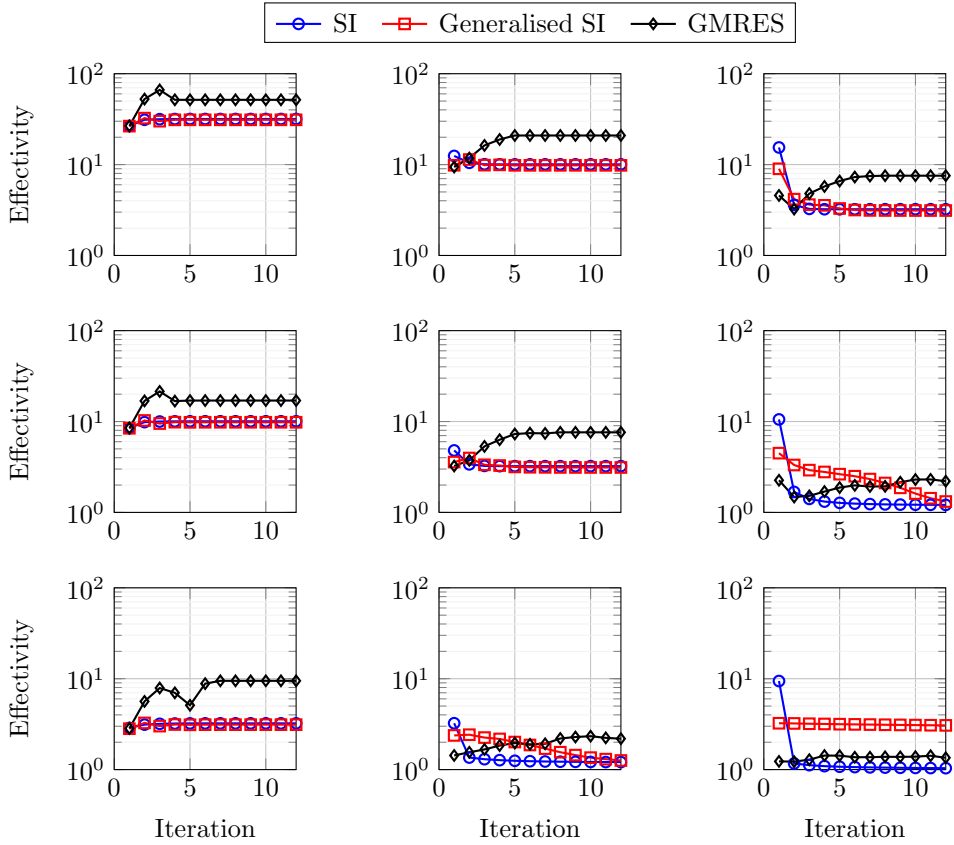
\begin{figure}[h!]
\centering
\begin{tikzpicture}
	\begin{groupplot}[group style={group size= 3 by 3, horizontal sep=1.9cm, vertical sep=1cm}, height=4cm, width=4cm]
		% (sigma,L) = (0.1,0.1)
		\nextgroupplot[xmode=linear, xmin=0, xmax=12, xlabel=,
					   ymode=log, log basis y={10}, ymin=1e0, ymax=1e2, ylabel=Effectivity,
					   axis background/.style={fill=gray!0}, 
					   legend pos=north west,
					   grid=both, grid style={line width=.1pt, draw=gray!10}, major grid style={line width=.2pt,draw=gray!50}]

			% source_iteration plots.
			\addplot+[mark=o, thick, solid, blue, mark options={blue, solid}] table [x=krylov_dim, y expr=\thisrowno{4}/\thisrowno{3}, col sep=comma] {deal_II_results/monoenergetic_benchmark/source_iteration/c7/a-1_l-1/monoenergetic_apost_si_2D_errors_p2_n4};
			\label{plots:monoenergetic_benchmark_disc_opt_thick_dep_si_eff}

			% modified_source_iteration plots.
			\addplot+[mark=square, thick, solid, red, mark options={red, solid}] table [x=krylov_dim, y expr=\thisrowno{4}/\thisrowno{3}, col sep=comma] {deal_II_results/monoenergetic_benchmark/modified_source_iteration/c7/a-1_l-1/monoenergetic_apost_msi_2D_errors_p2_n4};
			\label{plots:monoenergetic_benchmark_disc_opt_thick_dep_msi_eff}

			% GMRES plots.
			\addplot+[mark=diamond, thick, solid, black, mark options={black, solid}] table [x=krylov_dim, y expr=\thisrowno{4}/\thisrowno{3}, col sep=comma] {deal_II_results/monoenergetic_benchmark/gmres/c7/a-1_l-1/monoenergetic_apost_gmres_2D_errors_p2_n4};
			\label{plots:monoenergetic_benchmark_disc_opt_thick_dep_gmres_eff}

			% Place coordinate for first figure.			
			\coordinate (top) at (rel axis cs:0,1); % coordinate at top of the first plot
			
		% (sigma,L) = (0.1,1)
		\nextgroupplot[xmode=linear, xmin=0, xmax=12, xlabel=,
					   ymode=log, log basis y={10}, ymin=1e0, ymax=1e2, ylabel=,
					   axis background/.style={fill=gray!0}, 
					   legend pos=north west,
					   grid=both, grid style={line width=.1pt, draw=gray!10}, major grid style={line width=.2pt,draw=gray!50}]

			% source_iteration plots.
			\addplot+[mark=o, thick, solid, blue, mark options={blue, solid}] table [x=krylov_dim, y expr=\thisrowno{4}/\thisrowno{3}, col sep=comma] {deal_II_results/monoenergetic_benchmark/source_iteration/c7/a-1_l0/monoenergetic_apost_si_2D_errors_p2_n4};

			% modified_source_iteration plots.
			\addplot+[mark=square, thick, solid, red, mark options={red, solid}] table [x=krylov_dim, y expr=\thisrowno{4}/\thisrowno{3}, col sep=comma] {deal_II_results/monoenergetic_benchmark/modified_source_iteration/c7/a-1_l0/monoenergetic_apost_msi_2D_errors_p2_n4};

			% GMRES plots.
			\addplot+[mark=diamond, thick, solid, black, mark options={black, solid}] table [x=krylov_dim, y expr=\thisrowno{4}/\thisrowno{3}, col sep=comma] {deal_II_results/monoenergetic_benchmark/gmres/c7/a-1_l0/monoenergetic_apost_gmres_2D_errors_p2_n4};
					   
		% (sigma,L) = (0.1,10)
		\nextgroupplot[xmode=linear, xmin=0, xmax=12, xlabel=,
					   ymode=log, log basis y={10}, ymin=1e0, ymax=1e2, ylabel=,
					   axis background/.style={fill=gray!0}, 
					   legend pos=north west,
					   grid=both, grid style={line width=.1pt, draw=gray!10}, major grid style={line width=.2pt,draw=gray!50}]

			% source_iteration plots.
			\addplot+[mark=o, thick, solid, blue, mark options={blue, solid}] table [x=krylov_dim, y expr=\thisrowno{4}/\thisrowno{3}, col sep=comma] {deal_II_results/monoenergetic_benchmark/source_iteration/c7/a-1_l1/monoenergetic_apost_si_2D_errors_p2_n4};

			% modified_source_iteration plots.
			\addplot+[mark=square, thick, solid, red, mark options={red, solid}] table [x=krylov_dim, y expr=\thisrowno{4}/\thisrowno{3}, col sep=comma] {deal_II_results/monoenergetic_benchmark/modified_source_iteration/c7/a-1_l1/monoenergetic_apost_msi_2D_errors_p2_n4}; 

			% GMRES plots.
			\addplot+[mark=diamond, thick, solid, black, mark options={black, solid}] table [x=krylov_dim, y expr=\thisrowno{4}/\thisrowno{3}, col sep=comma] {deal_II_results/monoenergetic_benchmark/gmres/c7/a-1_l1/monoenergetic_apost_gmres_2D_errors_p2_n4};
			
		% (sigma,L) = (1,0.1)
		\nextgroupplot[xmode=linear, xmin=0, xmax=12, xlabel=,
					   ymode=log, log basis y={10}, ymin=1e0, ymax=1e2, ylabel=Effectivity,
					   axis background/.style={fill=gray!0}, 
					   legend pos=north west,
					   grid=both, grid style={line width=.1pt, draw=gray!10}, major grid style={line width=.2pt,draw=gray!50}]

			% source_iteration plots.
			\addplot+[mark=o, thick, solid, blue, mark options={blue, solid}] table [x=krylov_dim, y expr=\thisrowno{4}/\thisrowno{3}, col sep=comma] {deal_II_results/monoenergetic_benchmark/source_iteration/c7/a0_l-1/monoenergetic_apost_si_2D_errors_p2_n4};

			% modified_source_iteration plots.
			\addplot+[mark=square, thick, solid, red, mark options={red, solid}] table [x=krylov_dim, y expr=\thisrowno{4}/\thisrowno{3}, col sep=comma] {deal_II_results/monoenergetic_benchmark/modified_source_iteration/c7/a0_l-1/monoenergetic_apost_msi_2D_errors_p2_n4};

			% GMRES plots.
			\addplot+[mark=diamond, thick, solid, black, mark options={black, solid}] table [x=krylov_dim, y expr=\thisrowno{4}/\thisrowno{3}, col sep=comma] {deal_II_results/monoenergetic_benchmark/gmres/c7/a0_l-1/monoenergetic_apost_gmres_2D_errors_p2_n4};
			
		% (sigma,L) = (1,1)
		\nextgroupplot[xmode=linear, xmin=0, xmax=12, xlabel=,
					   ymode=log, log basis y={10}, ymin=1e0, ymax=1e2, ylabel=,
					   axis background/.style={fill=gray!0}, 
					   legend pos=north west,
					   grid=both, grid style={line width=.1pt, draw=gray!10}, major grid style={line width=.2pt,draw=gray!50}]

			% source_iteration plots.
			\addplot+[mark=o, thick, solid, blue, mark options={blue, solid}] table [x=krylov_dim, y expr=\thisrowno{4}/\thisrowno{3}, col sep=comma] {deal_II_results/monoenergetic_benchmark/source_iteration/c7/a0_l0/monoenergetic_apost_si_2D_errors_p2_n4};
			
			% modified_source_iteration plots.
			\addplot+[mark=square, thick, solid, red, mark options={red, solid}] table [x=krylov_dim, y expr=\thisrowno{4}/\thisrowno{3}, col sep=comma] {deal_II_results/monoenergetic_benchmark/modified_source_iteration/c7/a0_l0/monoenergetic_apost_msi_2D_errors_p2_n4};

			% GMRES plots.
			\addplot+[mark=diamond, thick, solid, black, mark options={black, solid}] table [x=krylov_dim, y expr=\thisrowno{4}/\thisrowno{3}, col sep=comma] {deal_II_results/monoenergetic_benchmark/gmres/c7/a0_l0/monoenergetic_apost_gmres_2D_errors_p2_n4};
					   
		% (sigma,L) = (1,10)
		\nextgroupplot[xmode=linear, xmin=0, xmax=12, xlabel=,
					   ymode=log, log basis y={10}, ymin=1e0, ymax=1e2, ylabel=,
					   axis background/.style={fill=gray!0}, 
					   legend pos=north west,
					   grid=both, grid style={line width=.1pt, draw=gray!10}, major grid style={line width=.2pt,draw=gray!50}]

			% source_iteration plots.
			\addplot+[mark=o, thick, solid, blue, mark options={blue, solid}] table [x=krylov_dim, y expr=\thisrowno{4}/\thisrowno{3}, col sep=comma] {deal_II_results/monoenergetic_benchmark/source_iteration/c7/a0_l1/monoenergetic_apost_si_2D_errors_p2_n4};

			% modified_source_iteration plots.
			\addplot+[mark=square, thick, solid, red, mark options={red, solid}] table [x=krylov_dim, y expr=\thisrowno{4}/\thisrowno{3}, col sep=comma] {deal_II_results/monoenergetic_benchmark/modified_source_iteration/c7/a0_l1/monoenergetic_apost_msi_2D_errors_p2_n4};

			% GMRES plots.
			\addplot+[mark=diamond, thick, solid, black, mark options={black, solid}] table [x=krylov_dim, y expr=\thisrowno{4}/\thisrowno{3}, col sep=comma] {deal_II_results/monoenergetic_benchmark/gmres/c7/a0_l1/monoenergetic_apost_gmres_2D_errors_p2_n4};
			
		% (sigma,L) = (10,0.1)
		\nextgroupplot[xmode=linear, xmin=0, xmax=12, xlabel=Iteration,
					   ymode=log, log basis y={10}, ymin=1e0, ymax=1e2, ylabel=Effectivity,
					   axis background/.style={fill=gray!0}, 
					   legend pos=north west,
					   grid=both, grid style={line width=.1pt, draw=gray!10}, major grid style={line width=.2pt,draw=gray!50}]

			% source_iteration plots.
			\addplot+[mark=o, thick, solid, blue, mark options={blue, solid}] table [x=krylov_dim, y expr=\thisrowno{4}/\thisrowno{3}, col sep=comma] {deal_II_results/monoenergetic_benchmark/source_iteration/c7/a1_l-1/monoenergetic_apost_si_2D_errors_p2_n4};

			% modified_source_iteration plots.
			\addplot+[mark=square, thick, solid, red, mark options={red, solid}] table [x=krylov_dim, y expr=\thisrowno{4}/\thisrowno{3}, col sep=comma] {deal_II_results/monoenergetic_benchmark/modified_source_iteration/c7/a1_l-1/monoenergetic_apost_msi_2D_errors_p2_n4};

			% GMRES plots.
			\addplot+[mark=diamond, thick, solid, black, mark options={black, solid}] table [x=krylov_dim, y expr=\thisrowno{4}/\thisrowno{3}, col sep=comma] {deal_II_results/monoenergetic_benchmark/gmres/c7/a1_l-1/monoenergetic_apost_gmres_2D_errors_p2_n4};
			
		% (sigma,L) = (10,1)
		\nextgroupplot[xmode=linear, xmin=0, xmax=12, xlabel=Iteration,
					   ymode=log, log basis y={10}, ymin=1e0, ymax=1e2, ylabel=,
					   axis background/.style={fill=gray!0}, 
					   legend pos=north west,
					   grid=both, grid style={line width=.1pt, draw=gray!10}, major grid style={line width=.2pt,draw=gray!50}]

			% source_iteration plots.
			\addplot+[mark=o, thick, solid, blue, mark options={blue, solid}] table [x=krylov_dim, y expr=\thisrowno{4}/\thisrowno{3}, col sep=comma] {deal_II_results/monoenergetic_benchmark/source_iteration/c7/a1_l0/monoenergetic_apost_si_2D_errors_p2_n4};

			% modified_source_iteration plots.
			\addplot+[mark=square, thick, solid, red, mark options={red, solid}] table [x=krylov_dim, y expr=\thisrowno{4}/\thisrowno{3}, col sep=comma] {deal_II_results/monoenergetic_benchmark/modified_source_iteration/c7/a1_l0/monoenergetic_apost_msi_2D_errors_p2_n4};

			% GMRES plots.
			\addplot+[mark=diamond, thick, solid, black, mark options={black, solid}] table [x=krylov_dim, y expr=\thisrowno{4}/\thisrowno{3}, col sep=comma] {deal_II_results/monoenergetic_benchmark/gmres/c7/a1_l0/monoenergetic_apost_gmres_2D_errors_p2_n4};
					   
		% (sigma,L) = (10,10)
		\nextgroupplot[xmode=linear, xmin=0, xmax=12, xlabel=Iteration,
					   ymode=log, log basis y={10}, ymin=1e0, ymax=1e2, ylabel=,
					   axis background/.style={fill=gray!0}, 
					   legend pos=north west,
					   grid=both, grid style={line width=.1pt, draw=gray!10}, major grid style={line width=.2pt,draw=gray!50}]

			% source_iteration plots.
			\addplot+[mark=o, thick, solid, blue, mark options={blue, solid}] table [x=krylov_dim, y expr=\thisrowno{4}/\thisrowno{3}, col sep=comma] {deal_II_results/monoenergetic_benchmark/source_iteration/c7/a1_l1/monoenergetic_apost_si_2D_errors_p2_n4};

			% modified_source_iteration plots.
			\addplot+[mark=square, thick, solid, red, mark options={red, solid}] table [x=krylov_dim, y expr=\thisrowno{4}/\thisrowno{3}, col sep=comma] {deal_II_results/monoenergetic_benchmark/modified_source_iteration/c7/a1_l1/monoenergetic_apost_msi_2D_errors_p2_n4};

			% GMRES plots.
			\addplot+[mark=diamond, thick, solid, black, mark options={black, solid}] table [x=krylov_dim, y expr=\thisrowno{4}/\thisrowno{3}, col sep=comma] {deal_II_results/monoenergetic_benchmark/gmres/c7/a1_l1/monoenergetic_apost_gmres_2D_errors_p2_n4};
			
			% Place coordinate for next figure.			
            \coordinate (bot) at (rel axis cs:1,0); % coordinate at bottom of the last plot
	\end{groupplot}

    % legend
	\path (top|-current bounding box.north)--
    	coordinate(legendpos)
    	(bot|-current bounding box.north);
	\matrix[
    	matrix of nodes,
    	anchor=south,
    	draw,
    	inner sep=0.2em,
    	draw
  	]at([yshift=1ex]legendpos)
  	{
    	\ref{plots:monoenergetic_benchmark_disc_opt_thick_dep_si_eff}& SI &[5pt]
    	\ref{plots:monoenergetic_benchmark_disc_opt_thick_dep_msi_eff}& Generalised SI &[5pt]
    	\ref{plots:monoenergetic_benchmark_disc_opt_thick_dep_gmres_eff}& GMRES \\
    }; 
\end{tikzpicture}
\caption{Effectivities of the \emph{a posteriori} solver error estimates employed in source iteration, generalised source iteration and right-preconditioned GMRES as shown in Figure \ref{fig:mono_benchmark_optical_thickness_dep_errs}. Left column: $L=\frac{1}{10}$. Middle column: $L=1$. Right column: $L=10$. Top row: $\sigma=\frac{1}{10}$. Middle row: $\sigma=1$. Bottom row: $\sigma=10$.}
\label{fig:mono_benchmark_optical_thickness_dep_effs}
\end{figure}

Figure \ref{fig:mono_benchmark_optical_thickness_dep_errs} shows the convergence of all three solvers as functions of the domain size $L$ and the magnitude of the macroscopic total cross-section $\sigma$ for a problem with fixed scattering ratio $c=\frac{7}{10}$ and fixed number of spatial elements $|\spacefespace|=256$. Introducing the quantity $\sigma L$, Figure \ref{fig:mono_benchmark_optical_thickness_dep_errs} shows that all three solvers converge more rapidly for small values of $\sigma L$ and approach their theoretical convergence rates as $\sigma L\rightarrow\infty$ (given in Theorems \ref{thm:dgfem_si_contraction_apost} and \ref{thm:dgfem_general_si_contraction_apost} for standard and generalised source iteration, respectively).

The behaviour of the \emph{a posteriori} solver error estimates for each solver can similarly be characterised in terms of $\sigma L$. Figure \ref{fig:mono_benchmark_optical_thickness_dep_effs}, shows that the effectivity indices for each solver deteriorate in the limit $\sigma L\rightarrow 0$, and for source iteration and GMRES approach $1$ in the limit $\sigma L\rightarrow \infty$. The behaviour of the effectivity index for the generalised source iteration is more subtle - for small enough values of $\sigma L$, the value of the effectivity index closely mirrors that of standard source iteration, but appears to deteriorate (relative to standard source iteration) as $\sigma L\rightarrow\infty$.

\subsection{Poly-energetic benchmark}

We now consider the application of poly-energetic source iteration and right-preconditioned GMRES for the numerical solution of the DGFEM \reff{eqn:dgfem_lbte_full} for a benchmark problem posed in two spatial dimensions, one angular dimension and one energetic dimension. The spatial domain is taken to be $(0,20)^2$ (in units of cm) and the energetic domain is taken to be $(10,1000)$ (in units of keV). The data terms $\alpha$ and $\theta$ are chosen to mimic the Compton scattering of photons through water. This is achieved by setting $\alpha=0$ and
\begin{equation*}
    \theta(\mathbf{x},\bm{\mu}'\cdot\bm{\mu},E'\rightarrow E) = \rho(\mathbf{x}) \theta_{KN}(E',E,\bm{\mu}'\cdot\bm{\mu}) \delta(F(E',E,\bm{\mu}'\cdot\bm{\mu})),
\end{equation*}

\noindent where $\rho(\mathbf{x})\approx3.34281\times10^{29}$e/m$^3$ denotes the electron density of water and $\theta_{KN}$ denotes the Klein-Nishina differential scattering cross-section per electron \cite{Klein1929} given by
\begin{equation*}
    \theta_{KN}(E',E,\cos\phi) = \frac12 r_e^2 \left( \frac{E}{E'} \right)^2 \left( \frac{E}{E'} + \frac{E'}{E} - \sin^2\phi \right)
\end{equation*}

\noindent with $r_e\approx2.81794$m$^{-15}$ denoting the classical electron radius. Additionally, $\delta$ denotes the Dirac delta distribution and is used to enforce the kinematic constraint $F(E',E,\bm{\mu}'\cdot\bm{\mu}) = 0$ representing conservation of total energy and momentum during a Compton scattering event. Here, $F$ is defined by
\begin{equation*}
    F(E',E,\cos\phi) = E - \frac{E'}{1+\frac{E'}{m_ec^2}(1-\cos\phi)},
\end{equation*}

\noindent where $m_ec^2\approx511$keV denotes the electron rest energy. The macroscopic scattering cross-section $\beta$ is numerically evaluated via \reff{eqn:macro_scatter_cs}. The forcing data $f$ and the boundary data $g_D$ are chosen such that the analytical solution to \reff{eqn:lbte_exact} is given by
\begin{equation*}
    u(\mathbf{x},\bm{\mu},E) = \textnormal{exp} \left( -k\left( \frac{E}{E_{max}} \mathbf{x}\cdot\bm{\mu} \right)^2 \right) \psi\left( \frac{E}{E_{max}} \right),
\end{equation*}

\noindent where $k=0.16$cm$^{-2}$, $E_{max}=1000$keV and $\psi(x)=e^{-\frac{1}{1-x^2}}$ is a mollifier.

We investigate the behaviour of the proposed poly-energetic source iteration and right-preconditioned GMRES solvers for the DGFEM \reff{eqn:dgfem_lbte_full}. A space-angle-energy mesh $\mesh=\spacemesh\times\anglemesh\times\energymesh$ is employed, consisting of $|\spacemesh|=256$ tensor-product spatial elements, $|\anglemesh|=64$ curved angular elements (constructed as in Section \ref{section:discretisation}) and $|\energymesh|=16$ energy groups. Each space-angle-energy element $\kappa=\spacekappa\times\anglekappa\times\kappa_g\in\mesh$ is endowed with a finite element space $\mathbb{P}_p(\spacekappa)\times\mathbb{Q}_q(\anglekappa)\times\mathbb{P}_r(\kappa_g)$ employing a uniform polynomial degrees $p=q=r=2$. The total number of degrees of freedom of the finite element space $\fespace=\spacefespace\otimes\anglefespace\otimes\energyfespace$ is given by $|\fespace|=14,155,776$.

Since Compton scattering only permits energy losses in scattered photons, the DGFEM equations \reff{eqn:dgfem_lbte_full} are solved in the direction of increasing energy groups (and thus in the direction of decreasing energies). With this in mind, the \emph{a posteriori} error estimates of Theorems \ref{thm:dgfem_si_contraction_apost} and \ref{thm:general_solver_apost_est} are applied for each energy group rather than for the full energy domain. Thereby, the error estimates presented in this section are no longer guaranteed upper bounds for the DGFEM-energy norm solver error within each energy group since an additional error is incurred due to the use of inexact group fluxes from higher energy groups as source terms for the lower group problems.

\subsubsection{Fixed number of inner iterations per group}

We first investigate the effect of employing a fixed number of inner iterations per energy group. For $1\le n\le 12$, the group solver is terminated either when the number of inner iterations reaches $n$ or if the \emph{a posteriori} solver error estimate at a prior iteration reaches $10^{-10}$.

\begin{figure}[t!]
\centering
\resizebox{0.6\textwidth}{!}{
	\begin{tikzpicture}
	\begin{axis}[xmode=linear, ymode=log,
				 xlabel=Iteration, ylabel=Error,
				 width=\textwidth,
				 axis background/.style={fill=gray!0}, legend pos=south west,
				 grid=both,
				 grid style={line width=.1pt, draw=gray!10},
				 major grid style={line width=.2pt,draw=gray!50},
				 legend pos=south west]
				 
		% source_iteration plots.
		\addplot+[mark=square, thick, dashed, blue, mark options={blue, solid}] table [x=krylov_dim, y=dg_error_solver, col sep=comma] {deal_II_results/polyenergetic_benchmark/source_iteration/polyenergetic_apost_si_2D_errors};
		\addlegendentry{SI solver error};
		\addplot+[mark=square*, thick, solid, blue, mark options={blue, solid}] table [x=krylov_dim, y=solver_apost_estimate, col sep=comma] {deal_II_results/polyenergetic_benchmark/source_iteration/polyenergetic_apost_si_2D_errors};
		\addlegendentry{SI \emph{a post.} estimate};
		
		% GMRES plots.
		\addplot+[mark=o, thick, dashed, red, mark options={red, solid}] table [x=krylov_dim, y=dg_error_solver, col sep=comma] {deal_II_results/polyenergetic_benchmark/gmres/polyenergetic_apost_gmres_2D_errors};
		\addlegendentry{GMRES solver error};
		\addplot+[mark=*, thick, solid, red, mark options={red, solid}] table [x=krylov_dim, y=solver_apost_estimate, col sep=comma] {deal_II_results/polyenergetic_benchmark/gmres/polyenergetic_apost_gmres_2D_errors};
		\addlegendentry{GMRES \emph{a post.} estimate};
	\end{axis}
	\end{tikzpicture}
}
\caption{DGFEM-energy norm errors and \emph{a posteriori} solver error estimates against number of inner iterations per group.}
\label{fig:poly_benchmark_total_errs_fixed_its}
\end{figure}
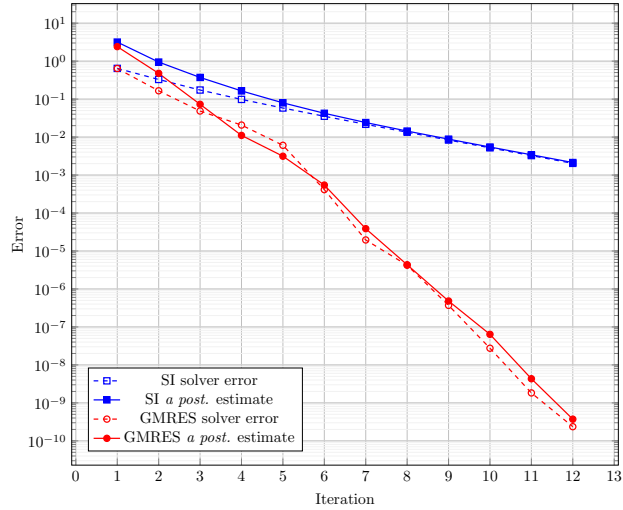

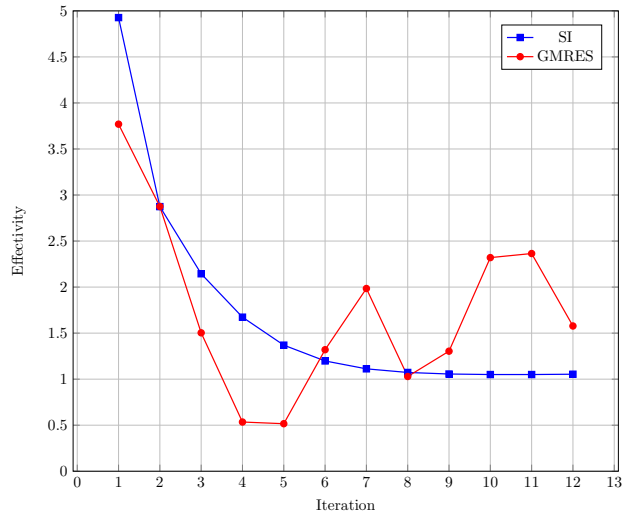
\begin{figure}[t!]
\centering
\resizebox{0.6\textwidth}{!}{
	\begin{tikzpicture}
	\begin{axis}[xmode=linear, ymode=linear,
				 xlabel=Iteration, ylabel=Effectivity,
				 ymin=0, ymax=5,
				 width=\textwidth,
				 axis background/.style={fill=gray!0}, legend pos=south west,
				 grid=both,
				 grid style={line width=.1pt, draw=gray!10},
				 major grid style={line width=.2pt,draw=gray!50},
				 legend pos=north east]
				 
		% source_iteration plot.
		\addplot+[mark=square*, thick, solid, blue, mark options={blue, solid}] table [x=krylov_dim, y expr=\thisrowno{4}/\thisrowno{3}, col sep=comma] {deal_II_results/polyenergetic_benchmark/source_iteration/polyenergetic_apost_si_2D_errors};
		\addlegendentry{SI};
		
		% GMRES plot.
		\addplot+[mark=*, thick, solid, red, mark options={red, solid}] table [x=krylov_dim, y expr=\thisrowno{4}/\thisrowno{3}, col sep=comma] {deal_II_results/polyenergetic_benchmark/gmres/polyenergetic_apost_gmres_2D_errors};
		\addlegendentry{GMRES };
	\end{axis}
	\end{tikzpicture}
}
\caption{Effectivity of \emph{a posteriori} solver error estimates against number of inner iterations per group.}
\label{fig:poly_benchmark_total_effs_fixed_its}
\end{figure}

Figure \ref{fig:poly_benchmark_total_errs_fixed_its} shows the performance of poly-energetic source iteration and right-preconditioned GMRES solver. While the right-preconditioned GMRES method clearly exhibits a faster convergence rate to the true solution as the number of inner iterations increases, the corresponding \emph{a posteriori} solver error estimate occasionally fails to bound the true DGFEM-energy norm error from above.

For both methods, the behaviours of both the DGFEM-energy norm solver error and the \emph{a posteriori} error estimate as the number of inner iterations increases appears to be most heavily influenced by the lowest energy groups. While not presented here, it is observed that both the solver errors and error estimates for higher energy groups decay very quickly with respect to the number of inner iterations.

Figure \ref{fig:poly_benchmark_total_effs_fixed_its} shows the effectivities of both methods. Despite not being a guaranteed upper bound, the effectivity of the \emph{a posteriori} error estimate used in the poly-energetic source iteration method converges to one from above, while that of GMRES varies between each iteration, occasionally dropping below one when a small number of iterations are employed.

\subsubsection{Fixed solver tolerance per group}

Since the \emph{a posteriori} solver error estimates in Theorems \ref{thm:dgfem_si_contraction_apost} and \ref{thm:general_solver_apost_est} can be applied on a per-energy-group basis, we also consider the effect of employing a fixed solver tolerance for each energy group. For a given global solver tolerance $\epsilon$, we terminate each group solver when the \emph{a posteriori} solver error estimate for the current group drops below $\frac{\epsilon}{N_\energydom}$ or if the number of inner iterations reaches 50.

\begin{figure}[t!]
\centering
\resizebox{0.6\textwidth}{!}{
	\begin{tikzpicture}
	\begin{axis}[xmode=log, ymode=log,
				 xlabel=Solver tolerance, ylabel=Error,
				 width=\textwidth,
				 axis background/.style={fill=gray!0}, legend pos=south west,
				 grid=both,
				 grid style={line width=.1pt, draw=gray!10},
				 major grid style={line width=.2pt,draw=gray!50},
				 legend pos=south east]
				 
		% source_iteration plots.
		\addplot+[mark=square, thick, dashed, blue, mark options={blue, solid}] table [x=solver_tolerance, y=dg_error_solver, col sep=comma] {deal_II_results/polyenergetic_benchmark/source_iteration/polyenergetic_apost_tolerance_si_2D_errors};
		\addlegendentry{SI solver error};
		\addplot+[mark=square*, thick, solid, blue, mark options={blue, solid}] table [x=solver_tolerance, y=solver_apost_estimate, col sep=comma] {deal_II_results/polyenergetic_benchmark/source_iteration/polyenergetic_apost_tolerance_si_2D_errors};
		\addlegendentry{SI \emph{a post.} estimate};
		
		% GMRES plots.
		\addplot+[mark=o, thick, dashed, red, mark options={red, solid}] table [x=solver_tolerance, y=dg_error_solver, col sep=comma] {deal_II_results/polyenergetic_benchmark/gmres/polyenergetic_apost_tolerance_rpgmres_2D_errors};
		\addlegendentry{GMRES solver error};
		\addplot+[mark=*, thick, solid, red, mark options={red, solid}] table [x=solver_tolerance, y=solver_apost_estimate, col sep=comma] {deal_II_results/polyenergetic_benchmark/gmres/polyenergetic_apost_tolerance_rpgmres_2D_errors};
		\addlegendentry{GMRES \emph{a post.} estimate};
		
		% Reference plot - the required tolerance for an individual group (assuming that solver error is to be equidistributed amongst all groups).
		\addplot+[mark=none, thick, solid, black] ({1e-10},{1e-10}) -- ({1},{1});
		\addlegendentry{Global stopping tolerance};
	\end{axis}
	\end{tikzpicture}
}
\caption{DGFEM-energy norm errors and \emph{a posteriori} solver error estimates against number of inner iterations per group.}
\label{fig:poly_benchmark_total_errs_fixed_tol}
\end{figure}
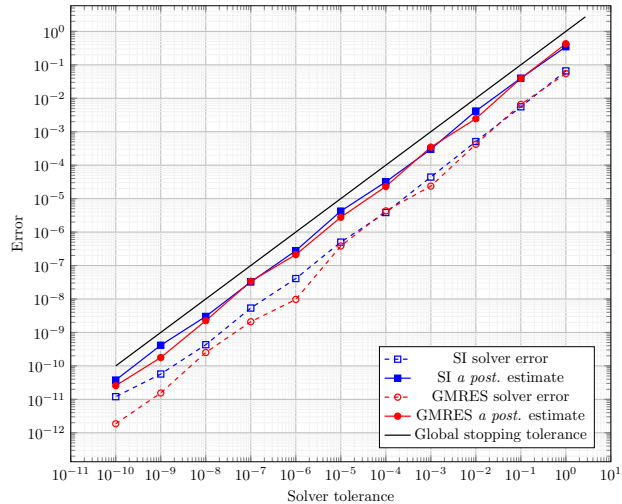

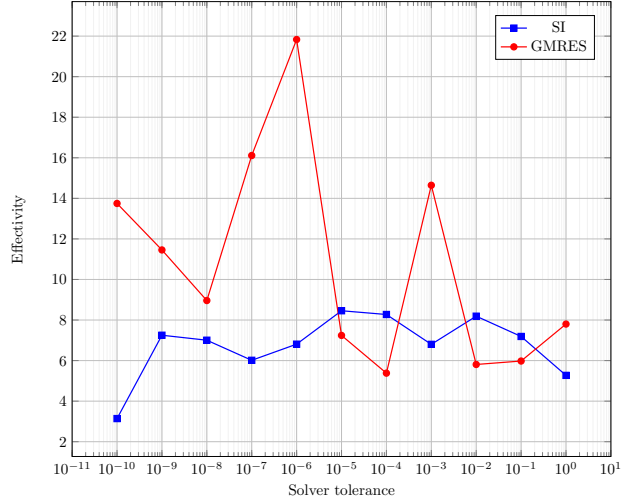
\begin{figure}[t!]
\centering
\resizebox{0.6\textwidth}{!}{
	\begin{tikzpicture}
	\begin{axis}[xmode=log, ymode=linear,
				 xlabel=Solver tolerance, ylabel=Effectivity,
				 width=\textwidth,
				 axis background/.style={fill=gray!0}, legend pos=south west,
				 grid=both,
				 grid style={line width=.1pt, draw=gray!10},
				 major grid style={line width=.2pt,draw=gray!50},
				 legend pos=north east]
				 
		% source_iteration plot.
		\addplot+[mark=square*, thick, solid, blue, mark options={blue, solid}] table [x=solver_tolerance, y expr=\thisrowno{4}/\thisrowno{3}, col sep=comma] {deal_II_results/polyenergetic_benchmark/source_iteration/polyenergetic_apost_tolerance_si_2D_errors};
		\addlegendentry{SI};
		
		% GMRES plot.
		\addplot+[mark=*, thick, solid, red, mark options={red, solid}] table [x=solver_tolerance, y expr=\thisrowno{4}/\thisrowno{3}, col sep=comma] {deal_II_results/polyenergetic_benchmark/gmres/polyenergetic_apost_tolerance_rpgmres_2D_errors};
		\addlegendentry{GMRES };
	\end{axis}
	\end{tikzpicture}
}
\caption{Effectivity of \emph{a posteriori} solver error estimates against solver tolerance per group.}
\label{fig:poly_benchmark_total_effs_fixed_tol}
\end{figure}

Figure \ref{fig:poly_benchmark_total_errs_fixed_tol} shows the performance of poly-energetic source iteration and right-preconditioned GMRES, as well as the group-wise stopping tolerance, for this solution strategy. Here, both methods are able to resolve the solution to the lowest tolerance tested. This is not necessarily true in general - a method may appear to stagnate at a solver error higher than the target solver tolerance if an insufficient number of iterations is employed on any energy group. This is particularly true in the case of poly-energetic source iteration used to resolve the solution at low energy groups; see Table \ref{tab:poly_source_iteration_group_its}.

Figure \ref{fig:poly_benchmark_total_effs_fixed_tol} shows the effectivities of both methods using the fixed inner iteration strategy. While the effectivities of both methods remain above 1, we again stress that the computed \emph{a posteriori} error estimates are not guaranteed upper bounds for the DGFEM-energy norm solver error.

\begin{table}[h!]
\begin{tabular}{c|ccccccccccc}
\multirow{2}{*}{Group}  & \multicolumn{11}{c}{Global solver tolerance} \\ & $10^{0}$ & $10^{-1}$ & $10^{-2}$ & $10^{-3}$ & $10^{-4}$ & $10^{-5}$ & $10^{-6}$ & $10^{-7}$ & $10^{-8}$ & $10^{-9}$ & $10^{-10}$ \\ \hline
1  & 2 & 2   & 3    & 4     & 5      & 5     & 6     & 6     & 7     & 7     & 8     \\
2  & 2 & 2   & 3    & 4     & 5      & 5     & 6     & 6     & 7     & 7     & 8     \\
3  & 2 & 2   & 3    & 4     & 5      & 5     & 6     & 6     & 7     & 8     & 8     \\
4  & 2 & 3   & 3    & 4     & 5      & 5     & 6     & 7     & 7     & 8     & 8     \\
5  & 2 & 3   & 3    & 4     & 5      & 5     & 6     & 7     & 7     & 8     & 8     \\
6  & 2 & 3   & 3    & 4     & 5      & 6     & 6     & 7     & 7     & 8     & 9     \\
7  & 2 & 3   & 4    & 4     & 5      & 6     & 6     & 7     & 8     & 8     & 9     \\
8  & 2 & 3   & 4    & 5     & 5      & 6     & 7     & 7     & 8     & 8     & 9     \\
9  & 2 & 3   & 4    & 5     & 5      & 6     & 7     & 8     & 8     & 9     & 9     \\
10 & 2 & 3   & 4    & 5     & 6      & 6     & 7     & 8     & 9     & 9     & 10    \\
11 & 2 & 3   & 4    & 5     & 6      & 7     & 8     & 8     & 9     & 10    & 10    \\
12 & 2 & 4   & 5    & 6     & 6      & 7     & 8     & 9     & 10    & 10    & 11    \\
13 & 3 & 4   & 5    & 6     & 7      & 8     & 9     & 10    & 10    & 11    & 12    \\
14 & 3 & 5   & 6    & 7     & 8      & 9     & 10    & 11    & 12    & 13    & 14    \\
15 & 4 & 6   & 8    & 9     & 10     & 11    & 12    & 14    & 16    & 17    & 19    \\
16 & 5 & 10  & 15   & 20    & 25     & 29    & 34    & 38    & 43    & 47    & 50    
\end{tabular}
\caption{Table of iteration counts required per energy group for the poly-energetic source iteration \emph{a posteriori} error estimate to drop below the specified group solver tolerance (defined to be the global solver tolerance divided by the number of groups).}
\label{tab:poly_source_iteration_group_its}
\end{table}

\begin{table}[h!]
\begin{tabular}{c|ccccccccccc}
\multirow{2}{*}{Group}  & \multicolumn{11}{c}{Global solver tolerance} \\ & $10^0$ & $10^{-1}$ & $10^{-2}$ & $10^{-3}$ & $10^{-4}$ & $10^{-5}$ & $10^{-6}$ & $10^{-7}$ & $10^{-8}$ & $10^{-9}$ & $10^{-10}$ \\ \hline
1  & 2 & 3   & 4    & 4     & 5      & 5     & 6     & 6     & 7     & 7     & 8     \\
2  & 2 & 3   & 4    & 4     & 5      & 5     & 6     & 6     & 7     & 7     & 8     \\
3  & 2 & 3   & 4    & 4     & 5      & 5     & 6     & 6     & 7     & 7     & 8     \\
4  & 2 & 3   & 4    & 4     & 5      & 5     & 6     & 6     & 7     & 7     & 8     \\
5  & 2 & 3   & 4    & 4     & 5      & 6     & 6     & 6     & 7     & 7     & 8     \\
6  & 2 & 3   & 4    & 5     & 5      & 6     & 6     & 7     & 7     & 8     & 8     \\
7  & 3 & 3   & 4    & 5     & 5      & 6     & 6     & 7     & 7     & 8     & 8     \\
8  & 3 & 3   & 4    & 5     & 5      & 6     & 6     & 7     & 7     & 8     & 8     \\
9  & 3 & 4   & 4    & 5     & 5      & 6     & 6     & 7     & 7     & 8     & 8     \\
10 & 3 & 4   & 4    & 5     & 6      & 6     & 7     & 7     & 8     & 8     & 9     \\
11 & 3 & 4   & 5    & 5     & 6      & 6     & 7     & 7     & 8     & 9     & 9     \\
12 & 3 & 4   & 5    & 5     & 6      & 7     & 7     & 8     & 8     & 9     & 10    \\
13 & 3 & 4   & 5    & 6     & 6      & 7     & 8     & 8     & 9     & 10    & 10    \\
14 & 4 & 4   & 5    & 6     & 7      & 8     & 8     & 9     & 10    & 11    & 11    \\
15 & 4 & 5   & 6    & 7     & 8      & 9     & 10    & 11    & 12    & 13    & 14    \\
16 & 4 & 6   & 7    & 8     & 9      & 10    & 12    & 12    & 13    & 14    & 15
\end{tabular}
\caption{Table of iteration counts required per energy group for the poly-energetic GMRES \emph{a posteriori} error estimate to drop below the specified group solver tolerance (defined to be the global solver tolerance divided by the number of groups).}
\label{tab:poly_gmres_group_its}
\end{table}

It is also useful to compare the number of iterations per energy group taken by poly-energetic source iteration and GMRES to reach the desired group solver tolerance. Tables \ref{tab:poly_source_iteration_group_its} and \ref{tab:poly_gmres_group_its} show the number of iterations per energy group taken by both solvers to reduce the corresponding group's \emph{a posteriori} solver error estimate by $\frac{\varepsilon}{N_\energydom}=\frac{\varepsilon}{16}$, where $\varepsilon$ denotes the global solver tolerance. For high energy groups (indicated by low values of the group index), both methods require few iterations to resolve the solution to a very low tolerance, while low energy groups (indicated by high values of the group index) require many more iterations to achieve the same tolerance. The number of iterations taken by poly-energetic source iteration and GMRES to resolve the solution on a single group to within that group's solver tolerance are roughly comparable for most energy groups, but on the lowest energy group the poly-energetic source iteration requires many more iterations than GMRES to achieve the same group solver tolerance.

Finally, we comment on the storage requirements of poly-energetic source iteration and GMRES. While both methods return a solution vector of length $\dim\fespace=\dim\spacefespace\times\dim\anglefespace\times\dim\energyfespace$, these vectors may be assembled incrementally starting from the highest energy group. Thus, for each energy group $1\le g\le N_\energydom$, both methods return a vector of length $M=\dim\spacefespace\times\dim\anglefespace\times\dim\mathbb{P}_p(\kappa_g)$ corresponding to a block of the full solution vector. The number of vectors of length $M$ required to perform the implementation of poly-energetic source iteration as outlined in Section \ref{section:implementation} remains constant as the number of iterations increases, whereas the storage requirements for GMRES grow linearly with the dimension of the Krylov subspace generated at each iteration. While $M$ may be moderately large for practical applications, Table \ref{tab:poly_gmres_group_its} suggests that, even for the lowest energy group, only a small number of Krylov vectors of length $M$ may be required to achieve a modest global solver tolerance.

\section{Conclusions} \label{section:conclusions}

In this article we have investigated the application of iterative linear solvers for the system of equations arising from the DGFEM discretisation of the linear Boltzmann transport equation. In particular, we presented a generalisation of source iteration and modified source iteration for the poly-energetic and, respectively, mono-energetic settings and showed that the resulting iterative solvers are robustly convergent with respect to the discretisation parameters employed in the DGFEM scheme. Furthermore, we introduced a GMRES solver based on employing source iteration as a preconditioner. For all three approaches computable {\em a posteriori} error bounds on the solver error were deduced. In the case of GMRES, this bound was established through a judicious choice of the left- and right-preconditioners for the original system.

The practical performance of the proposed iterative solvers was studied through a series of numerical experiments. Here we observed that GMRES outperformed source iteration (and modified source iteration for mono-energetic problems) when considering the number of iterations required to attain convergence. This was particularly evident for poly-energetic problems at lower energy groups. In terms of computational costs, we recall that one step of source iteration is comparable to one iteration of GMRES and hence the latter approach is highly advantageous in these lower energy regimes.

As part of our future work, we plan to consider the exploitation of DSA approaches, though their extension to poly-energetic problems requires additional investigation, as well as the use of alternative preconditioning strategies, cf., for example, \cite{patton2002application,badri2019preconditioned,warsa2004krylov}. Furthermore, the linear solver {\em a posteriori} error bounds developed here will be implemented within an adaptive mesh refinement framework whereby both solver and discretisation error are simultaneously controlled.

\section*{Acknowledgements}
PH and MEH acknowledge the financial support of the EPSRC (grant EP/R030707/1).

\bibliographystyle{ieeetr}
\bibliography{references}

\end{document}